\documentclass[12pt,a4paper]{amsart}
\usepackage{amssymb,stmaryrd,blkarray}
\usepackage[a4paper,left=2.25cm,right=2.25cm,top=3cm,bottom=3cm,headsep=1cm]{geometry}
\usepackage{tikz}\usetikzlibrary{graphs,quotes,fit,positioning,matrix,calc,decorations.markings,angles,decorations.pathmorphing,decorations.pathreplacing}%,matrix,fit,shapes.geometric}
\newcommand\tikzif[2][]{
\tikzifinpicture{#2}{\begin{tikzpicture}[#1]#2\end{tikzpicture}}% don't process options if not
}
\tikzset{math mode/.style = {execute at begin node=$, execute at end node=$}}%makes all nodes math mode
\tikzset{arrow/.style={postaction={decorate,thick,decoration={markings,mark = at position #1 with {\arrow{>}}}}},arrow/.default=0.5}
\tikzset{invarrow/.style={postaction={decorate,thick,decoration={markings,mark = at position #1 with {\arrow{<}}}}},invarrow/.default=0.5}
\usepackage{hyperref}
\allowdisplaybreaks[1]
\renewcommand\ss{\scriptstyle}

\renewcommand\P{{\mathcal P}}
\def\a_#1{\varphi_{#1}\hspace{-1pt}(t)}
\newcommand\fug{\mathrm{fug}}
\newcommand\wt{\mathrm{wt}}
\newcommand\ZZ{{\mathbb Z}}

\newcommand\CC{{\mathbb C}}

\newcommand\RR{{\mathbb R}}
\theoremstyle{plain}
\newtheorem{thm}{Theorem}
\newtheorem{prop}{Proposition}
\newtheorem{lem}{Lemma}
\theoremstyle{remark}
\newtheorem*{rmk*}{Remark}
\newtheorem*{ex*}{Example}
\title{Honeycombs for Hall polynomials}
\author{Paul Zinn-Justin}
\address{Paul Zinn-Justin, School of Mathematics and Statistics, The University of Melbourne, 
Victoria 3010, Australia}
\email{pzinn@unimelb.edu.au}
\thanks{PZJ was supported by ARC grant FT150100232. He would like to thank A.~Knutson
for numerous discussions as part of a parallel collaboration.}
\date{\today}

\newcommand\rem[2][]{}
\long\def\junk#1{}
\definecolor{myred}{rgb}{0.6,0,0.4}
\definecolor{mygreen}{rgb}{0,0.4,0.5}
\definecolor{myblue}{rgb}{0,0,0}
\def\puzzlescale{0.7}
\tikzset{puznode/.style={font=\scriptsize}}
\tikzset{puz/.style={yscale=-1.732,rotate=45,scale=\puzzlescale,line cap=round}}
\tikzset{bgline/.style={gray,dotted}}%style for boundary of tiles
\tikzset{lattice/.style={fill=gray,circle,inner sep=1pt}}%style for vertices of lattice
\tikzset{honeyline/.style={black,thin,opaque}}
\def\rawhoneycomb#1{
\newcount\i\i=0
\newcount\j\j=0
\newcount\k\k=0
\foreach\a/\aa/\b/\bb/\c/\cc in {#1} {
\node[lattice] at (\i+1/3,\j+1/3) {};
\global\advance\k by 1 % k is now i+j+1
  \if!\ifnum9<1\a!\else_\fi\edef\at{\a}\else\edef\at{1}\fi
  \if!\ifnum9<1\b!\else_\fi\edef\bt{\b}\else\edef\bt{1}\fi
  \if!\ifnum9<1\c!\else_\fi\edef\ct{\c}\else\edef\ct{1}\fi
  \if!\ifnum9<1\aa!\else_\fi\edef\aat{\aa}\else\edef\aat{1}\fi
  \if!\ifnum9<1\bb!\else_\fi\edef\bbt{\bb}\else\edef\bbt{1}\fi
  \if!\ifnum9<1\cc!\else_\fi\edef\cct{\cc}\else\edef\cct{1}\fi
  \ifnum\at>0
  \foreach\r in {1,...,\at}
  \draw[honeyline,xshift=\r*0.04cm-\at*0.02cm-0.02cm,yshift=-\r*0.02cm+\at*0.01cm+0.01cm] (\i+1/3,\j+1/3) -- (\i+1/3,\j-1/3);
  \node[puznode,myred,\ifnum\j>0 \else opaque\fi] at (\i+1/3,\j) {\a};
  \fi
  \ifnum\bt>0
  \foreach\r in {1,...,\bt}
  \draw[honeyline,yshift=\r*0.02cm-\bt*0.01cm-0.01cm,xshift=\r*0.02cm-\bt*0.01cm-0.01cm] (\i+1/3,\j+1/3) -- (\i-1/3,\j+1);
  \node[puznode,myred,\ifnum\i>0 \else opaque\fi] at (\i,\j+2/3) {\b};
  \fi
  \ifnum\ct>0
  \foreach\r in {1,...,\ct}
  \draw[honeyline,yshift=\r*0.04cm-\ct*0.02cm-0.02cm,xshift=-\r*0.02cm+\ct*0.01cm+0.01cm] (\i+1/3,\j+1/3) -- (\i+1,\j+1/3);
  \node[puznode,myred,\ifnum\size>\k \else opaque\fi] at (\i+2/3,\j+1/3) {\c};
  \fi
  \ifnum\aat>0
  \foreach\r in {1,...,\aat}
  \draw[honeyline,xshift=\r*0.02cm-\aat*0.01cm-0.01cm,yshift=\r*0.02cm-\aat*0.01cm-0.01cm] (\i+1/3,\j+1/3) -- (\i+2/3,\j);
  \node[puznode,mygreen] at (\i+2/3,\j) {\aa};
  \fi
  \ifnum\bbt>0
  \foreach\r in {1,...,\bbt}
  \draw[honeyline,yshift=\r*0.04cm-\bbt*0.02cm-0.02cm,xshift=-\r*0.02cm+\bbt*0.01cm+0.01cm] (\i+1/3,\j+1/3) -- (\i,\j+1/3);
  \node[puznode,mygreen] at (\i,\j+1/3) {\bb};
  \fi
  \ifnum\cct>0
  \foreach\r in {1,...,\cct}
  \draw[honeyline,xshift=\r*0.04cm-\cct*0.02cm-0.02cm,yshift=-\r*0.02cm+\cct*0.01cm+0.01cm] (\i+1/3,\j+1/3) -- (\i+1/3,\j+2/3);
  \node[puznode,mygreen] at (\i+1/3,\j+2/3) {\cc};
  \fi
\global\advance\i by 1
\ifnum\k=\size\global\i=0\global\advance\j by 1\global\k=\j
\fi
}
}
\def\rawpuzzlex#1{
\newcount\i\i=0
\newcount\j\j=0
\newcount\k\k=0
\foreach\a/\aa/\b/\bb/\c/\cc in {#1} {
\node[puznode,myred] at (\i+1/3,\j) {\a};
\node[puznode,myred] at (\i,\j+2/3) {\b};
\node[puznode,myred] at (\i+2/3,\j+1/3) {\c};
\node[puznode,mygreen] at (\i+2/3,\j) {\aa};
\node[puznode,mygreen] at (\i,\j+1/3) {\bb};
\node[puznode,mygreen] at (\i+1/3,\j+2/3) {\cc};
\global\advance\i by 1
\global\advance\k by 1
\ifnum\k=\size\global\i=0\global\advance\j by 1\global\k=\j
\fi
}
}
\newcommand\honeycomb[2][]{%
\tikzif{
\begin{scope}[puz,math mode]
\begin{scope}[bgline]
\draw (0,0) -- (\size,0) -- (0,\size) -- cycle;
\foreach\x in {1,...,\size} \draw (\x,0) -- (\x,\size-\x);
\foreach\x in {1,...,\size} \draw (0,\x) -- (\size-\x,\x);
\foreach\x in {1,...,\size} \draw (\size-\x,0) -- (0,\size-\x);
\end{scope}
#1
\rawhoneycomb{#2}
\end{scope}
}
}
\newcommand\puzzlex[2][]{%
\tikzif{
\begin{scope}[puz,math mode]
\begin{scope}[bgline]
\draw (0,0) -- (\size,0) -- (0,\size) -- cycle;
\foreach\x in {1,...,\size} \draw (\x,0) -- (\x,\size-\x);
\foreach\x in {1,...,\size} \draw (0,\x) -- (\size-\x,\x);
\foreach\x in {1,...,\size} \draw (\size-\x,0) -- (0,\size-\x);
\end{scope}
#1
\rawpuzzlex{#2}
\end{scope}
}
}
\newdimen{\cellsize}
\newcommand\medboxes{\setlength{\cellsize}{14.22pt}\def\boxformat{}}%exactly 0.5cm

\medboxes
\tikzset{tableaubox/.style={draw=black,thin,sharp corners,solid,minimum size=\cellsize,inner sep=0pt}}
\tikzset{tableau/.style={matrix,name=tab,matrix anchor=tab-1-1.south west,inner sep=1pt,matrix of math nodes,cells={anchor=center,draw=black,thin,solid,arrows=-},nodes={tableaubox,execute at begin node=\boxformat},nodes in empty cells,row sep={\cellsize,between origins},column sep={\cellsize,between origins}}}
\newcommand\colorcell[1]{|[fill=#1]|}
\newcommand\missingcell{|[draw=none]|}
\makeatletter
\newcommand\cellextra[1]{#1\expandafter\tikz@lib@matrix@start@cell}%trick: after doing the extra stuff we restart the node reading macro. hopefully expandafter is ok
\makeatother
\newcommand\vcell{\cellextra{\draw (-0.5*\cellsize,-0.5*\cellsize) --++(0,\cellsize);}\missingcell}
\def\activate#1{\begingroup
  \lccode`\~=`#1%
  \lowercase{\endgroup \let~#1}%
  \catcode`#1=13\relax}
\activate &

\begin{document}
\maketitle

\begin{abstract}
We propose a new formulation of Hall polynomials in terms of honeycombs,
which were 
previously introduced in the context of the Littlewood--Richardson rule.
We prove a Pieri rule and associativity for our honeycomb formula,
thus showing equality with Hall polynomials. Our proofs are linear algebraic
in nature, extending nontrivially the corresponding bijective
results for ordinary Littlewood--Richardson coefficients \cite{KTW-octa}.
\end{abstract}

\tikzset{bgline/.style={transparent}}
\tikzset{puznode/.style={transparent}}
\section{Introduction}
\subsection{Hall--Littlewood functions and Hall polynomials}
{\em Hall--Littlewood symmetric functions}\/ form a classical basis of the ring of symmetric
functions \cite{Littlewood-poly}; they interpolate between Schur functions at $t=0$ and
monomial symmetric functions at $t=1$. We refer the reader to the book \cite{Macdonald} for details.
In this paper, we are interested in the {\em structure constants}\/
of the algebra of symmetric functions $\Lambda$ (with coefficients which are rational functions
of the formal parameter $t$) 
in the basis of Hall--Littlewood functions:
if $P^\lambda\in \Lambda$ is the Hall--Littlewood symmetric function associated to the partition $\lambda$,
then write
\[
P^\lambda P^\mu = \sum_\nu c^{\lambda,\mu}_\nu P^\nu
\]
where the sum is over all partitions. The $c^{\lambda,\mu}_\nu$ turn out to be polynomials in $t$ with integer
coefficients. They are called {\em Hall polynomials}\/ \cite{Hall-poly} and are interesting
in their own right: they
count short exact sequences of finite abelian $p$-groups (with $t=1/p$).
Hall noticed that they were structure constants of an algebra.
Even though we only study Hall polynomials and not Hall--Littlewood symmetric functions in the present paper,
we note that Hall--Littlewood symmetric functions do occur naturally in the context of solvable lattice models, 
which is the implicit framework of the current work, cf the related papers \cite{artic65,artic67}.

In this paper, we shall restrict to a finite-dimensional quotient $\Lambda_{k,n}$ of $\Lambda$ defined
in two steps. First we consider the usual map $\Lambda\to\Lambda_k$ from symmetric functions to
{\em symmetric polynomials}\/ in $k$ variables. Under it, $P^\lambda$ is mapped to zero unless $\lambda$
has no more than $k$ parts, in which case we obtain the {\em Hall--Littlewood polynomial}
\[
P^\lambda(x_1,\ldots,x_k)=
\sum_{w\in \mathcal{S}_k/\mathcal{S}_\lambda} w\left(x_1^{\lambda_1}\ldots x_k^{\lambda_k}
\prod_{i,j: \lambda_i>\lambda_j} \frac{x_i-t x_j}{x_i-x_j}\right)
\]
where $\mathcal{S}_\lambda$ is the subgroup of $\mathcal{S}_k$ that stabilizes $\lambda$.

In a second stage,  denoting by $\lambda'$ the conjugate partition of $\lambda$,
let $\P_{k,n}$ be the set of partitions $\lambda$ such that $\lambda_1\le n-1$, $\lambda'_1\le k$,
and consider the further quotient by the ideal generated by all $P^\lambda$ for which $\lambda\not\in \P_{k,n}$.
It is easy to see, as a consequence of the Pieri rule that will be discussed below,
that $\Lambda_{k,n}$ has as a basis the $P^\lambda$ for $\lambda\in \P_{k,n}$,
and that its structure constants $c^{\lambda,\mu}_\nu$ are the same as those of $\Lambda$,
with $\lambda,\mu,\nu$ restricted to $\P_{k,n}$.

In what follows we implicitly pad partitions $\lambda\in \P_{k,n}$ with zeroes in such a way that they have exactly $k$ parts.
Denote then by $m_r(\lambda)$, $r=0,\ldots,n-1$, the number of parts $r$ of $\lambda$. In particular,
$m_0(\lambda)=k-\lambda'_1$, and $\sum_{r=0}^{n-1} m_r(\lambda)=k$. The map $\lambda\mapsto m_r(\lambda)$ is a bijection between $\P_{k,n}$ and the set of
$n$-tuplets of nonnegative integers summing to $k$.

\subsection{Honeycombs}
\rem[gray]{careful that compared to [KTW \url{https://arxiv.org/pdf/math/0306274.pdf}], my Young diagrams are conjugate.
both seem to correspond to what I call dual honeycombs in my viewer
(plus conjugate Young diagram for [KTW] of course).
honeycombs would be like [KTW] except probably need to flip left right
or something. in any case HL does not have conjugation symmetry (presumably
we get a rule for q-whitaker if we conjugate).

strangely I seem to agree with [KT \url{https://arxiv.org/pdf/math/9807160.pdf}]

the ``dual numbering'' of the viewer is a bit confusing in the sense that it does give
the conjugate partition.
}
We view vertices of the triangular lattice as 
\begin{equation}\label{eq:defL2}
L_2=\{(a,b,c)\in\ZZ^3: a+b+c=0\}
\end{equation}
(which is naturally identified with the root lattice of $\mathfrak{sl}_3$).
Given $k\ge 1$, a $GL(k)$ {\em honeycomb}\/ \cite{KT-I} is
a subset of line segments with multiplicities sitting inside the triangular lattice, in such a way
that there are $k$ semi-infinite segments going off in each of the three directions
\tikz[baseline=-3pt,scale=0.7]
{
\draw [->] (0,0) -- ++(180:1) node[left] {$\ss(0,1,-1)$};
\draw [->] (0,0) -- node[right] {$\ss(-1,0,1)$} ++(300:1);
\draw [->] (0,0) -- node[right] {$\ss(1,-1,0)$} ++(60:1);
%\draw [->] (0,0) -- ++(30:0.577) node[right] {$e_{1}$};
%\draw [->] (0,0) -- ++(150:0.577) node[left] {$e_{2}$};
%\draw [->] (0,0) -- node[left] {$e_{3}$} ++(270:0.577);
},
while all other segments are finite; and such that the endpoints of segments form vertices
where a balance condition is satisfied, namely that at each vertex, 
if we denote by $j,i,j',i',j'',i''$ the multiplicities of line segments around that vertex, then
the following two equalities are satisfied:
\begin{equation}\label{eq:balance}
\tikz[baseline=-3pt,scale=0.5]{
  \draw (0:1) node[right=-1mm,myred] {$\ss j$} -- (180:1) node[left=-1mm,myred] {$\ss i'$};
  \draw (120:1) node[left=-1mm,myred] {$\ss j'$} -- (300:1) node[right=-1mm,myred] {$\ss i''$};
  \draw (240:1) node[left=-1mm,myred] {$\ss j''$} -- (60:1) node[right=-1mm,myred] {$\ss i$};
}
\qquad
i'-j=i''-j'=i-j''
\end{equation}
(if there is no line the corresponding multiplicity is zero).

Fixing $n\ge1$, we say that a $GL(k)$ honeycomb has boundaries $\lambda,\mu,\nu\in \P_{k,n}$ if all vertices
are contained inside an equilateral triangle of size $n-1$, in such a way
that if one numbers lattice vertices of the boundary of that triangle from $0$ to $n-1$
left to right for each side, then there are exactly $m_r(\lambda)$ (resp.\ $m_r(\mu)$,
$m_r(\nu)$) semi-infinite lines oriented 
\tikz[baseline=-3pt,scale=0.7]
{
\draw [->] (0,0) -- ++(180:1);% node[left] {$\ss(0,-1,1)$};
}
(resp.\ 
\tikz[baseline=-3pt,scale=0.7]
{
\draw [->] (0,0) -- ++(60:1);%node[right] {$\ss(-1,1,0)$} 
},
\tikz[baseline=-3pt,scale=0.7]
{
\draw [->] (0,0) -- ++(300:1);%node[right] {$\ss(1,0,-1)$} 
})
going through vertex $r$, $r=0,\ldots,n-1$ of the North-West (resp.\ North-East, South) side of the triangle.

\begin{ex*}
A honeycomb with $\lambda=(6,3,1,0)$, $\mu=(3,2,1,0)$, $\nu=(6,5,4,1)$, $k=4$, $n=7$:
\begin{center}
\def\size{7}\honeycomb[{
\node at (-2/3,1+1/3) {6};
\node at (-2/3,4+1/3) {3};
\node at (-2/3,6+1/3) {1};
\node at (-2/3,7+1/3) {0};
\node at (1/3,-2/3) {0};
\node at (1+1/3,-2/3) {1};
\node at (2+1/3,-2/3) {2};
\node at (3+1/3,-2/3) {3};
\node at (2+1/3,5+1/3) {1};
\node at (5+1/3,2+1/3) {4};
\node at (6+1/3,1+1/3) {5};
\node at (7+1/3,1/3) {6};
}]{1/0/1/0/1/0,1/0/0/1/1/1,1/0/0/1/1/1,1/0/0/1/1/1,0/0/0/1/1/0,0/0/0/1/1/0,0/0/0/1/1/0,0/0/0/0/0/0,1/0/0/0/0/1,1/0/1/0/1/0,1/0/0/1/1/1,0/0/0/1/1/0,0/0/0/1/1/0,0/0/0/0/0/0,1/1/1/0/0/1,0/0/0/0/0/0,1/0/1/0/1/0,0/0/0/1/1/0,0/1/1/0/0/0,1/0/0/0/0/1,0/1/1/0/0/0,0/0/0/0/0/0,0/0/0/0/0/0,1/1/1/0/0/1,0/0/0/0/0/0,0/1/1/0/0/0,1/0/1/0/1/0,0/1/1/0/0/0}
\end{center}
In general, lines can have multiplicities which will be drawn as lines stacked next to each other and
redundantly labelled in purple.
\end{ex*}

To each vertex of a honeycomb we associate a {\em fugacity}\/ given by the formula
\begin{align}\label{eq:fug}
\fug\Big(\tikz[baseline=-3pt,scale=0.5]{
  \draw (0:1) node[right=-1mm,myred] {$\ss j$} -- (180:1) node[left=-1mm,myred] {$\ss i'$};
  \draw (120:1) node[left=-1mm,myred] {$\ss j'$} -- (300:1) node[right=-1mm,myred] {$\ss i''$};
  \draw (240:1) node[left=-1mm,myred] {$\ss j''$} -- (60:1) node[right=-1mm,myred] {$\ss i$};
}\Big)
&=
  \sum_{r=0}^{\min(i,i')} (-1)^r t^{j'r+r(r+1)/2} \frac{\a_{i+j-r}}{\a_{i-r}\a_r \a_{i'-r}}
\\\notag
&=
\frac{\a_{i+j}}{\a_i\a_{i'}}\,
\,{}_2\phi_1\left({t^{-i},t^{-i'}\atop t^{-(i+j)}};t,t^{i''+1}\right) 
\end{align}
where $\a_i=\prod_{r=1}^{i} (1-t^r)$,
and 
$\,{}_2\phi_1\left({t^{-i},t^{-i'}\atop t^{-(i+j)}};t,t^{i''+1}\right)=
\sum_{n=0}^{\min(i,i')} \frac{(t^{-i};t)_n(t^{-i'})_n}{(t;t)_n(t^{-(i+j)};t)_n}t^{n(i''+1)}$
is a terminating basic hypergeometric series \cite{GR-basic}. Fugacities are actually polynomials in $t$,
see Appendix~\ref{app:fug} for its first few values.

The fugacity of a honeycomb is the product of fugacities of its vertices.

The main theorem of this paper is a new formulation of the product rule for Hall--Littlewood polynomials
in terms of honeycombs:
\begin{thm}\label{thm:main}
The structure constants $c^{\lambda,\mu}_\nu$ of $\Lambda_{k,n}$ are given by
\[
c^{\lambda,\mu}_\nu = \sum \big\{ \fug(H)\, :\, H\ \text{honeycomb with boundaries $\lambda$, $\mu$, $\nu$}) \big\}
\]
\end{thm}
Equivalently, one has, in $\Lambda_{k,n}$,
\[
P^\lambda P^\mu = \sum_{H:\ \substack{\text{NW boundary of }H=\lambda\\\text{NE boundary of }H=\mu}} \fug(H)\ P^{\text{S boundary of $H$}}
\]
where the sum is over all honeycombs $H$ whose vertices are contained inside an equilateral triangle of 
size $n-1$.

\begin{ex*}
Here are the three honeycombs with boundaries $\lambda=(3,2,1,1,0)$, $\mu=(3,1,0,0,0)$, $\nu=(4,3,2,1,1)$,
$k=5$, $n=5$: 
\begin{center}
\def\size{5}
\honeycomb[{
  \node[myred] at (1/3,2.5+1/3) {2};
}]{3/0/0/0/0/3,1/0/1/0/1/0,0/0/0/1/1/0,1/0/0/1/1/1,0/0/0/1/1/0,3/1/1/0/0/3,0/0/0/0/0/0,0/0/0/0/0/0,1/0/1/0/1/0,3/0/1/0/1/2,0/0/0/1/1/0,0/1/1/1/1/0,2/0/2/0/2/0,0/1/1/2/2/0,0/1/1/0/0/0}
\honeycomb[{
  \node[myred] at (1/3,1.5+1/3) {2};
  \node[myred] at (1/3,2.5+1/3) {2};
  \node[myred] at (1.5+1/3,1+1/3) {2};
\path (1+1/3,1+1/3) node[right,green!50!black] {\ss 1+t};
\path (2+1/3,1+1/3) node [left,green!50!black] {\ss 1-t};
}]{3/0/0/0/0/3,1/0/0/0/0/1,0/0/0/0/0/0,1/0/1/0/1/0,0/0/0/1/1/0,3/0/1/0/1/2,1/0/1/1/2/0,0/1/0/2/1/1,0/0/0/1/1/0,2/1/1/0/0/2,0/0/0/0/0/0,1/0/1/0/1/0,2/0/2/0/2/0,0/1/1/2/2/0,0/1/1/0/0/0}
\honeycomb[{
  \node[myred] at (1/3,1.5+1/3) {2};
\path (1+1/3,2+1/3) node[above=-1mm,green!50!black] {\ss 1+t-t^2};
\path (1+1/3,3+1/3) node [right,green!50!black] {\ss 1+t};
}]{3/0/0/0/0/3,1/0/0/0/0/1,0/0/0/0/0/0,1/0/1/0/1/0,0/0/0/1/1/0,3/0/1/0/1/2,1/0/0/1/1/1,0/1/1/1/1/0,0/0/0/1/1/0,2/0/1/0/1/1,1/1/1/1/1/1,0/0/0/1/1/0,1/1/2/0/1/0,1/0/1/1/2/0,0/1/1/0/0/0}
\end{center}
Nontrivial fugacities are marked in green next to the vertices. In total, we find
\[
c^{\lambda,\mu}_\nu=
1+(1+t)(1-t)+(1+t-t^2)(1+t)=3+2t-t^2-t^3
\]
\end{ex*}

Product rules for Hall--Littlewood functions, or equivalently expressions for Hall polynomials, 
already exist in the literature; most notably, in \cite[II.4]{Macdonald}, a formula for $c^{\lambda,\mu}_\nu$
is given as a sum over Littlewood--Richardson tableaux, where the coefficients are very similar to ours.
Since honeycombs and Littlewood--Richardson tableaux are in bijection, one can presumably relate our two rules.
However, we believe the honeycomb formulation displays better some of the underlying symmetries, as should be made
clear in what follows.

\begin{rmk*}One can check that the expression \eqref{eq:fug} evaluated at $t=0$ is equal to $1$, so that 
the fugacity of every honeycomb at $t=0$ is $1$ as well; in other words,
\[
c^{\lambda,\mu}_\nu (t=0)= \text{number of honeycombs with boundaries $\lambda$, $\mu$, $\nu$}
\]
In this way we recover one of the many formulations of the {\em Littlewood--Richardson rule}\/ for Schur
polynomials \cite{KT-I}.
\end{rmk*}

\subsection{Plan of proof}
The rest of the paper is devoted to the proof of Theorem~\ref{thm:main}. The logic that we follow 
is analogous to
the paper \cite{KTW-octa}, which is concerned with the Littlewood--Richardson rule for Schur polynomials,
the special case $t=0$ of our result (as remarked right above).\footnote{There are minor differences between our setup
and that of \cite{KTW-octa}: (a) We work in $\Lambda_{k,n}$ whereas they work in $\Lambda_k$; (b) relatedly,
we use honeycombs whereas they use hives (the bijection between the two is explained in \cite[Sect.~6]{KTW-octa};
(c) our partitions are {\em conjugate}\/ of theirs, which is irrelevant at $t=0$, but not so for general $t$.}

We identify partitions with Young diagrams.
It is known that Hall--Littlewood polynomials (or symmetric functions) satisfy the {\em Pieri rule}\/
\cite[(5.7--5.8) p228]{Macdonald}: given $r\in\ZZ_{>0}$,
\begin{equation}\label{eq:pieri}
P^\lambda P^{(r)} = \sum_{\substack{\nu:\ \nu/\lambda\text{ is a}\\\text{horizontal strip}\\\text{with $r$ boxes}}}
P^\nu\,
\frac{1}{1-t}
\prod_{i\in I_{\nu/\lambda}} (1-t^{m_i(\nu)})
\end{equation}
where $I_{\nu/\lambda}$ is the set of $i$ such that $\nu/\lambda$ has a box in column $i$ but not in column $i+1$.

As an aside, we note that
since the $P^{(r)}$, $r\in\ZZ_{>0}$, generate $\Lambda_k$ as an algebra \cite[p 209]{Macdonald},
and since in \eqref{eq:pieri},
$\nu_1\ge \lambda_1$ and $\nu'_1\ge \lambda'_1$,
the {\em ideal}\/ of $\Lambda_k$ generated by the $P^\lambda$, $\lambda\not\in \P_{k,n}$, is equal
to the linear span of these $P^\lambda$.
This implies, as noted in the introduction, that the $P^\lambda$, $\lambda\in \P_{k,n}$, form a basis of $\Lambda_{k,n}$, and that the structure constants of $\Lambda_{k,n}$ are the same as those of $\Lambda$.

Now {\em define}\/ a bilinear operation 
$\times$ on the linear span $\Lambda_{k,n}$ of the $P^\lambda$, $\lambda\in \P_{k,n}$, by
\[
P^\lambda\times P^\mu = \sum_{H:\ \substack{\text{NW boundary of }H=\lambda\\\text{NE boundary of }H=\mu}} \fug(H)\ P^{\text{S boundary of $H$}}
\]
Theorem~\ref{thm:main} says that $\times$ agrees with the usual product of $\Lambda_{k,n}$.
We shall first show that the Pieri rule \eqref{eq:pieri} 
is satisfied by $\times$, i.e.,
\begin{prop}\label{prop:pieri}
One has, for any $r\in\ZZ_{>0}$,
\[
P^\lambda\times P^{(r)}=P^\lambda P^{(r)}
\]
as a relation in $\Lambda_{k,n}$.
\end{prop}

Secondly we show
\begin{prop}\label{prop:assoc}
$\times$ is associative.
\end{prop}
In other words, the coefficients
$\sum \big\{ \fug(H)\, :\, H\ \text{honeycomb with boundaries $\lambda$, $\mu$, $\nu$}) \big\}$
are the structure constants of an associative algebra.

Theorem~\ref{thm:main} then follows from Propositions~\ref{prop:pieri} and \ref{prop:assoc}, as we recall briefly.
The $P^{(r)}$, $r\in\ZZ_{>0}$, generate $\Lambda_{k,n}$ as an algebra,
and the $P^\lambda$, $\lambda\in \P_{k,n}$ are a linear basis of it,
so we can content ourselves with showing
\[
P^\lambda\times (P^{(r_1)}\ldots P^{(r_\ell)})=
P^\lambda P^{(r_1)}\ldots P^{(r_\ell)}
\]
We prove this by induction on $\ell$:
\begin{align*}
P^\lambda\times (P^{(r_1)}\ldots P^{(r_\ell)})&=
P^\lambda\times (P^{(r_1)}\ldots P^{(r_{\ell-1})} P^{(r_\ell)})
\\
&=
P^\lambda\times (P^{(r_1)}\ldots P^{(r_{\ell-1})}\times P^{(r_\ell)})&\text{by Prop.~\ref{prop:pieri}}&
\\
&=
(P^\lambda\times P^{(r_1)}\ldots P^{(r_{\ell-1})})\times P^{(r_\ell)}&\text{by Prop.\ \ref{prop:assoc}}&
\\
&=
(P^\lambda P^{(r_1)}\ldots P^{(r_{\ell-1})})\times P^{(r_\ell)}&\text{by induction}&
\\
&=
P^\lambda P^{(r_1)}\ldots P^{(r_{\ell-1})}P^{(r_\ell)}&\text{by Prop.~\ref{prop:pieri}}&
\end{align*}

Section~\ref{sec:pieri} is devoted to the proof of Proposition~\ref{prop:pieri}.
Section~\ref{sec:bosonic}, though not strictly necessary for the proof, 
introduces the formalism of ``tensor calculus''  which is used in related work on puzzles
(see in particular \cite{artic46,artic71}); the same type of graphical calculus is then used in 
Section~\ref{sec:assoc}, which contains the proof of Proposition~\ref{prop:assoc}. 
In fact, it is perhaps in the latter proof that the interest of the paper lies, 
rather than in the result of Theorem~\ref{thm:main} itself.
Indeed we prove associativity by reducing it to elementary ``excavation'' moves of a tetrahedron, in the same spirit
of three-dimensional geometry as \cite{KTW-octa}; 
however, a major difference is that while the method of \cite{KTW-octa} is {\em combinatorial},
resulting in a bijective proof of associativity, our method is {\em linear algebraic}\/ (in fact, secretly
representation-theoretic), expressing the whole of $c^{\lambda,\mu}_\nu$ as a certain entry of a tensor
and then manipulating it using linear algebra identities, rather than manipulating individual honeycombs. (In fact,
we show in Appendix~\ref{app:assoc}
an example for which no fugacity-preserving bijection between pairs of honeycombs exists, barring a
combinatorial proof of associativity away from $t=0$.)
\rem{refer to Hannah?}

\section{Pieri rule}\label{sec:pieri}
\rem[gray]{technically we don't need to prove that we get the right diagrams --
it's true because of the $t=0$ case.}

Given a partition $\lambda\in \P_{k,n}$ viewed as a Young diagram, and $\nu\in \P_{k,n}$ obtained from $\lambda$ by addition
of a horizontal strip with $r$ boxes ($1\le r\le k$), subdivide the strip $\nu/\lambda$
into subsets of boxes in consecutive columns; indexing them $1,\ldots,\ell$
from top to bottom,
denote by $c_i$ the column of the rightmost box of the $i^{\rm th}$ subset,
and by $b_i$ its number of boxes.

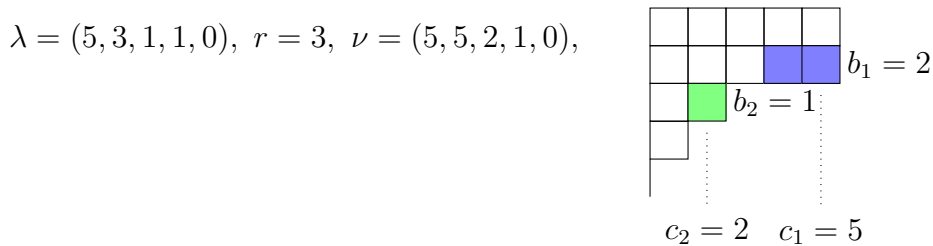
\begin{ex*}
\[
\lambda=(5,3,1,1,0),\ r=3,\ \nu=(5,5,2,1,0),\qquad
\tikz[baseline=0]{\node[tableau]{&&&&\\&&&\colorcell{blue!50!white}&\colorcell{blue!50!white}\\&\colorcell{green!50!white}\\\\\vcell\\};
\path (tab-2-5) node[right=2mm] {$b_1=2$};
\draw[dotted] (tab-2-5) ++ (0,-0.4) -- ++(0,-1.5) node[below] {$c_1=5$}; 
\path (tab-3-2) node[right=2mm] {$b_2=1$};
\draw[dotted] (tab-3-2) ++ (0,-0.4) -- ++(0,-1) node[below] {$c_2=2$}; 
}
\]
\[
\lambda=(3,2,1,1,0),\ r=3,\ \nu=(4,3,1,1,1),\qquad
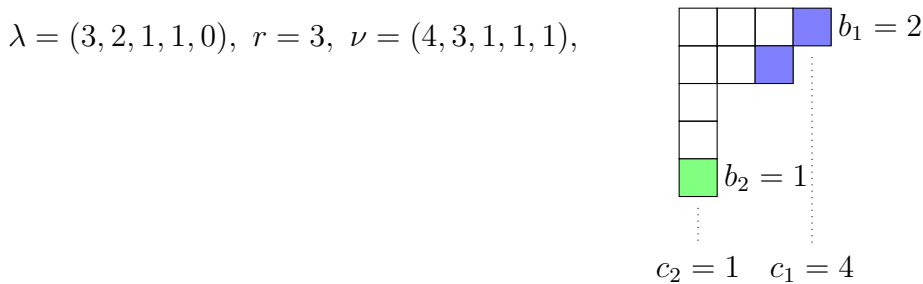
\tikz[baseline=0]{\node[tableau]{&&&\colorcell{blue!50!white}\\&&\colorcell{blue!50!white}\\\\\\\colorcell{green!50!white}\\};
\path (tab-1-4) node[right=2mm] {$b_1=2$};
\draw[dotted] (tab-1-4) ++ (0,-0.4) -- ++(0,-2.5) node[below] {$c_1=4$}; 
\path (tab-5-1) node[right=2mm] {$b_2=1$};
\draw[dotted] (tab-5-1) ++ (0,-0.4) -- ++(0,-0.5) node[below] {$c_2=1$}; 
}
\]
\end{ex*}

Construct a honeycomb with boundaries $\lambda$, $\mu=(r)$, $\nu$ as follows.
Consider the path starting on the right side of the honeycomb
at location $r$, and successively make it go 
South-West until it reaches the NW/SE column
numbered $c_1$, then make it go $b_1$ steps straight West, then again
South-West till it reaches NW/SE column $c_2$, etc, and finally
$b_\ell$ steps to the left, at which point it reaches the $0^{\rm th}$
(leftmost) NE/SW column.

Note that the rest of the honeycomb is then uniquely determined by the balance condition;\footnote{
We leave the proof of uniqueness as an exercise to the reader; note that it can in principle
be extracted from the corresponding proof of \cite[Prop.~4]{KTW-octa} by applying the honeycomb--hive bijection.}
in particular, in the first NE/SW column, vertices are of the form
\tikz[baseline=-3pt,scale=0.5]{
  \draw (0,0) -- (180:1) node[left=-1mm,myred] {$\ss i'$};
  \draw  (0,0) -- (300:1) node[right=-1mm,myred] {$\ss i'$};
  \draw (240:1) node[left=-1mm,myred] {$\ss i-i'$} -- (60:1) node[right=-1mm,myred] {$\ss i$};
},
% \tikz[baseline=-3pt,scale=0.5]{
% \draw (60:1) -- node[pos=0.4,myred] {$\ss i$} (0,0) -- node[pos=0.6,mygreen] {$\ss j''$} (240:1);
%   \draw (0,0) -- node[pos=0.6,myred] {$\ss i'$} (180:1);
%   \draw (0,0) -- node[pos=0.6,myred] {$\ss i'$} (300:1);
% } with $j''=i-i'$,
except at the spot where the special path described above arrives from the right, where we have
\tikz[baseline=-3pt,scale=0.5]{
  \draw (0:1) -- (180:1) node[left=-1mm,myred] {$\ss i'$};
  \draw  (0,0) -- (300:1) node[right=-1mm,myred] {$\ss i'-1$};
  \draw (240:1) node[left=-1mm,myred] {$\ss i-i'+1$} -- (60:1) node[right=-1mm,myred] {$\ss i$};
}.
% \tikz[baseline=-3pt,scale=0.5]{
%   \draw (60:1) -- node[pos=0.4,myred] {$\ss i$} (0,0) -- node[pos=0.6,mygreen] {$\ss j''$} (240:1);
%   \draw (0:1) -- (0,0) -- node[pos=0.6,myred] {$\ss i'$} (180:1);
%   \draw (0,0) -- node[pos=0.6,myred] {$\ss\ i'-1$} (300:1);
% } with $j''=i-i'+1$.

\begin{ex*}With the same partitions as above, and $n=6$, we obtain
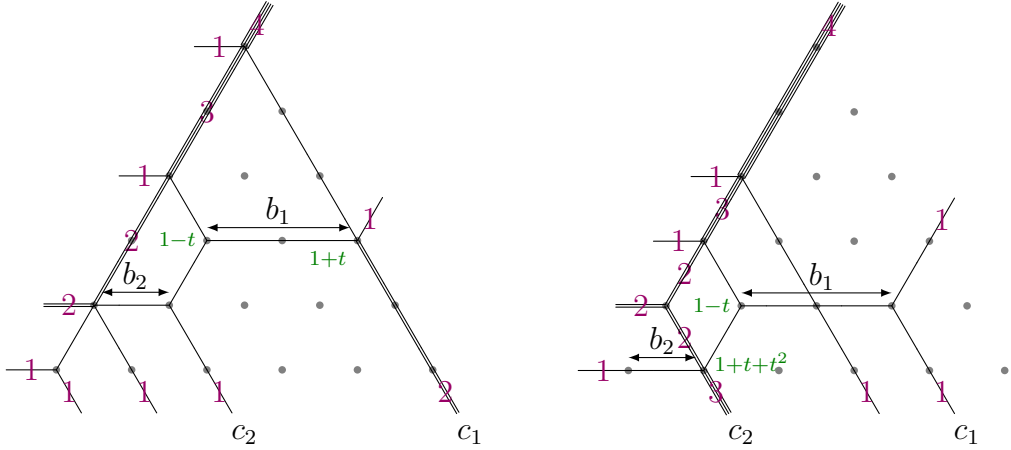
\begin{center}
\def\size{6}\honeycomb[{
  \node[myred] at (1/3,1+1/3) {3};
  \node[myred] at (1/3,3+1/3) {2};
  \path (1+1/3,2+1/3) node[left,green!50!black] {\ss 1-t};
  \path (3+1/3,0+1/3) node[below left,green!50!black] {\ss 1+t};
%  \draw[myred] (5.67,0.33) circle[radius=0.2cm];
%  \draw[myred] (2.67,3.33) circle[radius=0.2cm];
\node at (6+1/3,1/3) {c_1};
\node at (3+1/3,3+1/3) {c_2};
\draw[latex-latex] (3+1/3-0.2,1/3) -- node[above=-1mm] {b_1} (1+1/3-0.1,2+1/3-0.1);
\draw[latex-latex] (1+1/3-0.1,3+1/3-0.1) -- node[above=-1mm] {b_2} (1/3,4+1/3-0.2);
}]{4/0/1/0/1/3,0/0/0/1/1/0,0/0/0/1/1/0,1/0/1/1/2/0,0/0/0/2/2/0,0/0/0/2/2/0,3/0/0/0/0/3,0/0/0/0/0/0,0/1/1/0/0/0,0/0/0/0/0/0,0/0/0/0/0/0,3/0/1/0/1/2,0/1/0/1/0/1,0/0/0/0/0/0,0/0/0/0/0/0,2/0/0/0/0/2,1/0/1/0/1/0,0/0/0/1/1/0,2/1/2/0/1/1,0/0/0/1/1/0,1/0/1/0/1/0}%{3/0/0/0/0/3,0/0/0/0/0/0,0/0/0/0/0/0,1/0/1/0/1/0,0/0/0/1/1/0,0/0/0/1/1/0,3/0/0/0/0/3,0/0/0/0/0/0,0/1/1/0/0/0,0/0/0/0/0/0,0/0/0/0/0/0,3/0/1/0/1/2,0/1/0/1/0/1,0/0/0/0/0/0,0/0/0/0/0/0,2/0/0/0/0/2,1/0/1/0/1/0,0/0/0/1/1/0,2/1/2/0/1/1,0/0/0/1/1/0,1/0/1/0/1/0}
\qquad
\honeycomb[{
  \node[myred] at (1/3,2.5+1/3) {3};
  \node[myred] at (1/3,3.5+1/3) {2};
  \node[myred] at (0.5+1/3,4+1/3) {2};
  \path (1+1/3,3+1/3) node[left,green!50!black] {\ss 1-t};
  \path (1+1/3,4+1/3) node[right,yshift=1mm,green!50!black] {\ss 1+t+t^2};
\draw[latex-latex] (3+1/3-0.1,1+1/3-0.1) -- node[above=-1mm] {\ b_1} (1+1/3-0.1,3+1/3-0.1);
\draw[latex-latex] (1+1/3-0.2,4+1/3) -- node[above=-1mm] {b_2} (1/3-0.1,5+1/3-0.1);
\node at (5.33,1.33) {c_1};
\node at (2.33,4.33) {c_2};
}]{4/0/0/0/0/4,0/0/0/0/0/0,0/0/0/0/0/0,1/0/0/0/0/1,0/0/0/0/0/0,0/0/0/0/0/0,4/0/0/0/0/4,0/0/0/0/0/0,0/0/0/0/0/0,1/0/1/0/1/0,0/0/0/1/1/0,4/0/1/0/1/3,0/0/0/1/1/0,0/1/1/1/1/0,0/0/0/1/1/0,3/0/1/0/1/2,0/1/0/1/0/1,0/0/0/0/0/0,2/0/2/0/2/0,1/0/1/2/3/0,0/1/1/0/0/0}
%{3/0/0/0/0/3,0/0/0/0/0/0,0/0/0/0/0/0,1/0/0/0/0/1,0/0/0/0/0/0,0/0/0/0/0/0,3/0/0/0/0/3,0/0/0/0/0/0,0/0/0/0/0/0,1/0/1/0/1/0,0/0/0/1/1/0,3/0/1/0/1/2,0/0/0/1/1/0,0/1/1/1/1/0,0/0/0/1/1/0,2/0/1/0/1/1,0/1/0/1/0/1,0/0/0/0/0/0,1/0/1/0/1/0,1/0/1/1/2/0,0/1/1/0/0/0}
%{3/0/0/0/0/3,0/0/0/0/0/0,0/0/0/0/0/0,1/0/1/0/1/0,0/0/0/1/1/0,3/0/1/0/1/2,0/0/0/1/1/0,0/1/1/1/1/0,0/0/0/1/1/0,2/0/1/0/1/1,0/1/0/1/0/1,0/0/0/0/0/0,1/0/1/0/1/0,1/0/1/1/2/0,0/1/1/0/0/0}
\end{center}
\end{ex*}

The fugacity of the honeycomb is entirely concentrated along the special path; 
each right turn of the path at NW/SE column
$c_i$, of the form
\def\size{1}\tikzset{puznode/.style={font=\scriptsize}}%
\tikz[baseline=-3pt,scale=0.5]{
  \draw (0,0) -- (180:1);
  \draw (120:1) node[left=-1mm,myred] {$\ss m-1$} -- (300:1) node[right=-1mm,myred] {$\ss m$};
  \draw (0,0) -- (60:1);
}
% \tikz[baseline=-3pt,scale=0.5]{
%   \draw (0,0) -- (180:1);
%   \draw[very thick] (120:1)  -- node[mygreen,pos=0.15] {$\ss m-1$} node[myred,pos=0.75] {$\ss m$} (300:1);
%   \draw (0,0) -- (60:1);
% }, 
with $m=m_{c_i}(\nu)$,
incurs according to \eqref{eq:fug} a fugacity of % i=1 j'=m-1 i'=1 i''=m
\[
\sum_{r=0}^1
(-1)^r
t^{(m-1)r+r(r+1)/2}\frac{1}{\a_r \a_{1-r}}
=
\frac{1-t^m}{1-t}
\] 
whereas each left turn, of the form
\tikz[baseline=-3pt,scale=0.5]{
  \draw (120:1) -- (0,0) (240:1) -- (0,0) (0:1) -- (0,0);
}, 
incurs according to the same formula a fugacity of $\a_1=1-t$. The total fugacity is therefore
\[
(1-t)^{\ell-1}\prod_{i=1}^\ell \frac{1-t^{m_{c_i}(\nu)}}{1-t}
=\frac{1}{1-t}\prod_{i=1}^\ell (1-t^{m_{c_i}(\nu)})
\]
This coincides with \cite[(5.7--5.8) p228]{Macdonald}, taking into account
the notation $q_r=(1-t)P^{(r)}$ used by that reference. Proposition~\ref{prop:pieri} follows.

\tikzset{bgline/.style={gray,dotted}}%style for boundary of tiles
\tikzset{lattice/.style={fill=none,circle,inner sep=0.5pt}}%style for vertices of lattice
\tikzset{puznode/.style={font=\scriptsize}}
\section{Honeycombs as bosonic puzzles}\label{sec:bosonic}
We now pause to provide several alternative graphical representations of honeycombs.

\subsection{Bosonic puzzles}\label{sec:puzzle}
The first transformation is a ``puzzle-like'' representation of honeycombs: we draw a new triangular lattice
which is obtained from the original one by shifting every vertex in such a way that vertices of the old lattice
lie at the centers of up-pointing triangles of the new one;
and each time a line segment of a honeycomb crosses
an edge of this new triangular lattice, we record its multiplicity (and label empty edges with zeroes).
To each edge is therefore associated two numbers, which for improved readability we write in two colors,
purple and cyan, so that around each vertex of a honeycomb the colors alternate as follows:
\tikz[baseline=0,scale=0.5]{
  \draw (0:1) node[right=-1mm,mygreen] {$\ss j$} -- (180:1) node[left=-1mm,myred] {$\ss i'$};
  \draw (120:1) node[left=-1mm,mygreen] {$\ss j'$} -- (300:1) node[right=-1mm,myred] {$\ss i''$};
  \draw (240:1) node[left=-1mm,mygreen] {$\ss j''$} -- (60:1) node[right=-1mm,myred] {$\ss i$};
}.
We then erase the original honeycomb, keeping only the cyan and purple labels.

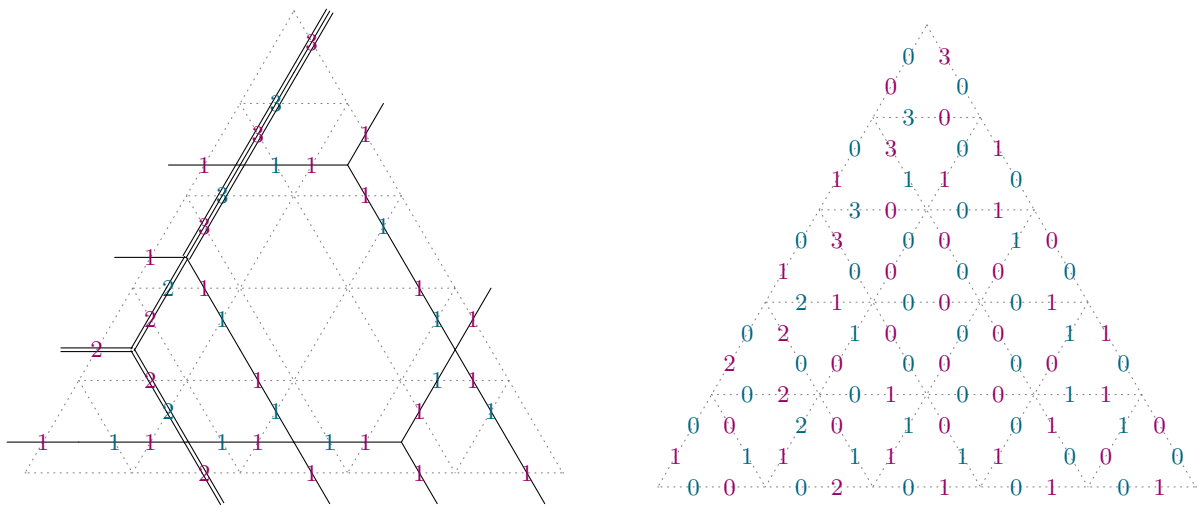
\begin{ex*}\ 

\begin{center}
\def\size{5}
\def\puzzlescale{1}
\honeycomb%
{3/0/0/0/0/3,1/0/1/0/1/0,0/0/0/1/1/0,1/0/0/1/1/1,0/0/0/1/1/0,3/1/1/0/0/3,0/0/0/0/0/0,0/0/0/0/0/0,1/0/1/0/1/0,3/0/1/0/1/2,0/0/0/1/1/0,0/1/1/1/1/0,2/0/2/0/2/0,0/1/1/2/2/0,0/1/1/0/0/0}
\qquad
\puzzlex%
{3/0/0/0/0/3,1/0/1/0/1/0,0/0/0/1/1/0,1/0/0/1/1/1,0/0/0/1/1/0,3/1/1/0/0/3,0/0/0/0/0/0,0/0/0/0/0/0,1/0/1/0/1/0,3/0/1/0/1/2,0/0/0/1/1/0,0/1/1/1/1/0,2/0/2/0/2/0,0/1/1/2/2/0,0/1/1/0/0/0}
\end{center}
\end{ex*}
The resulting picture is reminiscent of puzzles as defined in \cite{KTW-II}, except that the actual labels associated
to edges are quite different. In fact,
in order to dispel possible confusion, let us point out that this new representation 
is not directly related to the well-known fact that
honeycombs are in bijection with ordinary puzzles: here the puzzles that
we obtain are ``bosonic'', in the sense that they are associated with certain parabolic Verma modules
of $\mathcal U_{t^{1/2}}(\mathfrak{sl}_3)$ (see Section~\ref{sec:RT}).
The ordinary puzzles are ``fermionic'' in nature, and in this context,
the bijection between honeycombs and puzzles can be interpreted as a form of boson-fermion correspondence.

\subsection{The tensor calculus}\label{sec:tensor}
We now implement the same procedure that was formulated in \cite{artic46} and subsequently used in
\cite{artic68,artic71} to turn puzzles into entries of a certain tensor.
In order to do so, we switch to the more traditional
graphical calculus of mathematical physics:
this time, starting from our (new) triangular lattice,
we draw its dual honeycomb lattice, and transport the labels
of each edge of the former to the edge of the latter which it intersects. There is also a conventional
choice of orientation of each edge: we declare that all edges are oriented upwards.
This means that up- and down-pointing triangles now look like
\begin{equation}\label{eq:vertices}
\tikz[baseline=0,scale=1.2]{
  \draw[invarrow] (0,0) -- (-90:1) node[left,mygreen] {$\ss j''$} node[right,myred] {$\ss i''$};
  \draw[arrow] (0,0) -- (30:1) node[below,mygreen] {$\ss j$} node[above,myred] {$\ss i$};
  \draw[arrow] (0,0) -- (150:1) node[above,mygreen] {$\ss j'$} node[below,myred] {$\ss i'$};
}
\qquad
\tikz[baseline=0,scale=1.2]{
  \draw[arrow] (0,0) -- (90:1) node[left,mygreen] {$\ss j''$} node[right,myred] {$\ss i''$};
  \draw[invarrow] (0,0) -- (210:1) node[below,mygreen] {$\ss j$} node[above,myred] {$\ss i$};
  \draw[invarrow] (0,0) -- (-30:1) node[above,mygreen] {$\ss j'$} node[below,myred] {$\ss i'$};
}
\end{equation}
which we call $U$ and $D$ vertices respectively. 

To SouthWest, SouthEast, South oriented edges we shall associate
three (infinite-dimensional) vector spaces $V,V',\bar V''$, each of which equipped with
a basis indexed by pairs of nonnegative integers. 
(The bar will be justified in Section~\ref{sec:Z3}.
In what follows the bar denotes linear algebra duality. 
Changing the orientation of a line corresponds to switching to the dual vector space.
We shall indeed discuss another, perhaps more natural, choice of orientation in Section~\ref{sec:Z3};
it is however less convenient for the generalization we have in mind in Section~\ref{sec:assoc}.)

These vector spaces possess a $L_2\cong\ZZ^2$-grading,
the {\em weight} ($L_2$ was defined in \eqref{eq:defL2}); 
the basis vectors $v_{i,j}\in V$, $(i,j)\in \ZZ^2_{\ge0}$, are homogeneous of weight 
$\wt(v_{i,j})=(i,-i-j,j)$, and similarly,
$\wt(v'_{i',j'})=(j',i',-i'-j')$,
$\wt(\bar v''_{i'',j''})=-(-i''-j'',j'',i'')$.

The fugacities are now encoded as entries of a tensor associated to each vertex; namely,
to a $U$ (resp.\ $D$) vertex is associated an element
of $\bar V''\to V'\otimes V$, resp.\ $V\otimes V'\to \bar V''$.

Specifically, the entries are given by
\begin{align}\label{eq:defU}
U^{i',j',i,j}_{i'',j''}&=
\begin{cases}
\a_{i''}\a_{j''}u^{j,i,j',i',j'',i''}&\wt(v_{i,j})+\wt(v'_{i',j'})=\wt(\bar v''_{i'',j''})
\\
0&\text{else}
\end{cases}
\\\label{eq:defD}
D_{i,j,i',j'}^{i'',j''}&=
\begin{cases}
\rlap{$\a_j$}\phantom{\a_{i''}\a_{j''}y(j,i,j',i',j'',i'')}
&i=j'',\ i'=j,\ i''=j'
\\
0&\text{else}
\end{cases}
\end{align}
where we recall
$\a_i=\prod_{r=1}^{i} (1-t^r)$, and
$u^{j,i,j',i',j'',i''}$ is closely related to the fugacity that was introduced in \eqref{eq:fug}, 
and given by
\begin{equation}\label{eq:defy}
u^{j,i,j',i',j'',i''}
=
\frac{\a_{i+j}}{\a_i\a_{i'}\a_{i''}\a_j\a_{j''}}\,
\,{}_2\phi_1\left({t^{-i},t^{-i'}\atop t^{-(i+j)}};t,t^{i''+1}\right) 
\end{equation}

The condition of equality of weights in the definition \eqref{eq:defU} is known as {\em weight conservation}.
Note that the condition $i=j''$, $i'=j$, $i''=j'$ in the definition \eqref{eq:defD} implies
(but is stronger than) the weight conservation $\wt(v_{i,j})+\wt(v'_{i',j'})=\wt(\bar v''_{i'',j''})$.

One last (standard) graphical convention is that indices that are not marked are summed over. For example, in size $2$,
we can consider
\begin{equation}\label{eq:pic2}
\begin{tikzpicture}[scale=0.8,baseline=0]
  \draw[arrow] (0,0) -- (90:2);
  \draw[invarrow] (0,0) -- (210:2);
  \draw[invarrow] (0,0) -- (-30:2);
\begin{scope}[shift={(90:2)}]
  \draw[arrow] (0,0) -- (30:1) node[below,mygreen] {$\ss 0$} node[above,myred] {$\ss i_0$};
  \draw[arrow] (0,0) -- (150:1) node[above,mygreen] {$\ss 0$} node[below,myred] {$\ss i'_1$};
\end{scope}
\begin{scope}[shift={(210:2)}]
  \draw[invarrow] (0,0) -- (-90:1) node[left,mygreen] {$\ss 0$} node[right,myred] {$\ss i''_0$};
  \draw[arrow] (0,0) -- (150:1) node[above,mygreen] {$\ss 0$} node[below,myred] {$\ss i'_0$};
\end{scope}
\begin{scope}[shift={(-30:2)}]
  \draw[invarrow] (0,0) -- (-90:1) node[left,mygreen] {$\ss 0$} node[right,myred] {$\ss i''_1$};
  \draw[arrow] (0,0) -- (30:1) node[below,mygreen] {$\ss 0$} node[above,myred] {$\ss i_1$};
\end{scope}
\end{tikzpicture}
\end{equation}
Because of the implicit summation, this picture represents a certain entry of a product of $3$ $U$ tensors
and $1$ $D$ tensor.

The main result of this section is the simple reformulation:
\begin{lem}\label{lem:c}
$c^{\lambda,\mu}_\nu$ is the tensor entry corresponding to the following diagram ($n\times n$ triangle inside the
honeycomb lattice):
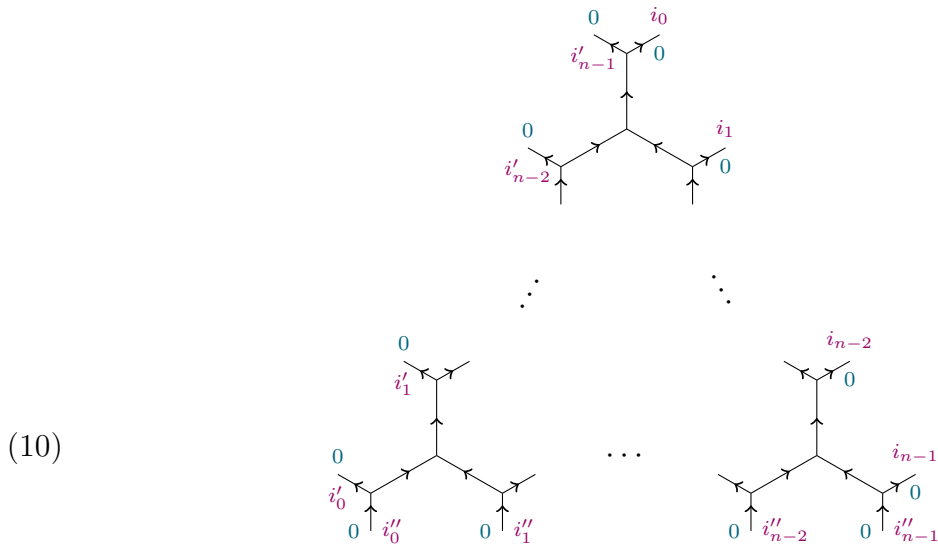
\begin{equation}\label{eq:c}
\begin{tikzpicture}[scale=0.5,baseline=0]
  \draw[arrow] (0,0) -- (90:2);
  \draw[invarrow] (0,0) -- (210:2);
  \draw[invarrow] (0,0) -- (-30:2);
\begin{scope}[shift={(90:2)}]
  \draw[arrow] (0,0) -- (30:1);
  \draw[arrow] (0,0) -- (150:1) node[above,mygreen] {$\ss 0$} node[below,myred] {$\ss i'_1$};
\end{scope}
\begin{scope}[shift={(210:2)}]
  \draw[invarrow] (0,0) -- (-90:1) node[left,mygreen] {$\ss 0$} node[right,myred] {$\ss i''_0$};
  \draw[arrow] (0,0) -- (150:1) node[above,mygreen] {$\ss 0$} node[below,myred] {$\ss i'_0$};
\end{scope}
\begin{scope}[shift={(-30:2)}]
  \draw[invarrow] (0,0) -- (-90:1) node[left,mygreen] {$\ss 0$} node[right,myred] {$\ss i''_1$};
  \draw[arrow] (0,0) -- (30:1);
\end{scope}
\node[rotate=60] at (60:5) {$\cdots$};
\begin{scope}[shift={(60:10)}]
  \draw[arrow] (0,0) -- (90:2);
  \draw[invarrow] (0,0) -- (210:2);
  \draw[invarrow] (0,0) -- (-30:2);
\begin{scope}[shift={(90:2)}]
  \draw[arrow] (0,0) -- (30:1) node[below,mygreen] {$\ss 0$} node[above,myred] {$\ss i_0$};
  \draw[arrow] (0,0) -- (150:1) node[above,mygreen] {$\ss 0$} node[below,myred] {$\ss i'_{n-1}$};
\end{scope}
\begin{scope}[shift={(210:2)}]
  \draw[invarrow] (0,0) -- (-90:1);
  \draw[arrow] (0,0) -- (150:1) node[above,mygreen] {$\ss 0$} node[below,myred] {$\ss i'_{n-2}$};
\end{scope}
\begin{scope}[shift={(-30:2)}]
  \draw[invarrow] (0,0) -- (-90:1);
  \draw[arrow] (0,0) -- (30:1) node[below,mygreen] {$\ss 0$} node[above,myred] {$\ss i_1$};
\end{scope}
\end{scope}
\path (60:5) ++(0:5) node[rotate=-60] {$\cdots$};
\node at (0:5) {$\cdots$};
\begin{scope}[shift={(0:10)}]
  \draw[arrow] (0,0) -- (90:2);
  \draw[invarrow] (0,0) -- (210:2);
  \draw[invarrow] (0,0) -- (-30:2);
\begin{scope}[shift={(90:2)}]
  \draw[arrow] (0,0) -- (30:1) node[below,mygreen] {$\ss 0$} node[above,myred] {$\ss i_{n-2}$};
  \draw[arrow] (0,0) -- (150:1);
\end{scope}
\begin{scope}[shift={(210:2)}]
  \draw[invarrow] (0,0) -- (-90:1) node[left,mygreen] {$\ss 0$} node[right,myred] {$\ss i''_{n-2}$};
  \draw[arrow] (0,0) -- (150:1);
\end{scope}
\begin{scope}[shift={(-30:2)}]
  \draw[invarrow] (0,0) -- (-90:1) node[left,mygreen] {$\ss 0$} node[right,myred] {$\ss i''_{n-1}$};
  \draw[arrow] (0,0) -- (30:1) node[below,mygreen] {$\ss 0$} node[above,myred] {$\ss i_{n-1}$};
\end{scope}
\end{scope}
\end{tikzpicture}
\end{equation}
with $i'_k=m_k(\lambda)$, $i_k=m_k(\mu)$, $i''_k=m_k(\nu)$, $k=0,\ldots,n-1$.
\end{lem}
(The picture is the generalization of \eqref{eq:pic2} to arbitrary $n$.)
\begin{proof}
At every edge of the diagram, we insert the decomposition of the identity in terms of bases
$v_{i,j},v'_{i',j'},v''_{i'',j''}$.
The definition \eqref{eq:defD} of the entries of $D$ simply means that honeycomb lines go across down-pointing triangles, as well as contributes a factor of $\a_j$ which we combine with the $U$ vertex below to the left of it.
As to the definition \eqref{eq:defU} of the entries of $U$,
it is easy to check that 
$\wt(v_{i,j})+\wt(v'_{i',j'})=\wt(\bar v''_{i'',j''})$ is exactly the balance condition for vertices of a
honeycomb; the contribution to the fugacity is the factor $\a_j$ coming from the $D$ vertex above and right
of it (noting that if the $U$ vertex is on the NorthEast boundary, $j=0$ so no such factor occurs) times
$\a_{i''}\a_{j''}u^{j,i,j',i',j'',i''}$; 
in the absence of a honeycomb vertex, i.e., if $j=i=j'=i'=j''=i''=0$, the result is $1$,
whereas in the presence of a honeycomb vertex we recover exactly the fugacity \eqref{eq:fug}.
Overall, the product over all entries involved in computing the tensor entry of the lemma
reproduces exactly the fugacity of the honeycomb; summing over possible values of $i,j$ or $i',j'$,
$i'',j''$ at every internal edge results in summing over all honeycombs.
\end{proof}

\subsection{The \texorpdfstring{$\ZZ_3$}{Z3}-invariant setting}\label{sec:Z3}
There is a more natural, $\ZZ_3$-invariant, choice of orientation,
which is to orient all edges from say up-pointing triangles to down-pointing triangles;
it leads to new vertices:
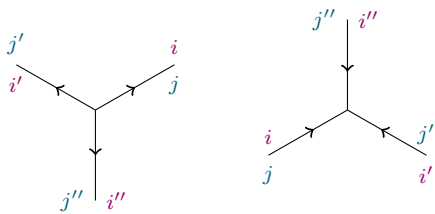
\[
\tikz[baseline=0,scale=1.2]{
  \draw[arrow] (0,0) -- (-90:1) node[left,mygreen] {$\ss j''$} node[right,myred] {$\ss i''$};
  \draw[arrow] (0,0) -- (30:1) node[below,mygreen] {$\ss j$} node[above,myred] {$\ss i$};
  \draw[arrow] (0,0) -- (150:1) node[above,mygreen] {$\ss j'$} node[below,myred] {$\ss i'$};
}
\qquad
\tikz[baseline=0,scale=1.2]{
  \draw[invarrow] (0,0) -- (90:1) node[left,mygreen] {$\ss j''$} node[right,myred] {$\ss i''$};
  \draw[invarrow] (0,0) -- (210:1) node[below,mygreen] {$\ss j$} node[above,myred] {$\ss i$};
  \draw[invarrow] (0,0) -- (-30:1) node[above,mygreen] {$\ss j'$} node[below,myred] {$\ss i'$};
}
\]
which we call $\tilde U$ and $\tilde D$ (these will not be used outside of this section).
\rem[gray]{careful that this is called $U_{sym'}$ from the notes, not $U_{sym}=\overset{4}{U}$}

The corresponding tensors are now $\tilde U\in V''\otimes V' \otimes V$ and $\tilde D:
V\otimes V'\otimes V''\to\CC$, where by definition $V''$ is the vector space dual to $\bar V''$ introduced
above; it comes equipped with the dual basis $v''_{i'',j''}$, $(i'',j'')\in \ZZ_{\ge0}^2$,
$\wt(v''_{i'',j''})=(-i''-j'',j'',i'')$.

The $\ZZ_3$-symmetry can also be promoted to their entries: write
\begin{align}\label{eq:defUt}
\tilde U^{i'',j'',i',j',i,j}
&=
\begin{cases}
u^{j,i,j',i',j'',i''}&\wt(v_{i,j})+\wt(v'_{i',j'})+\wt(v''_{i'',j''})=0
\\
0&\text{else}
\end{cases}
\\\label{eq:defDt}
\tilde D_{i,j,i',j',i'',j''}&=
\begin{cases}
\a_i\a_{i'}\a_{i''}
&i=j'',\ i'=j,\ i''=j'
\\
0&\text{else}
\end{cases}
\end{align}
where $u^{\ldots}$ was defined in \eqref{eq:defy}.

In fact, we have an even larger symmetry:
\begin{lem}\label{lem:D6}
Given $j,i,j',i',j'',i''\ge 0$ satisfying the weight conservation $i'-j=i''-j'=i-j''$, $u^{j,i,j',i',j'',i''}$ is invariant under the natural action of the dihedral
group $D_6$ on its variables.
\end{lem}
Note that the weight conservation itself has dihedral symmetry, restricting
the natural 6-dimensional representation of $D_6$ to a 4-dimensional subrepresentation.
\begin{proof}
Two order 2 transformations are easy to show. Noting $i+j=i'+j''$, we immediately
have from the definition \eqref{eq:defy} of $u^{\ldots}$ the invariance under the reflection
$(j,i,j',i',j'',i'')\mapsto (j'',i',j',i,j,i'')$.

Secondly, one of Heine's transformation formulae \cite[(III.3)]{GR-basic} for ${}_2\phi_1$ 
implies the invariance under $(j,i,j',i',j'',i'')\mapsto (i',j'',i'',j,i,j')$.

Together these generate a $\ZZ_2\times \ZZ_2$ subgroup of $D_6$. The other
symmetries do not seem to follow from standard transformation formulae of ${}_2\phi_1$, 
to the limited knowledge of the author. Instead one can prove them from the following representation
\cite[(III.8)]{GR-basic}:
\[
u^{j,i,j',i',j'',i''}
=\begin{cases}
\displaystyle\frac{{}_3\phi_1\left({t^{-i},t^{-i'},t^{-i''}\atop t^{c+1}};t,t^{i+i'+i''+c}\right)}{\a_i\a_{i'}\a_{i''}\a_c} 
&c\ge 0
\\[3mm]
\displaystyle\frac{{}_3\phi_1\left({t^{-j},t^{-j'},t^{-j''}\atop t^{-c+1}};t,t^{j+j'+j''-c}\right)}{\a_j\a_{j'}\a_{j''}\a_{-c}} 
&c\le 0
\end{cases}
\]
where to make the $\ZZ_3$ symmetry 
$(j,i,j',i',j'',i'')\mapsto (j',i',j'',i'',j,i)$
more apparent, we have written $c=j-i'=j'-i''=j''-i$.
Together, these transformations generate the whole of $D_6$.

\end{proof}

Finally, define $c^{\lambda,\mu,\nu}$ to be the tensor entry represented by the following diagram
\begin{equation}\label{eq:ct}
c^{\lambda,\mu,\nu}=
\begin{tikzpicture}[scale=0.5,baseline=0]
  \draw[invarrow] (0,0) -- (90:2);
  \draw[invarrow] (0,0) -- (210:2);
  \draw[invarrow] (0,0) -- (-30:2);
\begin{scope}[shift={(90:2)}]
  \draw[arrow] (0,0) -- (30:1);
  \draw[arrow] (0,0) -- (150:1) node[above,mygreen] {$\ss 0$} node[below,myred] {$\ss i'_1$};
\end{scope}
\begin{scope}[shift={(210:2)}]
  \draw[arrow] (0,0) -- (-90:1) node[left,mygreen] {$\ss 0$} node[right,myred] {$\ss i''_{n-1}$};
  \draw[arrow] (0,0) -- (150:1) node[above,mygreen] {$\ss 0$} node[below,myred] {$\ss i'_0$};
\end{scope}
\begin{scope}[shift={(-30:2)}]
  \draw[arrow] (0,0) -- (-90:1) node[left,mygreen] {$\ss 0$} node[right,myred] {$\ss i''_{n-2}$};
  \draw[arrow] (0,0) -- (30:1);
\end{scope}
\node[rotate=60] at (60:5) {$\cdots$};
\begin{scope}[shift={(60:10)}]
  \draw[invarrow] (0,0) -- (90:2);
  \draw[invarrow] (0,0) -- (210:2);
  \draw[invarrow] (0,0) -- (-30:2);
\begin{scope}[shift={(90:2)}]
  \draw[arrow] (0,0) -- (30:1) node[below,mygreen] {$\ss 0$} node[above,myred] {$\ss i_0$};
  \draw[arrow] (0,0) -- (150:1) node[above,mygreen] {$\ss 0$} node[below,myred] {$\ss i'_{n-1}$};
\end{scope}
\begin{scope}[shift={(210:2)}]
  \draw[arrow] (0,0) -- (-90:1);
  \draw[arrow] (0,0) -- (150:1) node[above,mygreen] {$\ss 0$} node[below,myred] {$\ss i'_{n-2}$};
\end{scope}
\begin{scope}[shift={(-30:2)}]
  \draw[arrow] (0,0) -- (-90:1);
  \draw[arrow] (0,0) -- (30:1) node[below,mygreen] {$\ss 0$} node[above,myred] {$\ss i_1$};
\end{scope}
\end{scope}
\path (60:5) ++(0:5) node[rotate=-60] {$\cdots$};
\node at (0:5) {$\cdots$};
\begin{scope}[shift={(0:10)}]
  \draw[invarrow] (0,0) -- (90:2);
  \draw[invarrow] (0,0) -- (210:2);
  \draw[invarrow] (0,0) -- (-30:2);
\begin{scope}[shift={(90:2)}]
  \draw[arrow] (0,0) -- (30:1) node[below,mygreen] {$\ss 0$} node[above,myred] {$\ss i_{n-2}$};
  \draw[arrow] (0,0) -- (150:1);
\end{scope}
\begin{scope}[shift={(210:2)}]
  \draw[arrow] (0,0) -- (-90:1) node[left,mygreen] {$\ss 0$} node[right,myred] {$\ss i''_1$};
  \draw[arrow] (0,0) -- (150:1);
\end{scope}
\begin{scope}[shift={(-30:2)}]
  \draw[arrow] (0,0) -- (-90:1) node[left,mygreen] {$\ss 0$} node[right,myred] {$\ss i''_0$};
  \draw[arrow] (0,0) -- (30:1) node[below,mygreen] {$\ss 0$} node[above,myred] {$\ss i_{n-1}$};
\end{scope}
\end{scope}
\end{tikzpicture}
\end{equation}
with $i'_k=m_k(\lambda)$, $i_k=m_k(\mu)$, $i''_k=m_k(\nu)$, $k=0,\ldots,n-1$. Note that
$\nu$ is now read in reverse at the bottom; in fact, let us denote $\nu^*$ to be
the partition such that $m_{n-1-k}(\nu^*)=m_k(\nu)$ for $k=0,\ldots,n-1$. In terms of Young
diagrams, this corresponds to taking the complement inside the $k\times (n-1)$ rectangle
and then rotating 180 degrees.

Though this choice of orientation may be more natural,
the $c^{\lambda,\mu,\nu}$ thus defined are not exactly computing the structure constants we are after. Rather,
we have the following.
Define, for any $\lambda\in \P_{k,m}$, 
\begin{equation}\label{eq:defh}
h_\lambda=
\prod_{r=0}^{n-1}\a_{m_r(\lambda)}
\end{equation}
Then
\begin{prop}\label{prop:Z3}
(a) $c^{\lambda,\mu,\nu}$ is related
to $c^{\lambda,\mu}_\nu$ by
\[
c^{\lambda,\mu,\nu}=
h_\nu^{-1}
c^{\lambda,\mu}_{\nu^*}
\]
(b) $c^{\lambda,\mu,\nu}$ is invariant under cyclic permutation of the $\lambda,\mu,\nu$.
\end{prop}
\begin{proof}
Let us compare the diagrams \eqref{eq:c} and \eqref{eq:ct} corresponding to $c^{\lambda,\mu}_{\nu^*}$ and $c^{\lambda,\mu,\nu}$ respectively. 
The labels on the boundaries match; the vertices have fugacities given by
\eqref{eq:defU}, \eqref{eq:defD}
and
\eqref{eq:defUt} \eqref{eq:defDt} respectively. The conditions for being
nonzero also match, so the only difference is in the nonzero entries. 
One has the following relation between them:
\begin{align*}
\tilde U&=\a_{i''}^{-1}\a_{j''}^{-1} U
\\
\tilde D&=\a_{i''}\a_{j''} D
\end{align*}
The factors in the tilded fugacities only occur on vertical edges.
Because each (vertical) edge that is summed over in the diagram connects a $U$ and a $D$ vertex, the factors above exactly cancel, so that the only difference occurs at the boundary (vertical) edges. The latter only occur at the bottom boundary of the diagram, and there one has all $j''=0$ and the $i''$ labels form the sequence $m_k(\nu)$, $k=0,\ldots,n-1$. This leads to
the desired relation in view of the definition \eqref{eq:defh}.

Due to Lemma~\ref{lem:D6}, the fugacities $\tilde U$ and $\tilde D$
are invariant by cyclic shift $(j,i,j',i',j'',i'')\mapsto(j',i',j'',i'',j,i)$.
This implies $c^{\lambda,\mu,\nu}=c^{\mu,\nu,\lambda}$.
\end{proof}

\begin{rmk*}
The $c^{\lambda,\mu,\nu}$ have the following interpretation.
Consider {\em dual}\/ Hall--Littlewood polynomials
\[
P_\lambda:=h_\lambda P^{\lambda^*}
\]
(these are a natural ``finitized'' version of the usual dual Hall--Littlewood
polynomials). Then our main Theorem \ref{thm:main} is trivially equivalent
to the fact that $c^{\lambda,\mu,\nu}$ is the coefficient of $P_\nu$ in
the expansion of $P^\lambda P^\mu$. Of course such an interpretation implies
that $c^{\lambda,\mu,\nu}$ is invariant under {\em every}\/ permutation of $\lambda,\mu,\nu$.

Similarly, had we used the fugacities
\begin{align}\label{eq:defu}
u_{j,i,j',i',j'',i''}&=\a_j \a_i \a_{j'} \a_{i'}\a_{j''}\a_{i''}
u^{j,i,j',i',j'',i''}
\\\notag
&=
\a_{i+j}\a_{j'}
\,{}_2\phi_1\left({t^{-i},t^{-i'}\atop t^{-(i+j)}};t,t^{i''+1}\right) 
\end{align}
% \[
% u_{j,i+1,j',i',j'',i''}
% =t^{i'}(1-t^{j''})u_{j,i,j',i',j''-1,i''}
% +(1-t^{i'})u_{j,i,j'+1,i'-1,j'',i''}
% \]
% (for all nonnegative integer values of the arguments)
% as well as all relations obtained from this one by $D_6$ action.
% \rem{the easiest one is the same with one less prime}
% Clearly, $u_{\ldots}$ is uniquely determined by these relations together with the initial
% condition $u_{0,0,0,0,0,0}=1$. The desired symmetry of $u_{\ldots}$, and therefore of $u^{\ldots}$,
% follows.
for $U$ vertices 
instead of $u^{\ldots}$, correspondingly redefined the fugacity of $D$ vertices to be
$(\a_j\a_{j'}\a_{j''})^{-1}$,
and labelled our puzzles counterclockwise, we would find a quantity $c_{\lambda,\mu,\nu}$ which
is nothing but the coefficient of $P^\nu$ in the expansion of $P_\lambda P_\mu$.
\end{rmk*}

\subsection{The representation theory}\label{sec:RT}
Although outside the scope of the present paper, we briefly sketch the representation-theoretic
interpretation of $U$ and $D$. $V$, $V'$ and $V''$ can be endowed with 
an action of the quantized algebra $\mathcal U_{t^{1/2}}(\mathfrak{sl}_3)$,
in such a way that they are parabolic Verma modules for distinct parabolic subalgebras,
their parabolic subalgebras and highest weights being related
to each other by 120 degree rotation of the weight lattice (the weights defined above are always
relative to the highest weight). Then one can show that there exist
intertwiners $V\otimes V'\to \bar V''$ and $\bar V''\to V'\otimes V$ which are unique up to normalization. There is only one parameter in the definition of 
such highest weights; call it $s$.
Finally, take the limit $t^s\to 0$;
the intertwiners then take the form \eqref{eq:defU}--\eqref{eq:defD}
(up to switching them, depending on the conventional sign of $s$).

\begin{rmk*}
Had we kept $s$ finite, we would have obtained instead the product rule for 
rank 1 Bethe wave functions of arbitrary spin
(also known as ``spin Hall--Littlewood functions'' in the recent literature), 
see e.g.~\cite{Bor-symfun} for their definition. This would encompass both the product rule of the present paper and those
of \cite{artic46} and \cite{artic68} (see also \cite{artic71} for the justification of the occurrence of the root system of
$\mathfrak{sl}_3$).
\end{rmk*}

\section{Associativity}\label{sec:assoc}
\subsection{The 3D geometry intepretation}\label{sec:3d}
We first briefly recall the interpretation of associativity in terms of three-dimensional geometry,
as advocated in \cite{KTW-octa}. Expanding $(P^\lambda P^\mu)P^\nu=P^\lambda(P^\mu P^\nu)$, we find
the quadratic equations
\[
\sum_{\sigma\in \P_{k,n}} c^{\lambda,\mu}_\sigma c^{\sigma,\nu}_{\rho}
=
\sum_{\tau\in \P_{k,n}} c^{\lambda,\tau}_{\rho} c^{\mu,\nu}_\tau 
\qquad
\forall\ \lambda,\mu,\nu,\rho\in\P_{k,n}
\]
It is natural to depict this equation in the usual two dual ways as
\begin{equation}\label{eq:assoc}
\begin{tikzpicture}[baseline=-3pt]
\shade[top color=white!50!black,bottom color=white] (-1,0) -- (0,1) -- (1,0) -- cycle;
\shade[bottom color=white!50!black,top color=white] (-1,0) -- (0,-1) -- (1,0) -- cycle;
\draw (-1,0) -- node {$\lambda\,$} (0,1) -- node {$\mu$} (1,0) -- node {$\nu$} (0,-1) -- node {$\rho$} cycle;
\draw (1,0) -- node {$\sigma$} (-1,0);
\end{tikzpicture}
=
\begin{tikzpicture}[baseline=-3pt]
\shade[left color=white!50!black,right color=white] (0,-1) -- (1,0) -- (0,1) -- cycle;
\shade[right color=white!50!black,left color=white] (0,-1) -- (-1,0) -- (0,1) -- cycle;
\draw (-1,0) -- node {$\lambda\,$} (0,1) -- node {$\mu$} (1,0) -- node {$\nu$} (0,-1) -- node {$\rho$} cycle;
\draw (0,-1) -- node {$\tau$} (0,1);
\end{tikzpicture}
\qquad
\text{or}
\qquad
\begin{tikzpicture}[baseline=-3pt,rotate=90]
\draw[ultra thick,arrow] (-1,-1) -- node[below] {$\nu$} (-0.25,0);
\draw[ultra thick,arrow] (0.25,0) -- node[right] {$\sigma$} (-0.25,0);
\draw[ultra thick,arrow] (0.25,0) -- node[above] {$\lambda$} (1,1);
\draw[ultra thick,invarrow] (-1,1) -- node[below] {$\rho$} (-0.25,0);
\draw[ultra thick,arrow] (1,-1) -- node[above] {$\mu$} (0.25,0);
\end{tikzpicture}
=
\begin{tikzpicture}[baseline=-3pt]
\draw[ultra thick,invarrow] (-1,-1) -- node[left] {$\rho$} (-0.25,0);
\draw[ultra thick,invarrow] (-0.25,0) -- node[above] {$\tau$} (0.25,0);
\draw[ultra thick,invarrow] (0.25,0) -- node[right] {$\mu$} (1,1);
\draw[ultra thick,arrow] (-1,1) -- node[left] {$\lambda$} (-0.25,0);
\draw[ultra thick,arrow] (1,-1) -- node[right] {$\nu$} (0.25,0);
\end{tikzpicture}
\end{equation}

Here the shaded triangles should be filled with an actual expression for $c^{\lambda,\mu}_\nu$ such as the
puzzles of Section~\ref{sec:puzzle};
or equivalently, the thick lines of the dual picture are really multiple lines, and the vertex a multi-vertex
similar to \eqref{eq:c}.

Let us focus on the left picture (the right picture will be used extensively in the next sections).
It is convenient to imagine it as a tetrahedron viewed from the top, where the l.h.s.\ 
corresponds to the full (opaque) tetrahedron with its top two faces shown, whereas the r.h.s.\ is
a view of its bottom two faces as if the tetrahedron had been excavated.

If we subdivide each triangle into smaller triangles, in the spirit of puzzles:
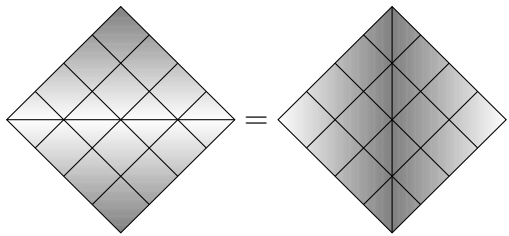
\[
\begin{tikzpicture}[baseline=-3pt,scale=1.5]
\shade[top color=white!50!black,bottom color=white] (-1,0) -- (0,1) -- (1,0) -- cycle;
\shade[bottom color=white!50!black,top color=white] (-1,0) -- (0,-1) -- (1,0) -- cycle;
\draw (-1,0) -- (0,1) -- (1,0) -- (0,-1) -- cycle;
\draw (1,0) -- (-1,0);
\draw (-0.75,-0.25) -- (0.25,0.75);
\draw (-0.5,-0.5) -- (0.5,0.5);
\draw (-0.25,-0.75) -- (0.75,0.25);
\draw (0.75,-0.25) -- (-0.25,0.75);
\draw (0.5,-0.5) -- (-0.5,0.5);
\draw (0.25,-0.75) -- (-0.75,0.25);
\end{tikzpicture}
=
\begin{tikzpicture}[baseline=-3pt,scale=1.5]
\shade[left color=white!50!black,right color=white] (0,-1) -- (1,0) -- (0,1) -- cycle;
\shade[right color=white!50!black,left color=white] (0,-1) -- (-1,0) -- (0,1) -- cycle;
\draw (-1,0) -- (0,1) -- (1,0) -- (0,-1) -- cycle;
\draw (0,-1) -- (0,1);
\draw (-0.75,-0.25) -- (0.25,0.75);
\draw (-0.5,-0.5) -- (0.5,0.5);
\draw (-0.25,-0.75) -- (0.75,0.25);
\draw (0.75,-0.25) -- (-0.25,0.75);
\draw (0.5,-0.5) -- (-0.5,0.5);
\draw (0.25,-0.75) -- (-0.75,0.25);
\end{tikzpicture}
\]
then one should correspondingly think of the tetrahedron as subdivided into smaller polyhedra. One finds
that the correct subdivision is into three types of polyhedra:
\begin{itemize}
\item {\em Tetrahedra}\/ which are obtained by homothecy from the full tetrahedron.
\item Other tetrahedra, called in what followed {\em dual tetrahedra}, obtained by top-bottom mirror
symmetry from the previous kind.
\item {\em Octahedra}.
\end{itemize}
All these polyhedra have equal edge length, which is $1/n$ times the edge length of the original tetrahedron.

The idea is then to prove associativity, i.e., \eqref{eq:assoc}, step by step by excavating the large
tetrahedron one small polyhedron at a time. Each kind of polyhedron 
(tetrahedron, dual tetrahedron and octahedron) corresponds to a {\em local}\/ transformation
of the puzzle-like objects.
This idea is realized in the context of hives in \cite{KTW-octa} (with uniform fugacities, corresponding
to our $t=0$ case). (See also Appendix~\ref{app:assoc} for a $n=3$ example in our context.)

In order to implement these local transformations for honeycombs, 
we need to extend the formalism of Section~\ref{sec:bosonic}
to the geometry of the root lattice of $\mathfrak{sl}_4$ (as opposed to $\mathfrak{sl}_3$). The rationale
for such a shift of perspective will be given elsewhere \cite{SchubIV}. We only remark in passing that since the picture itself
has been lifted from two dimensions to three, it is natural to also upgrade the root lattice from two to three dimensions.

\subsection{The tensor calculus revisited}
%The \texorpdfstring{$\mathfrak{sl}_4$}{sl4} setup}
Recall from Section~\ref{sec:tensor}
that the first step is an assignment of a vector space/a set of labels to each edge of our diagrams.
Compared to Section~\ref{sec:bosonic}, we will need more types of edges. The type of an (oriented)
edge is given by a subset $A$ of $\{0,1,2,3\}\cong \ZZ_4$, and the corresponding vector space denoted $V_A$;
in practice we shall only consider $|A|=1$, i.e., $V_\alpha$, $\alpha=0,\ldots,3$,
and $|A|=2$ with the specific choice $V_{\alpha\,\alpha+1}$,
$\alpha=0,\ldots,3$.

A basis of $V_A$ is labelled as follows: it is a collection of nonnegative integers $a_{\beta,\alpha}$ with 
$\beta\not\in A$
and $\alpha\in A$. For $|A|=1$ this means three labels, and for $|A|=2$, four labels.

As in Section~\ref{sec:bosonic}, when we draw the oriented lines corresponding to various vector spaces $V_A$,
we always give them the same direction to ease the interpretation of diagrams.
Our convention is that $V_0$ goes SouthWest, 
$V_1$ SouthEast, $V_2$ NorthEast, $V_3$ NorthWest, and $V_{01}$ South, $V_{12}$ East, $V_{23}$ North, $V_{30}$ West.
Redundantly, we also write $A$ next to the line carrying the space $V_A$.
One more convention is that we draw the lines of $V_A$, $|A|=2$, as double lines.

Finally, the allowed vertices come in the following types:
\begin{itemize}
\item Trivalent vertices that correspond to linear maps $V_{\alpha\,\beta}\to V_\beta\otimes V_\alpha$; they are the analogues of up-pointing triangles. We only use the following:
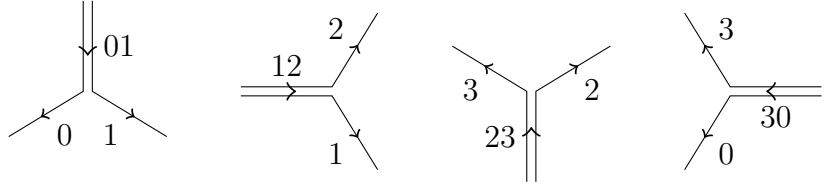
\begin{center}
\begin{tikzpicture}[baseline=-3pt,scale=-1.2]
\draw[invarrow=0.25] (150:1) -- node[below left] {$1$} (-0.05,0) -- node[right] {$01$} ++(0,-1);
\draw[invarrow=0.25] (30:1) -- node[below right] {$0$} (0.05,0) -- ++(0,-1);
\draw decorate [thin,decoration={markings,mark = at position 0.5 with {\arrow[scale=2]{<}}}] { (0,0) -- ++(0,-1) };
\end{tikzpicture}
\qquad
\begin{tikzpicture}[baseline=-3pt,rotate=270,scale=1.2]
\draw[invarrow=0.25] (150:1) -- node[above left] {$2$} (-0.05,0) -- node[above] {$12$} ++(0,-1);
\draw[invarrow=0.25] (30:1) -- node[below left] {$1$} (0.05,0) -- ++(0,-1);
\draw decorate [thin,decoration={markings,mark = at position 0.5 with {\arrow[scale=2]{<}}}] { (0,0) -- ++(0,-1) };
\end{tikzpicture}
\qquad
\begin{tikzpicture}[baseline=-3pt,scale=1.2]
\draw[invarrow=0.25] (150:1) -- node[below left] {$3$} (-0.05,0) -- node[left] {$23$} ++(0,-1);
\draw[invarrow=0.25] (30:1) -- node[below right] {$2$} (0.05,0) -- ++(0,-1);
\draw decorate [thin,decoration={markings,mark = at position 0.5 with {\arrow[scale=2]{<}}}] { (0,0) -- ++(0,-1) };
\end{tikzpicture}
\qquad
\begin{tikzpicture}[baseline=-3pt,rotate=90,scale=1.2]
\draw[invarrow=0.25] (150:1) -- node[below right] {$0$} (-0.05,0) -- node[below] {$30$} ++(0,-1);
\draw[invarrow=0.25] (30:1) -- node[above right] {$3$} (0.05,0) -- ++(0,-1);
\draw decorate [thin,decoration={markings,mark = at position 0.5 with {\arrow[scale=2]{<}}}] { (0,0) -- ++(0,-1) };
\end{tikzpicture}
\end{center}
\item Trivalent vertices that correspond to linear maps $V_\alpha\otimes V_\beta\to V_{\alpha\beta}$; 
they are the analogue of down-pointing triangles. Similarly, we use the following:
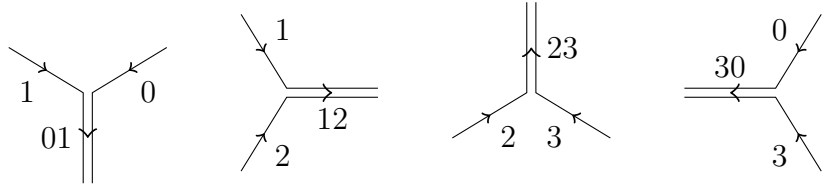
\begin{center}
\begin{tikzpicture}[baseline=-3pt,scale=1.2]
\draw[arrow=0.25] (150:1) -- node[below left] {$1$} (-0.05,0) -- node[left] {$01$} ++(0,-1);
\draw[arrow=0.25] (30:1) -- node[below right] {$0$} (0.05,0) -- ++(0,-1);
\draw decorate [thin,decoration={markings,mark = at position 0.5 with {\arrow[scale=2]{>}}}] { (0,0) -- ++(0,-1) };
\end{tikzpicture}
\qquad
\begin{tikzpicture}[baseline=-3pt,rotate=90,scale=1.2]
\draw[arrow=0.25] (150:1) -- node[below right] {$2$} (-0.05,0) -- node[below] {12} ++(0,-1);
\draw[arrow=0.25] (30:1) -- node[above right] {$1$} (0.05,0) -- ++(0,-1);
\draw decorate [thin,decoration={markings,mark = at position 0.5 with {\arrow[scale=2]{>}}}] { (0,0) -- ++(0,-1) };
\end{tikzpicture}
\qquad
\begin{tikzpicture}[scale=-1,baseline=-3pt,scale=1.2]
\draw[arrow=0.25] (150:1) -- node[below left] {$3$} (-0.05,0) -- node[right] {23} ++(0,-1);
\draw[arrow=0.25] (30:1) -- node[below right] {$2$} (0.05,0) -- ++(0,-1);
\draw decorate [thin,decoration={markings,mark = at position 0.5 with {\arrow[scale=2]{>}}}] { (0,0) -- ++(0,-1) };
\end{tikzpicture}
\qquad
\begin{tikzpicture}[baseline=-3pt,rotate=270,scale=1.2]
\draw[arrow=0.25] (150:1) -- node[above left] {$0$} (-0.05,0) -- node[above] {30} ++(0,-1);
\draw[arrow=0.25] (30:1) -- node[below left] {$3$} (0.05,0) -- ++(0,-1);
\draw decorate [thin,decoration={markings,mark = at position 0.5 with {\arrow[scale=2]{>}}}] { (0,0) -- ++(0,-1) };
\end{tikzpicture}
\end{center}
\item Elements in $V_A\otimes V_{\bar A}$, where $\bar A$ is the complement of $A$ in $\{0,1,2,3\}$, which we only
use for $|A|=|\bar A|=2$:
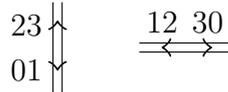
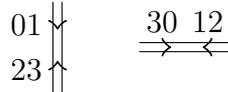
\begin{center}
\begin{tikzpicture}[baseline=-3pt,rotate=90,scale=1.2]
\draw (-0.5,-0.05) -- (0.5,-0.05) (-0.5,0.05) -- node[left,pos=0.25] {$01$} node[left,pos=0.75] {$23$} (0.5,0.05);
\draw decorate[thin,decoration={markings,mark = at position 0.35 with {\arrow[scale=2]{<}}}] { (-0.5,0) -- (0.5,0) };
\draw decorate[thin,decoration={markings,mark = at position 0.8 with {\arrow[scale=2]{>}}}] { (-0.5,0) -- (0.5,0) };
\end{tikzpicture}
\qquad
\begin{tikzpicture}[baseline=-3pt,scale=1.2]
\draw (0,-0.05) -- (1,-0.05) (0,0.05) -- node[above,pos=0.25] {$12$} node[above,pos=0.75] {$30$} (1,0.05);
\draw decorate[thin,decoration={markings,mark = at position 0.35 with {\arrow[scale=2]{<}}}] { (0,0) -- (1,0) };
\draw decorate[thin,decoration={markings,mark = at position 0.8 with {\arrow[scale=2]{>}}}] { (0,0) -- (1,0) };
\end{tikzpicture}
\end{center}
\item Their inverses $V_A\otimes V_{\bar A}\to\mathbb \CC$:
\begin{center}
\begin{tikzpicture}[baseline=-3pt,rotate=90,scale=1.2]
\draw (-0.5,-0.05) -- (0.5,-0.05) (-0.5,0.05) -- node[left,pos=0.25] {$23$} node[left,pos=0.75] {$01$} (0.5,0.05);
\draw decorate[thin,decoration={markings,mark = at position 0.35 with {\arrow[scale=2]{>}}}] { (-0.5,0) -- (0.5,0) };
\draw decorate[thin,decoration={markings,mark = at position 0.8 with {\arrow[scale=2]{<}}}] { (-0.5,0) -- (0.5,0) };
\end{tikzpicture}
\qquad
\begin{tikzpicture}[baseline=-3pt,scale=1.2]
\draw (0,-0.05) -- (1,-0.05) (0,0.05) -- node[above,pos=0.25] {$30$} node[above,pos=0.75] {$12$} (1,0.05);
\draw decorate[thin,decoration={markings,mark = at position 0.35 with {\arrow[scale=2]{>}}}] { (0,0) -- (1,0) };
\draw decorate[thin,decoration={markings,mark = at position 0.8 with {\arrow[scale=2]{<}}}] { (0,0) -- (1,0) };
\end{tikzpicture}
\end{center}
\end{itemize}

\begin{ex*}
The labelling around a vertex $V_{01}\to V_1\otimes V_0$ is given by
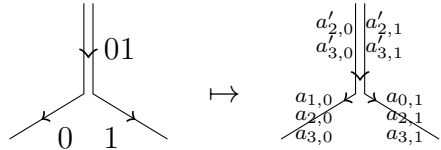
\[
\begin{tikzpicture}[baseline=-3pt,scale=-1.2]
\draw[invarrow=0.25] (150:1) -- node[below left] {$1$} (-0.05,0) -- node[right] {$01$} ++(0,-1);
\draw[invarrow=0.25] (30:1) -- node[below right] {$0$} (0.05,0) -- ++(0,-1);
\draw decorate [thin,decoration={markings,mark = at position 0.5 with {\arrow[scale=2]{<}}}] { (0,0) -- ++(0,-1) };
\end{tikzpicture}
\quad
\mapsto
\quad
\begin{tikzpicture}[baseline=-3pt,scale=-1.2]
% \draw[invarrow=0.45] (150:1) -- 
% node[above,pos=0.45,myred] {$\ss a_{0,1}$}
% node[pos=0.45,myblue] {$\ss a_{2,1}$}
% node[below,pos=0.45,mygreen] {$\ss a_{3,1}$}
% (-0.05,0) -- ++(0,-1);
% \draw[invarrow=0.45] (30:1) -- 
% node[above,pos=0.45,myblue] {$\ss a_{1,0}$} 
% node[pos=0.45,mygreen] {$\ss a_{2,0}$} 
% node[below,pos=0.45,myred] {$\ss a_{3,0}$} 
% (0.05,0) -- ++(0,-1);
% \path (0,0) -- node[above] {$\ss \color{mygreen}a'_{2,0}\ \color{myblue}a'_{2,1}$}
% node {$\ss \color{myred}a'_{3,0}\ \color{mygreen}a'_{3,1}$}  ++(0,-1);
% \draw decorate [thin,decoration={markings,mark = at position 0.25 with {\arrow[scale=2]{<}}}] { (0,0) -- ++(0,-1) };
\draw[invarrow=0.45] (150:1) -- 
node[above,pos=0.45] {$\ss a_{0,1}$}
node[pos=0.45] {$\ss a_{2,1}$}
node[below,pos=0.45] {$\ss a_{3,1}$}
(-0.05,0) -- ++(0,-1);
\draw[invarrow=0.45] (30:1) -- 
node[above,pos=0.45] {$\ss a_{1,0}$} 
node[pos=0.45] {$\ss a_{2,0}$} 
node[below,pos=0.45] {$\ss a_{3,0}$} 
(0.05,0) -- ++(0,-1);
\path (0,0) -- node[above] {$\ss a'_{2,0}\ a'_{2,1}$}
node {$\ss a'_{3,0}\ a'_{3,1}$}  ++(0,-1);
\draw decorate [thin,decoration={markings,mark = at position 0.25 with {\arrow[scale=2]{<}}}] { (0,0) -- ++(0,-1) };
\end{tikzpicture}
\]
Note that the four labels of the double line also appear on the other edges. This does not mean that they
are equal!
When there is a risk of confusion, we shall use primes or superscripts to distinguish identically named labels of different edges.
%To improve readability, we use again colors, $a_{i,j}$ being drawn in purple, cyan, brown, depending
%on $j-i=1,2,3\pmod4$. \rem{or not!}
\end{ex*}

The {\em weight}\/ of a label (or of its corresponding basis vector) is equal to $\sum a_{\alpha,\beta}(e_\alpha-e_\beta)$, where
$e_0,\ldots,e_3$ form a basis of $\RR^4$ (note that the $e_\alpha-e_\beta$ are nothing but the roots of $\mathfrak{sl}(4)$).
A major difference with the $\mathfrak{sl}_3$  setup is that
the weight of a label is not enough to reconstruct the label in the case $|A|=2$ (in other words, some
weight spaces have dimension greater than $1$).

We must now assign fugacities to these vertices. All our fugacities will be $\ZZ_4$-invariant, in the sense
that shifting all indices $\alpha\in \{0,1,2,3\}$ by $1\pmod4$ will leave them invariant. We also have weight
conservation: the fugacity will be zero unless the sum of weights of incoming edges is equal to the
sum of weights of outgoing edges.

The duality pairings are easy to define.
Note that labels of $V_A$ and of $V_{\bar A}$ correspond bijectively via $a_{\beta,\alpha}\leftrightarrow a_{\alpha,\beta}$,
and that their weights are negatives of each other. The rule is that these labels must match, and then the fugacity
is given by $\prod_{\alpha\in A,\beta\not\in A} \a_{a_{\alpha,\beta}}^\pm$ where the sign is $+$ (resp.\ $-$) for incoming (resp.\ outgoing) arrows:
\begin{align}\label{eq:defdual}
\begin{tikzpicture}[baseline=-3pt,scale=1.2]
\draw (0,-0.05) -- (1,-0.05) (0,0.05) -- node[above,pos=0.25] {12} node[above,pos=0.75] {30} (1,0.05);
\node at (-0.4,0.3) {$\ss a_{1,0}$};
\node at (-0.4,0.1) {$\ss a_{2,0}$};
\node at (-0.4,-0.1) {$\ss a_{1,3}$};
\node at (-0.4,-0.3) {$\ss a_{2,3}$};
\node at (1.4,0.3) {$\ss a_{0,1}$};
\node at (1.4,0.1) {$\ss a_{0,2}$};
\node at (1.4,-0.1) {$\ss a_{3,1}$};
\node at (1.4,-0.3) {$\ss a_{3,2}$};
\draw decorate[thin,decoration={markings,mark = at position 0.35 with {\arrow[scale=2]{>}},mark = at position 0.8 with {\arrow[scale=2]{<}}}] { (0,0) -- (1,0) };
% \draw (0,-0.05) -- (1,-0.05) (0,0.05) -- node[above,pos=0.25] {12} node[above,pos=0.75] {30} (1,0.05);
% \node[myblue] at (-0.4,0.3) {$\ss a_{1,0}$};
% \node[mygreen] at (-0.4,0.1) {$\ss a_{2,0}$};
% \node[mygreen] at (-0.4,-0.1) {$\ss a_{1,3}$};
% \node[myred] at (-0.4,-0.3) {$\ss a_{2,3}$};
% \node[myred] at (1.4,0.3) {$\ss a_{0,1}$};
% \node[mygreen] at (1.4,0.1) {$\ss a_{0,2}$};
% \node[mygreen] at (1.4,-0.1) {$\ss a_{3,1}$};
% \node[myblue] at (1.4,-0.3) {$\ss a_{3,2}$};
% \draw decorate[thin,decoration={markings,mark = at position 0.35 with {\arrow[scale=2]{>}}}] { (0,0) -- (1,0) };
% \draw decorate[thin,decoration={markings,mark = at position 0.8 with {\arrow[scale=2]{<}}}] { (0,0) -- (1,0) };
\end{tikzpicture}
&=\prod_{\alpha=3,0,\beta=1,2} \delta_{a_{\alpha,\beta},a_{\beta,\alpha}}\a_{a_{\alpha,\beta}}
\\\label{eq:defdualb}
\begin{tikzpicture}[baseline=-3pt,scale=1.2]
\draw (0,-0.05) -- (1,-0.05) (0,0.05) -- node[above,pos=0.25] {$30$} node[above,pos=0.75] {$12$} (1,0.05);
\draw decorate[thin,decoration={markings,mark = at position 0.35 with {\arrow[scale=2]{<}},mark = at position 0.8 with {\arrow[scale=2]{>}}}] { (0,0) -- (1,0) };
\node at (1.4,0.3) {$\ss a_{1,0}$};
\node at (1.4,0.1) {$\ss a_{2,0}$};
\node at (1.4,-0.1) {$\ss a_{1,3}$};
\node at (1.4,-0.3) {$\ss a_{2,3}$};
\node at (-0.4,0.3) {$\ss a_{0,1}$};
\node at (-0.4,0.1) {$\ss a_{0,2}$};
\node at (-0.4,-0.1) {$\ss a_{3,1}$};
\node at (-0.4,-0.3) {$\ss a_{3,2}$};
% \draw (0,-0.05) -- (1,-0.05) (0,0.05) -- node[above,pos=0.25] {$30$} node[above,pos=0.75] {$12$} (1,0.05);
% \draw decorate[thin,decoration={markings,mark = at position 0.35 with {\arrow[scale=2]{<}}}] { (0,0) -- (1,0) };
% \draw decorate[thin,decoration={markings,mark = at position 0.8 with {\arrow[scale=2]{>}}}] { (0,0) -- (1,0) };
% \node[myblue] at (1.4,0.3) {$\ss a_{1,0}$};
% \node[mygreen] at (1.4,0.1) {$\ss a_{2,0}$};
% \node[mygreen] at (1.4,-0.1) {$\ss a_{1,3}$};
% \node[myred] at (1.4,-0.3) {$\ss a_{2,3}$};
% \node[myred] at (-0.4,0.3) {$\ss a_{0,1}$};
% \node[mygreen] at (-0.4,0.1) {$\ss a_{0,2}$};
% \node[mygreen] at (-0.4,-0.1) {$\ss a_{3,1}$};
% \node[myblue] at (-0.4,-0.3) {$\ss a_{3,2}$};
\end{tikzpicture}
&=\prod_{\alpha=3,0,\beta=1,2} \delta_{a_{\alpha,\beta},a_{\beta,\alpha}}\a_{a_{\alpha,\beta}}^{-1}
\end{align}
and similarly for $01/23$.
In other words, 
via the identification $V_{\bar A}\cong \bar V_A$, our bases are dual of each other up to normalization.

The ``down-pointing'' maps $V_{\alpha}\otimes V_{\alpha+1}\to V_{\alpha\,\alpha+1}$ are equally simple. 
Once again labels naturally come in pairs whose contribution to the weight cancels, and the rule is that
these labels must match.
Let us for example take $\alpha=2$:
\begin{equation}\label{eq:defDD}
\begin{tikzpicture}[baseline=-3pt,scale=-1.2]
\draw[arrow=0.45] (150:1) -- 
node[below,pos=0.45] {$\ss a_{2,3}$}
node[above,pos=0.45] {$\ss a_{0,3}$}
node[pos=0.45] {$\ss a_{1,3}$}
(-0.05,0) -- ++(0,-1);
\draw[arrow=0.45] (30:1) -- 
node[below,pos=0.45] {$\ss a_{3,2}$} 
node[above,pos=0.45] {$\ss a_{0,2}$} 
node[pos=0.45] {$\ss a_{1,2}$} 
(0.05,0) -- ++(0,-1);
\path (0,0) -- node[above] {$\ss a'_{0,2}\ a'_{0,3}$}
node {$\ss a'_{1,2}\ a'_{1,3}$}  ++(0,-1);
\draw decorate [thin,decoration={markings,mark = at position 0.25 with {\arrow[scale=2]{>}}}] { (0,0) -- ++(0,-1) };
% \draw[arrow=0.45] (150:1) -- 
% node[below,pos=0.45,myred] {$\ss a_{2,3}$}
% node[above,pos=0.45,myblue] {$\ss a_{0,3}$}
% node[pos=0.45,mygreen] {$\ss a_{1,3}$}
% (-0.05,0) -- ++(0,-1);
% \draw[arrow=0.45] (30:1) -- 
% node[below,pos=0.45,myblue] {$\ss a_{3,2}$} 
% node[above,pos=0.45,mygreen] {$\ss a_{0,2}$} 
% node[pos=0.45,myred] {$\ss a_{1,2}$} 
% (0.05,0) -- ++(0,-1);
% \path (0,0) -- node[above] {$\ss \color{mygreen}a'_{0,2}\ \color{myblue}a'_{0,3}$}
% node {$\ss \color{myred}a'_{1,2}\ \color{mygreen}a'_{1,3}$}  ++(0,-1);
% \draw decorate [thin,decoration={markings,mark = at position 0.25 with {\arrow[scale=2]{>}}}] { (0,0) -- ++(0,-1) };
\end{tikzpicture}
=
\delta_{a_{0,2},a'_{0,2}}
\delta_{a_{1,2},a'_{1,2}}
\delta_{a_{0,3},a'_{0,3}}
\delta_{a_{1,3},a'_{1,3}}
\delta_{a_{2,3},a_{3,2}}
t^{a_{1,3}a_{0,2}/2}\a_{a_{2,3}}
\end{equation}
One pairs each primed label with its corresponding unprimed label, and $a_{2,3}$ with $a_{3,2}$.
The only nontrivial feature is a power of $t^{1/2}$, a formal variable squaring to $t$.
The definition is extended to other cases by $\ZZ_4$-symmetry.

Finally, the ``up-pointing'' maps $V_{\alpha\,\alpha+1}\to V_{\alpha+1}\otimes V_{\alpha}$ are defined as follows, again choosing
$\alpha=2$:
\begin{equation}\label{eq:defUU}
\begin{tikzpicture}[baseline=-3pt,scale=1.2]
\draw[invarrow=0.45] (150:1) -- 
node[above,pos=0.45] {$\ss a_{0,3}$}
node[pos=0.45] {$\ss a_{1,3}$}
node[below,pos=0.45] {$\ss a_{2,3}$}
(-0.05,0) -- ++(0,-1);
\draw[invarrow=0.45] (30:1) -- 
node[above,pos=0.45] {$\ss a_{0,2}$} 
node[pos=0.45] {$\ss a_{1,2}$} 
node[below,pos=0.45] {$\ss a_{3,2}$} 
(0.05,0) -- ++(0,-1);
\path (0,0) -- node {$\ss a'_{0,3}\ a'_{0,2}$}
node[below] {$\ss a'_{1,3}\ a'_{1,2}$}  ++(0,-1);
\draw decorate [thin,decoration={markings,mark = at position 0.25 with {\arrow[scale=2]{<}}}] { (0,0) -- ++(0,-1) };
% \draw[invarrow=0.45] (150:1) -- 
% node[above,myblue] {$\ss a_{0,3}$}
% node[mygreen] {$\ss a_{1,3}$}
% node[below,myred] {$\ss a_{2,3}$}
% (-0.05,0) -- ++(0,-1);
% \draw[invarrow=0.45] (30:1) -- 
% node[above,mygreen] {$\ss a_{0,2}$} 
% node[myred] {$\ss a_{1,2}$} 
% node[below,myblue] {$\ss a_{3,2}$} 
% (0.05,0) -- ++(0,-1);
% \path (0,0) -- node {$\ss \color{myblue}a'_{0,3}\ \color{mygreen}a'_{0,2}$}
% node[below] {$\ss \color{mygreen}a'_{1,3}\ \color{myred}a'_{1,2}$}  ++(0,-1);
% \draw decorate [thin,decoration={markings,mark = at position 0.25 with {\arrow[scale=2]{<}}}] { (0,0) -- ++(0,-1) };
\end{tikzpicture}
=
\begin{cases}
t^{a'_{1,3}a'_{0,2}/2}
\a_{a'_{0,2}}\a_{a'_{0,3}}\a_{a'_{1,2}}\a_{a'_{1,3}}\a_{b}
&\text{if weight is conserved}
\\
\qquad u^{a'_{0,2},a'_{0,3},a_{3,2},a_{0,2},a_{0,3},b}u^{a'_{1,3},a'_{1,2},a_{2,3},a_{1,3},a_{1,2},b}\hspace{-1cm}
\\
0&\text{else}
\end{cases}
\end{equation}
where $u^{\ldots}$ is the terminating basic hypergeometric series defined in \eqref{eq:defy},
and $b=a_{3,2}+a_{0,2}-a'_{0,2}=a_{2,3}+a_{1,3}-a'_{1,3}$ (the latter equality coming from weight conservation).
Weight conservation also implies that the arguments of $u^{\ldots}$ satisfy the balance condition.
%\rem{so the expression has a left-right symmetry}

Once again, one can interpret all these maps as intertwiners for certain $\mathcal U_{t^{1/2}}(\mathfrak{sl}_4)$ parabolic Verma modules (namely, $V_A$ has highest weight $s\sum_{\alpha\in A}e_\alpha$) in the limit $t^s\to 0$.

\subsection{Double puzzles}
We are ready to introduce the main actor of the proof of associativity, which we call
{\em double puzzles}. They are obtained by gluing to the bottom side of a puzzle, another puzzle
upside down. We define, as in Section~\ref{sec:3d}, two versions corresponding to either
side of the associativity equation \eqref{eq:assoc}:
\begin{align}\label{eq:defLHS}
\mathcal L
&=
\begin{tikzpicture}[yscale=0.75,xscale=1.05,baseline=0]
\path (0,0.5) ++(135:1) coordinate (a);
\draw[invarrow=0.167,arrow=0.833] (a) node[left=-1mm,align=center] {$\ss {\color{myred}a_{2,3}}=m_{0}(\lambda)$\\[-3mm]$\ss {\color{mygreen}a_{1,3}}={\color{myblue}a_{0,3}}=0$} -- node[above] {$3$} (-0.05,0.5) -- node[left=-1mm,pos=0.25] {$23$}
node[left=-1mm,pos=0.75] {$01$} (-0.05,-0.5) -- node[below] {$0$} ++(225:1) node[left=-1mm,align=center] {$\ss {\color{myred}a_{3,0}}=m_{0}(\rho)$\\[-3mm]$\ss {\color{myblue}a_{2,0}}={\color{mygreen}a_{1,0}}=0$};
\path (0,0.5) ++(45:1) coordinate (b);
\draw[invarrow=0.167,arrow=0.833] (b) -- node[above] {$2$} (0.05,0.5) -- (0.05,-0.5) -- node[below] {$1$} ++(-45:1);
\draw decorate [thin,decoration={markings,mark = at position 0.6 with {\arrow[scale=2]{<}}}] { (0,-0.5) -- (0,0) };
\draw decorate [thin,decoration={markings,mark = at position 0.6 with {\arrow[scale=2]{<}}}] { (0,0.5) -- (0,0) };
\begin{scope}[shift={(0:5)}]
\path (0,0.5) ++(135:1) coordinate (a);
\draw[invarrow=0.167,arrow=0.833] (a) -- node[above] {$3$} (-0.05,0.5) -- node[left=-1mm,pos=0.25] {$23$}
node[left=-1mm,pos=0.75] {$01$} (-0.05,-0.5) -- node[below] {$0$} ++(225:1);
\path (0,0.5) ++(45:1) coordinate (b);
\draw[invarrow=0.167,arrow=0.833] (b) node[right=-1mm,align=center] {$\ss {\color{myred}a_{1,2}}=m_{n-1}(\mu)$\\[-3mm]$\ss {\color{myblue}a_{0,2}}={\color{mygreen}a_{3,2}}=0$} -- node[above] {$2$} (0.05,0.5) -- (0.05,-0.5) -- node[below] {$1$} ++(-45:1) node[right=-1mm,align=center] {$\ss {\color{myred}a_{0,1}}=m_{0}(\nu)$\\[-3mm]$\ss {\color{mygreen}a_{3,1}}={\color{myblue}a_{2,1}}=0$};
\draw decorate [thin,decoration={markings,mark = at position 0.6 with {\arrow[scale=2]{<}}}] { (0,-0.5) -- (0,0) };
\draw decorate [thin,decoration={markings,mark = at position 0.6 with {\arrow[scale=2]{<}}}] { (0,0.5) -- (0,0) };
\end{scope}
\node[rotate=45] at (60:2.25) {$\cdots$};
\path (60:2.25) ++(0:2.75) node[rotate=-45] {$\cdots$};
\begin{scope}[shift={(60:5)}]
\draw[arrow=0.75] (0.05,-0.5) -- node[right=-1mm] {$23$} ++(0,1) -- node[above] {$2$} ++(45:1) node[right=-1mm,align=center] {$\ss {\color{myred}a_{1,2}}=m_{0}(\mu)$\\[-3mm]$\ss {\color{myblue}a_{0,2}}={\color{mygreen}a_{3,2}}=0$};
\draw[invarrow=0.15,arrow=0.9] (0.05,-0.5) -- node[below] {$3$} ++(-45:1) -- node[left=-1mm] {$23$} ++(270:1) ++(0.05,0) coordinate(a) ++(0.05,0) -- ++(90:1) -- node[above] {$2$} ++(45:1) node[right=-1mm,align=center] {$\ss {\color{myred}a_{1,2}}=m_{1}(\mu)$\\[-3mm]$\ss {\color{myblue}a_{0,2}}={\color{mygreen}a_{3,2}}=0$};
\draw[arrow=0.75] (-0.05,-0.5) -- ++(0,1) -- node[above] {$3$} ++(135:1) node[left=-1mm,align=center] {$\ss {\color{myred}a_{2,3}}=m_{n-1}(\lambda)$\\[-3mm]$\ss {\color{mygreen}a_{1,3}}={\color{myblue}a_{0,3}}=0$};
\draw[invarrow=0.15,arrow=0.9] (-0.05,-0.5) -- node[below] {$2$} ++(225:1) -- node[right=-1mm] {$23$} ++(270:1) ++(-0.05,0) coordinate(b) ++(-0.05,0) -- ++(90:1) -- node[above] {$3$} ++(135:1) node[left=-1mm,align=center] {$\ss {\color{myred}a_{2,3}}=m_{n-2}(\lambda)$\\[-3mm]$\ss {\color{mygreen}a_{1,3}}={\color{myblue}a_{0,3}}=0$};
\draw decorate [thin,decoration={markings,mark = at position 0.6 with {\arrow[scale=2]{>}}}] { (0,-0.5) -- (0,0.5) };
\draw decorate [thin,decoration={markings,mark = at position 0.6 with {\arrow[scale=2]{>}}}] { (a) -- ++(0,1) };
\draw decorate [thin,decoration={markings,mark = at position 0.6 with {\arrow[scale=2]{>}}}] { (b) -- ++(0,1) };
\end{scope}
\node at (0:2.5) {$\cdots$};
\node[rotate=-45] at (-60:2.25) {$\cdots$};
\path (-60:2.25) ++(0:2.75) node[rotate=45] {$\cdots$};
\begin{scope}[shift={(-60:5)},scale=-1]
\draw[arrow=0.75] (0.05,-0.5) -- node[left=-1mm] {$01$} ++(0,1) -- node[below] {$0$} ++(45:1) node[left=-1mm,align=center] {$\ss {\color{myred}a_{3,0}}=m_{n-1}(\rho)$\\[-3mm]$\ss {\color{myblue}a_{2,0}}={\color{mygreen}a_{1,0}}=0$};
\draw[invarrow=0.15,arrow=0.9] (0.05,-0.5) -- node[above] {$1$} ++(-45:1) -- node[right=-1mm] {$01$} ++(270:1) ++(0.05,0) coordinate (a) ++(0.05,0) -- ++(90:1) -- node[below] {$0$} ++(45:1) node[left=-1mm,align=center] {$\ss {\color{myred}a_{3,0}}=m_{n-2}(\rho)$\\[-3mm]$\ss {\color{myblue}a_{2,0}}={\color{mygreen}a_{1,0}}=0$};
\draw[arrow=0.75] (-0.05,-0.5) -- ++(0,1) -- node[below] {$1$} ++(135:1) node[right=-1mm,align=center] {$\ss {\color{myred}a_{0,1}}=m_{n-1}(\nu)$\\[-3mm]$\ss {\color{mygreen}a_{3,1}}={\color{myblue}a_{2,1}}=0$};
\draw[invarrow=0.15,arrow=0.9] (-0.05,-0.5) -- node[above] {$0$} ++(225:1) -- node[left=-1mm] {$01$} ++(270:1) ++(-0.05,0) coordinate (b) ++(-0.05,0) -- ++(90:1) -- node[below] {$1$} ++(135:1) node[right=-1mm,align=center] {$\ss {\color{myred}a_{0,1}}=m_{n-2}(\nu)$\\[-3mm]$\ss {\color{mygreen}a_{3,1}}={\color{myblue}a_{2,1}}=0$};
\draw decorate [thin,decoration={markings,mark = at position 0.6 with {\arrow[scale=2]{>}}}] { (0,-0.5) -- (0,0.5) };
\draw decorate [thin,decoration={markings,mark = at position 0.6 with {\arrow[scale=2]{>}}}] { (a) -- ++(0,1) };
\draw decorate [thin,decoration={markings,mark = at position 0.6 with {\arrow[scale=2]{>}}}] { (b) -- ++(0,1) };
\end{scope}
\end{tikzpicture}
\\\label{eq:defRHS}
\mathcal R
&=
\begin{tikzpicture}[baseline=2.5cm,rotate=90,yscale=0.75,xscale=1.05]
\path (0,0.5) ++(135:1) coordinate (a);
\draw[invarrow=0.167,arrow=0.833] (a) node[left=-1mm,align=center] {$\ss {\color{myred}a_{3,0}}=m_{n-1}(\rho)$\\[-3mm]$\ss {\color{mygreen}a_{2,0}}={\color{myblue}a_{1,0}}=0$} -- node[left] {$0$} (-0.05,0.5) -- node[above,pos=0.25] {$30$}
node[below=-1mm,pos=0.75] {$12$} (-0.05,-0.5) -- node[right] {$1$} ++(225:1) node[right=-1mm,align=center] {$\ss {\color{myred}a_{0,1}}=m_{n-1}(\nu)$\\[-3mm]$\ss {\color{myblue}a_{3,1}}={\color{mygreen}a_{2,1}}=0$};
\path (0,0.5) ++(45:1) coordinate (b);
\draw[invarrow=0.167,arrow=0.833] (b) -- node[left] {$3$} (0.05,0.5) -- (0.05,-0.5) -- node[right] {$2$} ++(-45:1);
\draw decorate [thin,decoration={markings,mark = at position 0.6 with {\arrow[scale=2]{<}}}] { (0,-0.5) -- (0,0) };
\draw decorate [thin,decoration={markings,mark = at position 0.6 with {\arrow[scale=2]{<}}}] { (0,0.5) -- (0,0) };
\begin{scope}[shift={(0:5)}]
\path (0,0.5) ++(135:1) coordinate (a);
\draw[invarrow=0.167,arrow=0.833] (a) -- node[left] {$0$} (-0.05,0.5) -- node[above,pos=0.25] {$30$}
node[below=-1mm,pos=0.75] {$12$} (-0.05,-0.5) -- node[right] {$1$} ++(225:1);
\path (0,0.5) ++(45:1) coordinate (b);
\draw[invarrow=0.167,arrow=0.833] (b) node[left=-1mm,align=center] {$\ss {\color{myred}a_{2,3}}=m_{n-1}(\lambda)$\\[-3mm]$\ss {\color{myblue}a_{1,3}}={\color{mygreen}a_{0,3}}=0$} -- node[left] {$3$} (0.05,0.5) -- (0.05,-0.5) -- node[right] {$2$} ++(-45:1) node[right=-1mm,align=center] {$\ss {\color{myred}a_{1,2}}=m_{0}(\mu)$\\[-3mm]$\ss {\color{mygreen}a_{0,2}}={\color{myblue}a_{3,2}}=0$};
\draw decorate [thin,decoration={markings,mark = at position 0.6 with {\arrow[scale=2]{<}}}] { (0,-0.5) -- (0,0) };
\draw decorate [thin,decoration={markings,mark = at position 0.6 with {\arrow[scale=2]{<}}}] { (0,0.5) -- (0,0) };
\end{scope}
\node[rotate=-45] at (60:2.25) {$\cdots$};
\path (60:2.25) ++(0:2.75) node[rotate=45] {$\cdots$};
\begin{scope}[shift={(60:5)}]
\draw[arrow=0.75] (0.05,-0.5) -- node[above=-1mm] {$30$} ++(0,1) -- node[left] {$3$} ++(45:1) node[left=-1mm,align=center] {$\ss {\color{myred}a_{2,3}}=m_{0}(\lambda)$\\[-3mm]$\ss {\color{myblue}a_{1,3}}={\color{mygreen}a_{0,3}}=0$};
\draw[invarrow=0.15,arrow=0.9] (0.05,-0.5) -- node[right] {$0$} ++(-45:1) -- node[above] {$30$} ++(270:1) ++(0.05,0) coordinate(a) ++(0.05,0) -- ++(90:1) -- node[left] {$3$} ++(45:1) node[left=-1mm,align=center] {$\ss {\color{myred}a_{2,3}}=m_{1}(\lambda)$\\[-3mm]$\ss {\color{myblue}a_{1,3}}={\color{mygreen}a_{0,3}}=0$};
\draw[arrow=0.75] (-0.05,-0.5) -- ++(0,1) -- node[left] {$0$} ++(135:1) node[left=-1mm,align=center] {$\ss {\color{myred}a_{3,0}}=m_{0}(\rho)$\\[-3mm]$\ss {\color{mygreen}a_{2,0}}={\color{myblue}a_{1,0}}=0$};
\draw[invarrow=0.15,arrow=0.9] (-0.05,-0.5) -- node[right] {$3$} ++(225:1) -- node[above=-1mm] {$30$} ++(270:1) ++(-0.05,0) coordinate(b) ++(-0.05,0) -- ++(90:1) -- node[right] {$0$} ++(135:1) node[left=-1mm,align=center] {$\ss {\color{myred}a_{3,0}}=m_{1}(\rho)$\\[-3mm]$\ss {\color{mygreen}a_{2,0}}={\color{myblue}a_{1,0}}=0$};
\draw decorate [thin,decoration={markings,mark = at position 0.6 with {\arrow[scale=2]{>}}}] { (0,-0.5) -- (0,0.5) };
\draw decorate [thin,decoration={markings,mark = at position 0.6 with {\arrow[scale=2]{>}}}] { (a) -- ++(0,1) };
\draw decorate [thin,decoration={markings,mark = at position 0.6 with {\arrow[scale=2]{>}}}] { (b) -- ++(0,1) };
\end{scope}
\node at (0:2.5) {$\vdots$};
\node[rotate=45] at (-60:2.25) {$\cdots$};
\path (-60:2.25) ++(0:2.75) node[rotate=-45] {$\cdots$};
\begin{scope}[shift={(-60:5)},scale=-1]
\draw[arrow=0.75] (0.05,-0.5) -- node[above] {$12$} ++(0,1) -- node[right] {$1$} ++(45:1) node[right=-1mm,align=center] {$\ss {\color{myred}a_{0,1}}=m_{0}(\nu)$\\[-3mm]$\ss {\color{myblue}a_{3,1}}={\color{mygreen}a_{2,1}}=0$};
\draw[invarrow=0.15,arrow=0.9] (0.05,-0.5) -- node[left] {$2$} ++(-45:1) -- node[above=-1mm] {$12$} ++(270:1) ++(0.05,0) coordinate (a) ++(0.05,0) -- ++(90:1) -- node[right] {$1$} ++(45:1) node[right=-1mm,align=center] {$\ss {\color{myred}a_{0,1}}=m_{1}(\nu)$\\[-3mm]$\ss {\color{myblue}a_{3,1}}={\color{mygreen}a_{2,1}}=0$};
\draw[arrow=0.75] (-0.05,-0.5) -- ++(0,1) -- node[right] {$2$} ++(135:1) node[right=-1mm,align=center] {$\ss {\color{myred}a_{1,2}}=m_{n-1}(\mu)$\\[-3mm]$\ss {\color{mygreen}a_{0,2}}={\color{myblue}a_{3,2}}=0$};
\draw[invarrow=0.15,arrow=0.9] (-0.05,-0.5) -- node[left] {$1$} ++(225:1) -- node[above] {$12$} ++(270:1) ++(-0.05,0) coordinate (b) ++(-0.05,0) -- ++(90:1) -- node[right] {$2$} ++(135:1) node[right=-1mm,align=center] {$\ss {\color{myred}a_{1,2}}=m_{n-2}(\mu)$\\[-3mm]$\ss {\color{mygreen}a_{0,2}}={\color{myblue}a_{3,2}}=0$};
\draw decorate [thin,decoration={markings,mark = at position 0.6 with {\arrow[scale=2]{>}}}] { (0,-0.5) -- (0,0.5) };
\draw decorate [thin,decoration={markings,mark = at position 0.6 with {\arrow[scale=2]{>}}}] { (a) -- ++(0,1) };
\draw decorate [thin,decoration={markings,mark = at position 0.6 with {\arrow[scale=2]{>}}}] { (b) -- ++(0,1) };
\end{scope}
\end{tikzpicture}
\end{align}
We have also colored the labels to ease identification with the results of Section~\ref{sec:bosonic}.

\begin{prop}\label{prop:doublepuzzle}
One has:
\begin{align*}
\mathcal L&=h_\rho\sum_{\sigma\in \P_{k,n}} c^{\lambda,\mu}_\sigma c^{\sigma,\nu}_{\rho}
\\
\mathcal R&=h_\rho\sum_{\tau\in \P_{k,n}} c^{\lambda,\tau}_{\rho}c^{\mu,\nu}_\tau 
\end{align*}
\end{prop}
For more explicit diagrams in size $n=3$, see Appendix~\ref{app:assoc}.
\begin{proof}
We first analyze each half of the double puzzle $\mathcal L$. They have exactly the same structure
as the puzzle \eqref{eq:c} in Lemma~\ref{lem:c}, except that the labeling of the vertices is in principle more complicated, 
and the fugacities are not obviously the same. We proceed in steps.

We consider the top half of $\mathcal L$. Note that every label $a_{\alpha,\beta}$ on its NorthWest and
NorthEast boundaries for which $\alpha=0$ are zero ($\beta=0$ never arises). Therefore, applying weight conservation
at every vertex on these boundaries, we conclude that the labels on the other side of these vertices satisfy
the same property. By induction this is true throughout the top half of $\mathcal L$.

Denote $i=a_{1,2}$, $j=a_{3,2}$ for edges of type $2$,
$i'=a_{2,3}$, $j'=a_{1,3}$ for edges of type $3$, and $i''=a_{1,3}$, $j''=a_{1,2}$ for edges of type $23$.
This way the labelling becomes identical to the one of \eqref{eq:vertices}, and one easily checks that
the $\mathfrak{sl}_4$ weight conservation reduces to the $\mathfrak{sl}_3$ weight conservation, itself
equivalent to the balance condition of honeycombs.

Next we compare fugacities: this amounts to setting all labels involving the index $0$ to zero in the
definitions \eqref{eq:defUU} and \eqref{eq:defDD}, as well as the correspondence of notations
of the previous paragraph. One easily checks that they indeed reduce to the definitions \eqref{eq:defU} and
\eqref{eq:defD}, using $u^{0,0,a_{3,2},0,0,a_{3,2}}=\a_{a_{3,2}}^{-1}$ and Lemma~\ref{lem:D6}.

Now let us analyze what happens at the bottom of that top half. There is a series of $n$ edges of type $23$,
whose nonzero labels are $a_{1,2}$, $a_{1,3}$, which we denote
$j''_r$ and $i''_r$ respectively, $r=0,\ldots,n-1$.
By summing weight conservation at every vertex of the top half, we obtain the ``global'' conservation equation
\[
\sum_{r=0}^{n-1}m_r(\lambda)(e_2-e_3)+\sum_{r=0}^{n-1}m_r(\mu)(e_1-e_2)=
\sum_{r=0}^{n-1}(j''_r(e_1-e_2)+i''_r(e_1-e_3))
\]
Recalling that all our partitions satisfy $\sum_{r=0}^{n-1}m_r(\lambda)=k$, we 
have
\[
k(e_1-e_3)=\sum_{r=0}^{n-1}(j''_r(e_1-e_2)+i''_r(e_1-e_3))
\]
from which we immediately derive $j''_r=0$ for all $r$, as well as $\sum_{r=0}^{n-1} i''_r=k$. Therefore, we can
write $i''_r=m_r(\sigma)$ for some partition $\sigma$ in $\P_{k,n}$.

Comparing with \eqref{eq:c}, we conclude that at fixed labels $i''_r$ at the bottom, the top half
of $\mathcal L$ reproduces exactly the tensor entry of Lemma~\ref{lem:c} and is therefore equal to
$c^{\lambda,\mu}_\sigma$.

The exact same reasoning can be repeated for the bottom half, noting that it is obtained from the top half
by the following procedure: rotate 180 degrees, increase all indices by $2$ mod $4$, replace $\lambda$
with $\nu$, $\mu$ with $\rho^*$ and $\sigma$ with some as yet unknown other partition
$\bar\sigma$, defined by $\bar i''_r=m_r(\bar\sigma)$, $r=0,\ldots,n-1$,
where the $\bar i''_r$  are the $a_{3,1}$ labels at the top of the bottom half,
numbered from right to left. 
Therefore, the bottom half of $\mathcal L$ contributes $c^{\nu,\rho^*}_{\bar\sigma}$.

Finally, we need to perform the summation over the $i''_r$, i.e., over $\sigma\in \P_{k,n}$.
According to \eqref{eq:defdualb}, the contribution is nonzero only if $i''_r=\bar i''_{n-1-r}$, so that $\bar\sigma=\sigma^*$,
and equal to $\prod_{r=0}^{n-1} \a_{m_r(\sigma)}^{-1}=h_\sigma^{-1}$, cf \eqref{eq:defh}.

We conclude that
\begin{align*}
\mathcal L&=\sum_{\sigma\in \P_{k,n}} c^{\lambda,\mu}_\sigma h_\sigma^{-1} c^{\nu,\rho^*}_{\sigma^*}
\\
&=\sum_{\sigma\in \P_{k,n}} c^{\lambda,\mu}_\sigma c^{\nu,\rho^*,\sigma} 
&&\text{by Prop.~\ref{prop:Z3} (a)}
\\
&=\sum_{\sigma\in \P_{k,n}} c^{\lambda,\mu}_\sigma c^{\sigma,\nu,\rho^*} 
&&\text{by Prop.~\ref{prop:Z3} (b)}
\\
&=h_\rho\sum_{\sigma\in \P_{k,n}} c^{\lambda,\mu}_\sigma c^{\sigma,\nu}_{\rho}
&&\text{by Prop.~\ref{prop:Z3} (a)}
\end{align*}

We proceed identically with $\mathcal R$. In fact, $\mathcal R$ is obtained from $\mathcal L$ 
by increasing the numbering
of all spaces by $1$ mod $4$ and by shifting cyclically all labels by 90 degrees (and conventionally
rotating the diagram back 90 degrees), so we obtain in exactly the same way
\begin{align*}
\mathcal R&=\sum_{\tau\in \P_{k,n}} c^{\rho^*,\lambda}_{\tau^*} h_\tau^{-1} c^{\mu,\nu}_\tau 
\\
&=h_\rho\sum_{\tau\in \P_{k,n}} c^{\lambda,\tau}_{\rho} c^{\mu,\nu}_\tau &&\text{by $3\times$ Prop.~\ref{prop:Z3}}
\end{align*}
\end{proof}

In conclusion, in order to prove Prop.~\ref{prop:assoc} (associativity), all we need is to go from
\eqref{eq:defLHS} to \eqref{eq:defRHS} by means of three identities
corresponding in the dual picture to the tetrahedron, octahedron and dual tetrahedron moves of Section~\ref{sec:3d}.
We prove such identities now, in increasing order of complexity.
All our proofs have in common with that of Prop.~\ref{prop:doublepuzzle} that l.h.s.\ and r.h.s.\ are related
by the $\ZZ_4$ action generated by 
90 degree rotation and shifting the numbering of all spaces and labels by $1$.
This implies that we only
need to analyze say the l.h.s.\ and prove its $\ZZ_4$-invariance.

\subsection{The dual tetrahedron identity}
\begin{prop}\label{prop:dualtetra}
The following identity holds in $\bar V_0\otimes \bar V_1\otimes \bar V_2\otimes \bar V_3$:
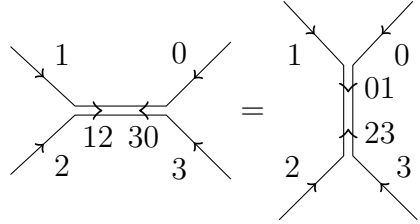
\[
\begin{tikzpicture}[baseline=-3pt,rotate=90,scale=1.2]
\path (0,0.5) ++(135:1) coordinate (a);
\draw[arrow=0.167,invarrow=0.833] (a) -- node[below right] {$2$} (-0.05,0.5) -- node[below,pos=0.25] {$12$}
node[below,pos=0.75] {$30$} (-0.05,-0.5) -- node[below left] {$3$} ++(225:1);
\path (0,0.5) ++(45:1) coordinate (b);
\draw[arrow=0.167,invarrow=0.833] (b) -- node[above right] {$1$} (0.05,0.5) -- (0.05,-0.5) -- node[above left] {$0$} ++(-45:1);
\draw decorate [thin,decoration={markings,mark = at position 0.6 with {\arrow[scale=2]{>}}}] { (0,-0.5) -- (0,0) };
\draw decorate [thin,decoration={markings,mark = at position 0.6 with {\arrow[scale=2]{>}}}] { (0,0.5) -- (0,0) };
\end{tikzpicture}
=
\begin{tikzpicture}[baseline=-3pt,scale=1.2]
\path (0,0.5) ++(135:1) coordinate (a);
\draw[arrow=0.167,invarrow=0.833] (a) -- node[below left] {$1$} (-0.05,0.5) -- (-0.05,-0.5) -- node[above left] {$2$} ++(225:1);
\path (0,0.5) ++(45:1) coordinate (b);
\draw[arrow=0.167,invarrow=0.833] (b) -- node[below right] {$0$} (0.05,0.5) -- node[right,pos=0.25] {$01$}
node[right,pos=0.75] {$23$} (0.05,-0.5) -- node[above right] {$3$} ++(-45:1);
\draw decorate [thin,decoration={markings,mark = at position 0.6 with {\arrow[scale=2]{>}}}] { (0,-0.5) -- (0,0) };
\draw decorate [thin,decoration={markings,mark = at position 0.6 with {\arrow[scale=2]{>}}}] { (0,0.5) -- (0,0) };
\end{tikzpicture}
\]
\end{prop}
\begin{proof}
We write out an entry of the l.h.s.\ explicitly:
\[
\begin{tikzpicture}[baseline=-3pt,rotate=90,scale=1.2]
\path (0,0.5) ++(135:1) coordinate (a);
\draw[arrow=0.3,invarrow=0.7] (a) -- node[pos=0.45,above] {$\ss a_{0,2}$} node[pos=0.45] {$\ss a_{1,2}$} node[pos=0.45,below] {$\ss a_{3,2}$} (-0.05,0.5) -- node[pos=0.25,below=-0.5mm] {$\ss a'_{02}$} node[pos=0.25,below=2.5mm] {$\ss a'_{32}$}
node[below=-0.5mm,pos=0.75] {$\ss a'_{13}$} node[below=2.5mm,pos=0.75] {$\ss a'_{23}$} (-0.05,-0.5) -- node[pos=0.65,above] {$\ss a_{0,3}$} node[pos=0.65] {$\ss a_{1,3}$} node[pos=0.65,below] {$\ss a_{2,3}$} ++(225:1);
\path (0,0.5) ++(45:1) coordinate (b);
\draw[arrow=0.3,invarrow=0.7] (b) -- node[pos=0.45,above] {$\ss a_{0,1}$} node[pos=0.45] {$\ss a_{2,1}$} node[pos=0.45,below] {$\ss a_{3,1}$} (0.05,0.5) -- node[pos=0.25,above=2.5mm] {$\ss a'_{01}$} node[pos=0.25,above=-0.5mm] {$\ss a'_{31}$} node[pos=0.75,above=2.5mm] {$\ss a'_{10}$} node[pos=0.75,above=-0.5mm] {$\ss a'_{20}$} (0.05,-0.5) -- node[above,pos=0.65] {$\ss a_{1,0}$} node[pos=0.65] {$\ss a_{2,0}$} node[pos=0.65,below] {$\ss a_{3,0}$} ++(-45:1);
\draw decorate [thin,decoration={markings,mark = at position 0.6 with {\arrow[scale=2]{>}}}] { (0,-0.5) -- (0,0) };
\draw decorate [thin,decoration={markings,mark = at position 0.6 with {\arrow[scale=2]{>}}}] { (0,0.5) -- (0,0) };
\end{tikzpicture}
\]
According to \eqref{eq:defdual} and \eqref{eq:defDD}, this entry is nonzero only if
\begin{align*}
a_{0,1}&=a'_{0,1}=a'_{1,0}=a_{1,0}
\\
a_{3,1}&=a'_{3,1}=a'_{1,3}=a_{1,3}
\\
a_{0,2}&=a'_{0,2}=a'_{2,0}=a_{2,0}
\\
a_{3,2}&=a'_{3,2}=a'_{2,3}=a_{2,3}
\\
a_{2,1}&=a_{1,2}
\\
a_{3,0}&=a_{0,3}
\end{align*}
in which case it is equal to
\[
\a_{a_{0,1}}\a_{a_{0,2}}\a_{a_{1,2}}\a_{a_{0,3}}\a_{a_{1,3}}\a_{a_{2,3}}
\]
The interpretation is obvious: the 12 external labels come in pairs of opposite weight, and they should be made equal, 
in which case the fugacity is the product over each pair $a_{\alpha,\beta}=a_{\beta,\alpha}$ of $\a_{a_{\alpha,\beta}}$. In particular this expression
is manifestly $\ZZ_4$-invariant.
\end{proof}

\subsection{The octahedron identity}
\begin{prop}\label{prop:octa}
The following identity holds in $V_{30}\otimes V_{23}\otimes V_{12}\otimes V_{01}$:
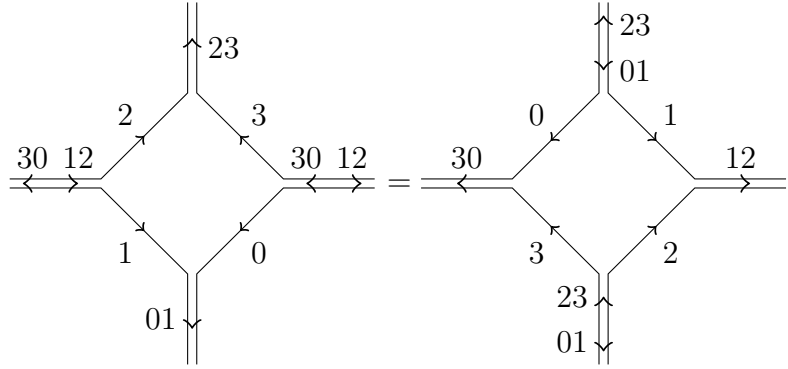
\[
\begin{tikzpicture}[baseline=-3pt,scale=1.2]
\draw[arrow=0.5] (-2,-0.05) -- (-1,-0.05) -- node[below left] {$1$} (-0.05,-1) -- node[left] {$01$} (-0.05,-2);
\draw[arrow=0.5] (2,0.05) -- node[above,pos=0.25] {$12$} node[above,pos=0.75] {$30$} (1,0.05) -- node[above right] {$3$} (0.05,1) -- node[right] {$23$} (0.05,2);
\draw[arrow=0.5] (-2,0.05) -- node[above,pos=0.25] {$30$} node[above,pos=0.75] {$12$} (-1,0.05) -- node[above left] {$2$} (-0.05,1) -- (-0.05,2);
\draw[arrow=0.5] (2,-0.05) -- (1,-0.05) -- node[below right] {$0$} (0.05,-1) -- (0.05,-2);
\draw decorate [thin,decoration={markings,mark = at position 0.25 with {\arrow[scale=2]{<}},mark = at position 0.75 with {\arrow[scale=2]{>}}}] { (-2,0) -- (-1,0) };
\draw decorate [thin,decoration={markings,mark = at position 0.25 with {\arrow[scale=2]{<}},mark = at position 0.75 with {\arrow[scale=2]{>}}}] { (2,0) -- (1,0) };
\draw decorate [thin,decoration={markings,mark = at position 0.5 with {\arrow[scale=2]{<}}}] { (0,-2) -- (0,-1) };
\draw decorate [thin,decoration={markings,mark = at position 0.5 with {\arrow[scale=2]{<}}}] { (0,2) -- (0,1) };
\end{tikzpicture}
=
\begin{tikzpicture}[baseline=-3pt,scale=1.2]
\draw[invarrow=0.5] (-2,-0.05) --(-1,-0.05) -- node[below left] {$3$} (-0.05,-1) -- node[left,pos=0.25] {$23$} node[left,pos=0.75] {$01$} (-0.05,-2);
\draw[invarrow=0.5] (2,0.05) -- node[above] {$12$} (1,0.05) -- node[above right] {$1$} (0.05,1) -- node[right,pos=0.25] {$01$} node[right,pos=0.75] {$23$} (0.05,2);
\draw[invarrow=0.5] (-2,0.05) -- node[above] {$30$} (-1,0.05) -- node[above left] {$0$} (-0.05,1) -- (-0.05,2);
\draw[invarrow=0.5] (2,-0.05) -- (1,-0.05) -- node[below right] {$2$} (0.05,-1) -- (0.05,-2);
\draw decorate [thin,decoration={markings,mark = at position 0.5 with {\arrow[scale=2]{<}}}] { (-2,0) -- (-1,0) };
\draw decorate [thin,decoration={markings,mark = at position 0.5 with {\arrow[scale=2]{<}}}] { (2,0) -- (1,0) };
\draw decorate [thin,decoration={markings,mark = at position 0.25 with {\arrow[scale=2]{<}},mark = at position 0.75 with {\arrow[scale=2]{>}}}] { (0,-2) -- (0,-1) };
\draw decorate [thin,decoration={markings,mark = at position 0.25 with {\arrow[scale=2]{<}},mark = at position 0.75 with {\arrow[scale=2]{>}}}] { (0,2) -- (0,1) };
\end{tikzpicture}
\]
\end{prop}
\begin{proof}
Again, we look at the l.h.s.:
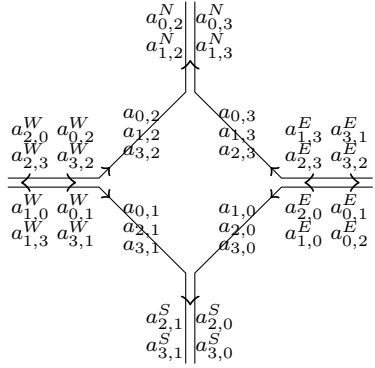
\[
\begin{tikzpicture}[baseline=-3pt,scale=1.2]
\draw[arrow=0.35] (-2,-0.05) -- node[pos=0.25,below=-0.5mm] {$\ss a^W_{1,0}$} node[pos=0.25,below=3mm] {$\ss a^W_{1,3}$} node[pos=0.75,below=-0.5mm] {$\ss a^W_{0,1}$} node[pos=0.75,below=3mm] {$\ss a^W_{3,1}$} (-1,-0.05) -- node[above] {$\ss a_{0,1}$} node {$\ss a_{2,1}$} node[below] {$\ss a_{3,1}$} (-0.05,-1) -- (-0.05,-2);
\draw[arrow=0.35] (2,0.05) -- node[pos=0.25,above=-0.5mm] {$\ss a^E_{3,2}$} node[pos=0.25,above=3mm] {$\ss a^E_{3,1}$} node[pos=0.75,above=-0.5mm] {$\ss a^E_{2,3}$} node[pos=0.75,above=3mm] {$\ss a^E_{1,3}$} (1,0.05) -- node[above] {$\ss a_{0,3}$} node {$\ss a_{1,3}$} node[below] {$\ss a_{2,3}$} (0.05,1) -- (0.05,2);
\draw[arrow=0.35] (-2,0.05) --  node[pos=0.25,above=-0.5mm] {$\ss a^W_{2,3}$} node[pos=0.25,above=3mm] {$\ss a^W_{2,0}$} node[pos=0.75,above=-0.5mm] {$\ss a^W_{3,2}$} node[pos=0.75,above=3mm] {$\ss a^W_{0,2}$} (-1,0.05) -- node[above] {$\ss a_{0,2}$} node {$\ss a_{1,2}$} node[below] {$\ss a_{3,2}$} (-0.05,1) -- (-0.05,2);
\draw[arrow=0.35] (2,-0.05) -- node[pos=0.25,below=-0.5mm] {$\ss a^E_{0,1}$} node[pos=0.25,below=3mm] {$\ss a^E_{0,2}$} node[pos=0.75,below=-0.5mm] {$\ss a^E_{2,0}$} node[pos=0.75,below=3mm] {$\ss a^E_{1,0}$} (1,-0.05) -- node[above] {$\ss a_{1,0}$} node {$\ss a_{2,0}$} node[below] {$\ss a_{3,0}$} (0.05,-1) -- (0.05,-2);
\draw decorate [thin,decoration={markings,mark = at position 0.25 with {\arrow[scale=2]{<}},mark = at position 0.75 with {\arrow[scale=2]{>}}}] { (-2,0) -- (-1,0) };
\draw decorate [thin,decoration={markings,mark = at position 0.25 with {\arrow[scale=2]{<}},mark = at position 0.75 with {\arrow[scale=2]{>}}}] { (2,0) -- (1,0) };
\path decorate [thin,decoration={markings,mark = at position 0.75 with {\arrow[scale=2]{<}}}] { (0,-2) -- node {$\ss a^S_{2,1}\ a^S_{2,0}$} node[below=0.5mm] {$\ss a^S_{3,1}\ a^S_{3,0}$} (0,-1) };
\draw decorate [thin,decoration={markings,mark = at position 0.75 with {\arrow[scale=2]{<}}}] { (0,2) -- node {$\ss a^N_{1,2}\ a^N_{1,3}$} node[above=0.5mm] {$\ss a^N_{0,2}\ a^N_{0,3}$} (0,1) };
\end{tikzpicture}
\]
In view of \eqref{eq:defdual} and \eqref{eq:defDD} (the latter being relevant to
South and North vertices), we have the following equalities
\begin{align*}
a_{2,0}^W&=a_{0,2}^W
&
a_{2,3}^W&=a_{3,2}^W
&
a_{1,0}^W&=a_{0,1}^W
&
a_{1,3}^W&=a_{3,1}^W
\\
a_{2,0}^E&=a_{0,2}^E
&
a_{2,3}^E&=a_{3,2}^E
&
a_{1,0}^E&=a_{0,1}^E
&
a_{1,3}^E&=a_{3,1}^E
\\
a_{0,2}^N&=a_{0,2}
&
a_{1,2}^N&=a_{1,2}
&
a_{0,3}^N&=a_{0,3}
&
a_{1,3}^N&=a_{1,3}
\\
a_{3,1}^S&=a_{3,1}
&
a_{2,0}^S&=a_{2,0}
&
a_{2,1}^S&=a_{2,1}
&
a_{3,0}^S&=a_{3,0}
\\
&&
a_{3,2}&=a_{2,3}
&
a_{0,1}&=a_{1,0}
\end{align*}

Weight conservation at the West and East vertices gives six more equalities:
\begin{align*}
a_{0,2}^W + a_{0,1}^W &= a_{0,2} + a_{01}\\
-a_{0,1}^W - a_{3,1}^W &=- a_{0,1} - a_{2,1} - a_{3,1} + a_{1,2}\\
-a_{0,2}^W - a_{3,2}^W &= -a_{0,2} - a_{1,2} - a_{3,2} + a_{2,1}\\
-a_{1,0}^E - a_{2,0}^E &= -a_{1,0} - a_{2,0} - a_{3,0} + a_{0,3}\\
a_{1,3}^E + a_{1,0}^E &= a_{1,3} + a_{1,0}\\
a_{2,3}^E + a_{2,0}^E &= a_{2,0} + a_{2,3}
\end{align*}
Solving them gives four constraints for the external labels:
\begin{align*}
a_{0,2}^E+a_{0,3}^N+a_{1,3}^N&=a_{3,1}^E+a_{2,0}^S+a_{3,0}^S
\\
a_{1,3}^N+a_{1,0}^W+a_{2,0}^W&=a_{0,2}^N+a_{3,1}^E+a_{0,1}^E
\\
a_{2,0}^W+a_{2,1}^S+a_{3,1}^S&=a_{1,3}^W+a_{0,2}^N+a_{1,2}^N
\\
  a_{3,1}^S+a_{3,2}^E+a_{0,2}^E&=a_{2,0}^S+a_{1,3}^W+a_{2,3}^W
\end{align*}
which are manifestly $\ZZ_4$-invariant,
as well as fixes all the internal ones; the nontrivial ones are
\begin{align*}
a_{0,1}&=a_{1,0}=a_{1,0}^W+a_{2,0}^W-a_{0,2}^N
\\
a_{2,3}&=a_{3,2}=a_{3,2}^E+a_{0,2}^E-a_{2,0}^S
\end{align*}

Finally, the associated nonzero entry is
\begin{gather*}
(\a_{a_{2,0}^W}\a_{a_{2,3}^W}\a_{a_{1,0}^W}\a_{a_{1,3}^W}\a_{a_{3,1}^E}\a_{a_{3,2}^E}\a_{a_{0,1}^E}\a_{a_{0,2}^E}
)^{-1}
\\
t^{(a_{0,2}^N a_{1,3}^N+a_{2,0}^W a_{1,3}^W+a_{2,0}^S a_{3,1}^S+a_{3,1}^E a_{0,2}^E)/2}
\a_{a_{0,1}}\a_{a_{2,3}}
\\
\a_{a_{2,0}^W}\a_{a_{2,3}^W}\a_{a_{1,0}^W}\a_{a_{1,3}^W}
\a_{a_{2,1}^S+a_{3,1}^S-a_{1,3}^W}
\\
\a_{a_{3,1}^E}\a_{a_{3,2}^E}\a_{a_{0,1}^E}\a_{a_{0,2}^E}
\a_{a_{0,3}^N+a_{1,3}^N-a_{3,1}^E}
\\
u^{a_{3,1}^W,a_{3,2}^W,a_{2,1}^S,a_{3,1}^S,a_{3,2},a_{2,1}^S+a_{3,1}^S-a_{3,1}^W}
u^{a_{0,2}^W,a_{0,1}^W,a_{1,2}^N,a_{0,2}^N,a_{0,1},a_{2,1}^S+a_{3,1}^S-a_{3,1}^W}
\\
u^{a_{1,3}^E,a_{1,0}^E,a_{0,3}^N,a_{1,3}^N,a_{1,0},a_{0,3}^N+a_{1,3}^N-a_{1,3}^E}
u^{a_{2,0}^E,a_{2,3}^E,a_{3,0}^S,a_{2,0}^S,a_{2,3},a_{0,3}^N+a_{1,3}^N-a_{1,3}^E}
\end{gather*}
After simplifying, substituting $a_{0,1}=a_{1,0}$ and $a_{2,3}=a_{3,2}$ with their values, and applying
Lemma~\ref{lem:D6}, we obtain the $\ZZ_4$-invariant expression
\begin{gather*}
t^{(a_{0,2}^N a_{1,3}^N+a_{1,3}^W a_{2,0}^W+a_{2,0}^S a_{3,1}^S+a_{3,1}^E a_{0,2}^E)/2}
\\
\a_{a_{0,3}^N+a_{1,3}^N-a_{3,1}^E}
\a_{a_{1,0}^W+a_{2,0}^W-a_{0,2}^N}
\a_{a_{2,1}^S+a_{3,1}^S-a_{1,3}^W}
\a_{a_{3,2}^E+a_{0,2}^E-a_{2,0}^S}
\\
u^{a_{1,3}^W,a_{2,3}^W,a_{2,1}^S,a_{3,1}^S,a_{3,2}^E+a_{0,2}^E-a_{2,0}^S,a_{2,1}^S+a_{3,1}^S-a_{1,3}^W}
\\
u^{a_{2,0}^S,a_{3,0}^S,a_{3,2}^E,a_{0,2}^E,a_{0,3}^N+a_{1,3}^N-a_{3,1}^E,a_{3,2}^E+a_{0,2}^E-a_{2,0}^S}
\\
u^{a_{3,1}^E,a_{0,1}^E,a_{0,3}^N,a_{1,3}^N,a_{1,0}^W+a_{2,0}^W-a_{0,2}^N,a_{0,3}^N+a_{1,3}^N-a_{3,1}^E}
\\
u^{a_{0,2}^N,a_{1,2}^N,a_{1,0}^W,a_{2,0}^W,a_{2,1}^S+a_{3,1}^S-a_{1,3}^W,a_{1,0}^W+a_{2,0}^W-a_{0,2}^N}
\end{gather*}
which concludes the proof.
\end{proof}

\subsection{The tetrahedron identity}
\begin{prop}\label{prop:tetra}
The following identity holds in $V_3\otimes V_2\otimes V_1\otimes V_0$:
\[
\begin{tikzpicture}[baseline=-3pt,scale=1.2]
\path (0,0.5) ++(135:1) coordinate (a);
\draw[invarrow=0.167,arrow=0.833] (a) -- node[below left] {$3$} (-0.05,0.5) -- node[left,pos=0.25] {$23$}
node[left,pos=0.75] {$01$} (-0.05,-0.5) -- node[above left] {$0$} ++(225:1);
\path (0,0.5) ++(45:1) coordinate (b);
\draw[invarrow=0.167,arrow=0.833] (b) -- node[below right] {$2$} (0.05,0.5) -- (0.05,-0.5) -- node[above right] {$1$} ++(-45:1);
\draw decorate [thin,decoration={markings,mark = at position 0.6 with {\arrow[scale=2]{<}}}] { (0,-0.5) -- (0,0) };
\draw decorate [thin,decoration={markings,mark = at position 0.6 with {\arrow[scale=2]{<}}}] { (0,0.5) -- (0,0) };
\end{tikzpicture}
=
\begin{tikzpicture}[baseline=-3pt,rotate=90,scale=1.2]
\path (0,0.5) ++(135:1) coordinate (a);
\draw[invarrow=0.167,arrow=0.833] (a) -- node[below right] {$0$} (-0.05,0.5) -- node[below,pos=0.25] {$30$}
node[below,pos=0.75] {$12$} (-0.05,-0.5) -- node[below left] {$1$} ++(225:1);
\path (0,0.5) ++(45:1) coordinate (b);
\draw[invarrow=0.167,arrow=0.833] (b) -- node[above right] {$3$} (0.05,0.5) -- (0.05,-0.5) -- node[above left] {$2$} ++(-45:1);
\draw decorate [thin,decoration={markings,mark = at position 0.6 with {\arrow[scale=2]{<}}}] { (0,-0.5) -- (0,0) };
\draw decorate [thin,decoration={markings,mark = at position 0.6 with {\arrow[scale=2]{<}}}] { (0,0.5) -- (0,0) };
\end{tikzpicture}
\]
\end{prop}
This is the only nontrivial identity, in the following sense: contrary to the diagrams of 
Prop.~\ref{prop:dualtetra}
and \ref{prop:octa}, here the internal labels are not uniquely fixed by the external labels, so that a summation
has to be performed. It also means that this part of the proof of associativity is {\em not}\/
bijective: the number of configurations of the l.h.s.\ may differ from that of the r.h.s.\ (e.g.,
several configurations in the l.h.s.\ may correspond to a single configuration in the r.h.s.;
such a phenomenon will be exhibited in Appendix~\ref{app:assoc}).
\begin{proof}
Let us denote by $\mathcal T$ an entry of the l.h.s., multiplied by some prefactors for convenience:
\[
\mathcal T=\prod_{\alpha\ne \beta}\a_{a_{\alpha,\beta}}
\begin{tikzpicture}[baseline=-3pt,scale=1.2]
\path (0,0.5) ++(135:1) coordinate (a);
\draw[invarrow=0.29,arrow=0.73] (a) -- node[pos=0.45,above] {$\ss a_{0,3}$} node[pos=0.45] {$\ss a_{1,3}$} node[pos=0.45,below] {$\ss a_{2,3}$} (-0.05,0.5) -- 
(-0.05,-0.5) --  node[pos=0.55,above] {$\ss a_{1,0}$} node[pos=0.55] {$\ss a_{2,0}$} node[pos=0.55,below] {$\ss a_{3,0}$} ++(225:1);
\path (0,0.5) ++(45:1) coordinate (b);
\draw[invarrow=0.29,arrow=0.73] (b) -- node[pos=0.45,above] {$\ss a_{0,2}$} node[pos=0.45] {$\ss a_{1,2}$} node[pos=0.45,below] {$\ss a_{3,2}$} (0.05,0.5) -- (0.05,-0.5) --  node[pos=0.55,above] {$\ss a_{0,1}$} node[pos=0.55] {$\ss a_{2,1}$} node[pos=0.55,below] {$\ss a_{3,1}$}  ++(-45:1);
\path decorate [thin,decoration={markings,mark = at position 0.6 with {\arrow[scale=2]{<}}}] { (0,-0.5) -- node[pos=0.7] {$\ss a'_{3,0}\quad a'_{2,0}$} node[below,pos=0.7] {$\ss a'_{3,1}\quad a'_{2,1}$}(0,0) };
\draw decorate [thin,decoration={markings,mark = at position 0.6 with {\arrow[scale=2]{<}}}] { (0,0.5) -- node[above,pos=0.7] {$\ss a'_{0,3}\quad a'_{0,2}$} node[pos=0.7] {$\ss a'_{1,3}\quad a'_{1,2}$} (0,0) };
\end{tikzpicture}
\]
It is a function of the 12 external labels $a_{\alpha,\beta}$, $\alpha,\beta=0,\ldots,3$, $\alpha\ne \beta$. Weight conservation at each vertex
and \eqref{eq:defdualb} imply the following equalities:
\begin{itemize}
\item The trivial equalities
$a'_{0,3}=a'_{3,0}$, $a'_{1,3}=a'_{3,1}$, $a'_{0,2}=a'_{2,0}$, $a'_{1,2}=a'_{2,1}$ coming from the pairing, cf \eqref{eq:defdualb}.
\item
The vanishing of the sum of weights of all external labels; this is the same for l.h.s.\ and r.h.s., and therefore
$\ZZ_4$-invariant.
\item Three relations involving the internal labels, namely:
\begin{align}\label{eq:subst}\notag
a'_{0,3}&=a_{0,2}+a_{0,3}-a'_{0,2}
\\
a'_{1,3}&=a_{1,3}-(a_{0,2}+a_{3,2}-a_{2,3})+a'_{0,2}
\\\notag
a'_{1,2}&=a_{2,0}+a_{2,1}-a'_{0,2}
\end{align}
so that there remains one free parameter among the internal labels, here chosen to be $a'_{0,2}$.
\end{itemize}
The resulting entry is therefore a sum:
\begin{align*}
\mathcal T=&\prod_{\alpha\ne \beta} \a_{a_{\alpha,\beta}} \sum_{a'_{0,2}\ge 0} t^{a'_{0,2}a'_{1,3}} 
\a_{a'_{0,2}}\a_{a'_{0,3}}\a_{a'_{1,2}}\a_{a'_{1,3}}\a_{a_{3,2}+a_{0,2}-a'_{0,2}}\a_{a_{1,0}+a_{2,0}-a'_{0,2}}
\\
&u^{a'_{0,2},a'_{0,3},a_{3,2},a_{0,2},a_{0,3},a_{3,2}+a_{0,2}-a'_{0,2}}u^{a'_{1,3},a'_{1,2},a_{2,3},a_{1,3},a_{1,2},a_{3,2}+a_{0,2}-a'_{0,2}}
\\
&u^{a'_{0,2},a'_{1,2},a_{1,0},a_{2,0},a_{2,1},a_{1,0}+a_{2,0}-a'_{0,2}}u^{a'_{1,3},a'_{0,3},a_{0,1},a_{3,1},a_{3,0},a_{1,0}+a_{2,0}-a'_{0,2}}
\end{align*}
where for compactness we have not performed the substitution \eqref{eq:subst} yet.

Our strategy will be as follows: we shall show that $\mathcal T$ satisfies a set of $\ZZ_4$-invariant
relations which allows to reduce the identity to a special case for which we can show directly
that l.h.s.\ and r.h.s.\ agree.

We use the redefinition \eqref{eq:defu}:
\begin{align}\label{eq:T}
\mathcal T=&\sum_{a'_{0,2}\ge 0} t^{a'_{0,2}a'_{1,3}} 
(\a_{a'_{0,2}}\a_{a'_{0,3}}\a_{a'_{1,2}}\a_{a'_{1,3}}\a_{a_{3,2}+a_{0,2}-a'_{0,2}}\a_{a_{1,0}+a_{2,0}-a'_{0,2}})^{-1}
\\\notag
&u_{a'_{0,2},a'_{0,3},a_{3,2},a_{0,2},a_{0,3},a_{3,2}+a_{0,2}-a'_{0,2}}u_{a'_{1,3},a'_{1,2},a_{2,3},a_{1,3},a_{1,2},a_{3,2}+a_{0,2}-a'_{0,2}}
\\\notag
&u_{a'_{0,2},a'_{1,2},a_{1,0},a_{2,0},a_{2,1},a_{1,0}+a_{2,0}-a'_{0,2}}u_{a'_{1,3},a'_{0,3},a_{0,1},a_{3,1},a_{3,0},a_{1,0}+a_{2,0}-a'_{0,2}}
\end{align}

\newcommand\T[1]{{\mathcal T}_{#1}}
Let us introduce the shorthand notation where we put in subscript substituted variables; e.g.,
$\T{a_{2,3}+1}$ stands for $\mathcal T$ in which $a_{2,3}$ has been
substituted with $a_{2,3}+1$, or $u_{i+1,i''+1}=u_{j,i+1,j',i',j'',i''+1}$.
Also denote $u^{(1)},\ldots,u^{(4)}$ for the four $u$ factors in \eqref{eq:T}.

We shall start from the following recurrence relation satisfied by $u_{\ldots}$, following directly from its definition:
\begin{equation}\label{eq:recu}
u_{i''+1}
=t^{i}(1-t^{j'})u_{j'-1}
+(1-t^{i})u_{j+1,i-1}
\end{equation}
(for all nonnegative integer values of the arguments)
as well as all other relations obtained from it by dihedral symmetry, cf Lemma~\ref{lem:D6}.

Apply the reflected version $u_{j''+1}=t^{j'}(1-t^i)u_{i-1}+(1-t^{j'})u_{i'+1,j'-1}$ of \eqref{eq:recu} to $u^{(1)}$:
\begin{align*}
\T{a_{0,3}+1}
=&
\sum_{a'_{0,2}\ge 0} t^{a'_{0,2}a'_{1,3}}
(\a_{a'_{0,2}}\a_{a'_{0,3}+1}\a_{a'_{1,2}}\a_{a'_{1,3}}\a_{a_{3,2}+a_{0,2}-a'_{0,2}}\a_{a_{1,0}+a_{2,0}-a'_{0,2}})^{-1}
\\
&\left(
t^{a_{3,2}}(1-t^{a'_{0,3}})u^{(1)}
+
(1-t^{a_{3,2}})u^{(1)}_{a_{3,2}-1,a_{0,2}+1}
\right)
u^{(2)}u^{(3)}u^{(4)}_{a'_{0,3}+1}
\end{align*}
where the shift of $a'_{0,3}$ is due to \eqref{eq:subst}.
The second term in the parentheses contributes precisely $(1-t^{a_{3,2}})\T{a_{3,2}-1,a_{0,2}+1}$, 
but the first term is unwanted.

Now apply $u_{i+1}=t^{i'}(1-t^{j''})u_{j''-1}+(1-t^{i'})u_{j'+1,i'-1}$ to $u^{(4)}$: % j' -> a_30 ~ j'' j,i -> a_01, a_31 ~ j' i'
\begin{align*}
&t^{a_{3,1}}(1-t^{a_{3,0}})\T{a_{3,0}-1}
\\
&=
\sum_{a'_{0,2}\ge 0} t^{a'_{0,2}a'_{1,3}}
(\a_{a'_{0,2}}\a_{a'_{0,3}}\a_{a'_{1,2}}\a_{a'_{1,3}}\a_{a_{3,2}+a_{0,2}-a'_{0,2}}\a_{a_{1,0}+a_{2,0}-a'_{0,2}})^{-1}
\\
&u^{(1)}u^{(2)}u^{(3)}
\left(
u^{(4)}_{a'_{0,3}+1}-(1-t^{a_{3,1}})u^{(4)}_{a_{0,1}+1,a_{3,1}-1}
\right)
\end{align*}
Again, the second term in the parentheses contributes $-(1-t^{a_{3,1}})\T{a_{0,1}+1,a_{3,1}-1}$, 
but the first term is unwanted.

Subtracting $t^{a_{3,2}}$ times the second expression from the first, we note that the unwanted terms cancel exactly, 
so that we obtain our first relation:
\begin{equation}\label{eq:recur}
\T{a_{0,3}+1}-t^{a_{3,2}+a_{3,1}}(1-t^{a_{3,0}})\T{a_{3,0}-1}
=(1-t^{a_{3,2}})\T{a_{3,2}-1,a_{0,2}+1}
+t^{a_{3,2}}(1-t^{a_{3,1}})\T{a_{0,1}+1,a_{3,1}-1}
\end{equation}

A second identity can be derived in a similar but slightly more involved way.
Apply \eqref{eq:recu} in the form $u_{j'+1}=t^{j}(1-t^{i''})u_{i''-1}+(1-t^{j})u_{j-1,i+1}$ to $u^{(1)}$ in \eqref{eq:T} and
$u_{i''+1}=t^{i}(1-t^{j'})u_{j'-1}+(1-t^{i})u_{j+1,i-1}$ to $u^{(2)}$, 
multiply the second identity by $t^{a'_{0,2}}$ and subtract; one obtains: 
\begin{align*}
&\T{a_{3,2}+1}-t^{a_{2,0}+a_{2,1}}(1-t^{a_{0,1}})\T{a_{2,3}-1}
\\
&=
\sum_{a'_{0,2}\ge 0} t^{a'_{0,2}a'_{1,3}}
(\a_{a'_{0,2}}\a_{a'_{0,3}}\a_{a'_{1,2}}\a_{a'_{1,3}}\a_{a_{3,2}+a_{0,2}-a'_{0,2}}\a_{a_{1,0}+a_{2,0}-a'_{0,2}})^{-1}
\\
&\left(
t^{a'_{1,3}/2}(1-t^{a'_{0,2}})u^{(1)}_{a'_{0,2}-1,a'_{0,3}+1}u^{(2)}+t^{a'_{0,2}/2}(1-t^{a'_{1,2}})u^{(1)}
u^{(2)}_{a'_{1,2}-1,a'_{1,3}+1}
\right)
u^{(3)}u^{(4)}
\end{align*}

Applying four times appropriate versions of \eqref{eq:recu}, one can derive similarly:
\begin{align*}
&t^{a_{2,1}}(1-t^{a_{2,0}})\T{a_{2,0}-1,a_{3,0}+1}
+(1-t^{a_{2,1}})\T{a_{2,1}-1,a_{3,1}+1}
\\
&=
\sum_{a'_{0,2}\ge 0} t^{a'_{0,2}a'_{1,3}} 
(\a_{a'_{0,2}}\a_{a'_{0,3}}\a_{a'_{1,2}}\a_{a'_{1,3}}\a_{a_{3,2}+a_{0,2}-a'_{0,2}}\a_{a_{1,0}+a_{2,0}-a'_{0,2}})^{-1}
\\
&u^{(1)}u^{(2)}
\left(
u^{(3)}_{}u^{(4)}_{}
+u^{(3)}_{}u^{(4)}_{}
\right)
\end{align*}

By reindexing $a'_{0,2}\to a'_{0,2}+ 1$, the two expressions are equal.
In conclusion,
\begin{equation}\label{eq:recur2}
\T{a_{3,2}+1}-t^{a_{2,0}+a_{2,1}}(1-t^{a_{0,1}})\T{a_{2,3}-1}
=
t^{a_{2,1}}(1-t^{a_{2,0}})\T{a_{2,0}-1,a_{3,1}+1}
+(1-t^{a_{2,1}})\T{a_{2,1}-1,a_{3,1}+1}
%\\\label{eq:recurb}
%\T{a_{0,1}+1}-t^{a_{1,2}+a_{1,3}}(1-t^{a_{1,0}})\T{a_{1,0}-1}&=(1-t^{a_{1,2}})\T{a_{0,2}+1,a_{1,2}-1}+t^{a_{1,2}}(1-t^{a_{1,3}})\T{a_{0,3}+1,a_{1,3}-1}
%\end{align}
\end{equation}
Now note that \eqref{eq:recur2} is obtained from \eqref{eq:recur} by shift $\alpha\mapsto \alpha-1$ of indices in $\ZZ_4$.
Furthermore, the original expression \eqref{eq:T} has the symmetry of indices $\alpha\mapsto \alpha+2$ and
$\alpha\mapsto 3-\alpha$. Together this means that we have derived 
all equations obtained from \eqref{eq:recur} by $D_4$ action on the indices of the labels,
that is the usual $\ZZ_4$ shift of indices 
$\alpha \mapsto \alpha+1$, and the flip $\alpha\mapsto -\alpha\pmod 4$.

Now it is clear that by applying repeatedly \eqref{eq:recur}, we can express all values of
$\mathcal T$ in terms of its special case $a_{0,3}=0$ (induction on the sum of all arguments). Using the $D_4$ action,
we can reduce further to the
special case where the 8 parameters $a_{0,3},a_{2,3},a_{1,2},a_{3,2},a_{1,0},a_{3,0},a_{0,1},a_{2,1}$ are zero. In the latter case, the relations \eqref{eq:subst} simplify, and in particular, $a'_{0,2}$
is restricted to the range $\{\max(0,a_{0,2}-a_{1,3}),\ldots,a_{0,2}\}$; furthermore, 
using $u_{\ldots,a,0,b,\ldots}=\a_a\a_b$, $\mathcal T$ simplifies to
\begin{align*}
\mathcal T&=\sum_{a'_{0,2}=\max(0,a_{0,2}-a_{1,3})}^{a_{0,2}} t^{a'_{0,2}a'_{1,3}} 
(\a_{a'_{0,2}}\a_{a'_{0,3}}^2\a_{a'_{1,2}}^2\a_{a'_{1,3}})^{-1}
\\
&\quad\a_{a'_{0,3}}\a_{a_{0,2}}
\a_{a'_{1,2}}\a_{a_{1,3}}
\\
&\quad \a_{a'_{1,2}}\a_{a_{2,0}}
\a_{a'_{0,3}}\a_{a_{3,1}}
\\
&=\a_{a_{0,2}}\a_{a_{1,3}}\a_{a_{2,0}}\a_{a_{3,1}}
\sum_{a'_{0,2}=\max(0,a_{0,2}-a_{1,3})}^{a_{0,2}} \frac{t^{a'_{0,2}(a_{1,3}-a_{0,2}+a'_{0,2})}}
{\a_{a'_{0,2}}\a_{a_{1,3}-a_{0,2}+a'_{0,2}}}
\\
&=
\a_{a_{0,2}}\a_{a_{1,3}}\a_{a_{2,0}}\a_{a_{3,1}}
\sum_{i=0}^{\min(a_{0,2},a_{1,3})} \frac{t^{(a_{0,2}-i)(a_{1,3}-i)}}
{\a_{a_{0,2}-i}\a_{a_{1,3}-i}}
\end{align*}
The weight conservation is equivalent to $a_{0,2}=a_{2,0}$ and $a_{1,3}=a_{3,1}$, implying the $\ZZ_4$-invariance
of the final expression. It means that l.h.s.\ and r.h.s.\ are equal in this particular case of
external labels; but then, backtracking, we can use the $\ZZ_4$-invariant recurrence relations (the
$D_4$ orbit of \eqref{eq:recur}) to conclude that they are equal for all external labels.
\end{proof}

\appendix
\section{First few fugacities}\label{app:fug}
\def\size{1}\def\honey#1{\begin{tikzpicture}[baseline=-0.5cm]
\tikzset{honeyline/.style={black,thin,opaque}}
\honeycomb[\clip (0,0) -- (1,0) -- (0,1);]{#1}
\tikzset{honeyline/.style={transparent}}
\honeycomb{#1}
\end{tikzpicture}}
\noindent$\honey{0/0/0/0/0/0}=\honey{0/0/0/1/1/0}=\honey{0/0/0/2/2/0}=\honey{0/1/1/0/0/0}=\honey{0/1/1/1/1/0}=\honey{0/1/1/2/2/0}=\honey{0/2/2/0/0/0}=\honey{0/2/2/1/1/0}=\honey{0/2/2/2/2/0}=\honey{1/0/0/0/0/1}=\honey{1/0/0/1/1/1}=\honey{1/0/0/2/2/1}=\honey{1/0/1/0/1/0}=\honey{1/1/1/0/0/1}=\honey{1/1/2/0/1/0}=\honey{1/2/2/0/0/1}=\honey{2/0/0/0/0/2}=\honey{2/0/0/1/1/2}=\honey{2/0/0/2/2/2}=\honey{2/0/1/0/1/1}=\honey{2/0/2/0/2/0}=\honey{2/1/1/0/0/2}=\honey{2/1/2/0/1/1}=\honey{2/2/2/0/0/2}=1$

\noindent$\honey{0/1/0/1/0/1}=\honey{0/1/0/2/1/1}=1 -t$

\noindent$\honey{0/2/0/2/0/2}=(1-t)^2 (1+t)$

\noindent$\honey{0/2/1/1/0/1}=\honey{0/2/1/2/1/1}=\honey{1/1/0/1/0/2}=\honey{1/1/0/2/1/2}=(1-t) (1+t)$

\noindent$\honey{1/0/1/1/2/0}=\honey{1/1/2/1/2/0}=\honey{2/0/1/1/2/1}=1 + t$

\noindent$\honey{1/1/1/1/1/1}=1 + t -t^2$

\noindent$\honey{1/1/1/2/2/1}=\honey{1/2/2/1/1/1}=\honey{2/1/1/1/1/2}=1 + t -t^3$

\noindent$\honey{1/2/1/1/0/2}=(1-t) (1+t)^2$

\noindent$\honey{1/2/1/2/1/2}=(1-t) (1+t) \left(1+t+t^2-t^3\right)$

\noindent$\honey{1/2/2/2/2/1}=\honey{2/1/1/2/2/2}=\honey{2/2/2/1/1/2}=1 + t + t^2  -t^3  -t^4$

\noindent$\honey{2/1/2/1/2/1}=1 + t + t^2 -t^3$

\noindent$\honey{2/2/2/2/2/2}=1 + t + 2 t^2 -t^4 -2 t^5 -t^6 + t^7$

For example, the last equality is obtained by setting all parameters to $2$ in \eqref{eq:fug}:
\[
\sum_{r=0}^2 (-1)^r t^{2r+r(r+1)/2} \frac{\a_{4-r}}{\a_{2-r}^2 \a_r}
=
(1+t^2)(1+t+t^2)-t^3(1+t)(1+t+t^2)+t^7
\]

\section{Example of associativity}\label{app:assoc}
Here is the sequence of transformations from \eqref{eq:defLHS} to \eqref{eq:defRHS}
in the case $n=3$, where commuting transformations of the same type
have been grouped together:
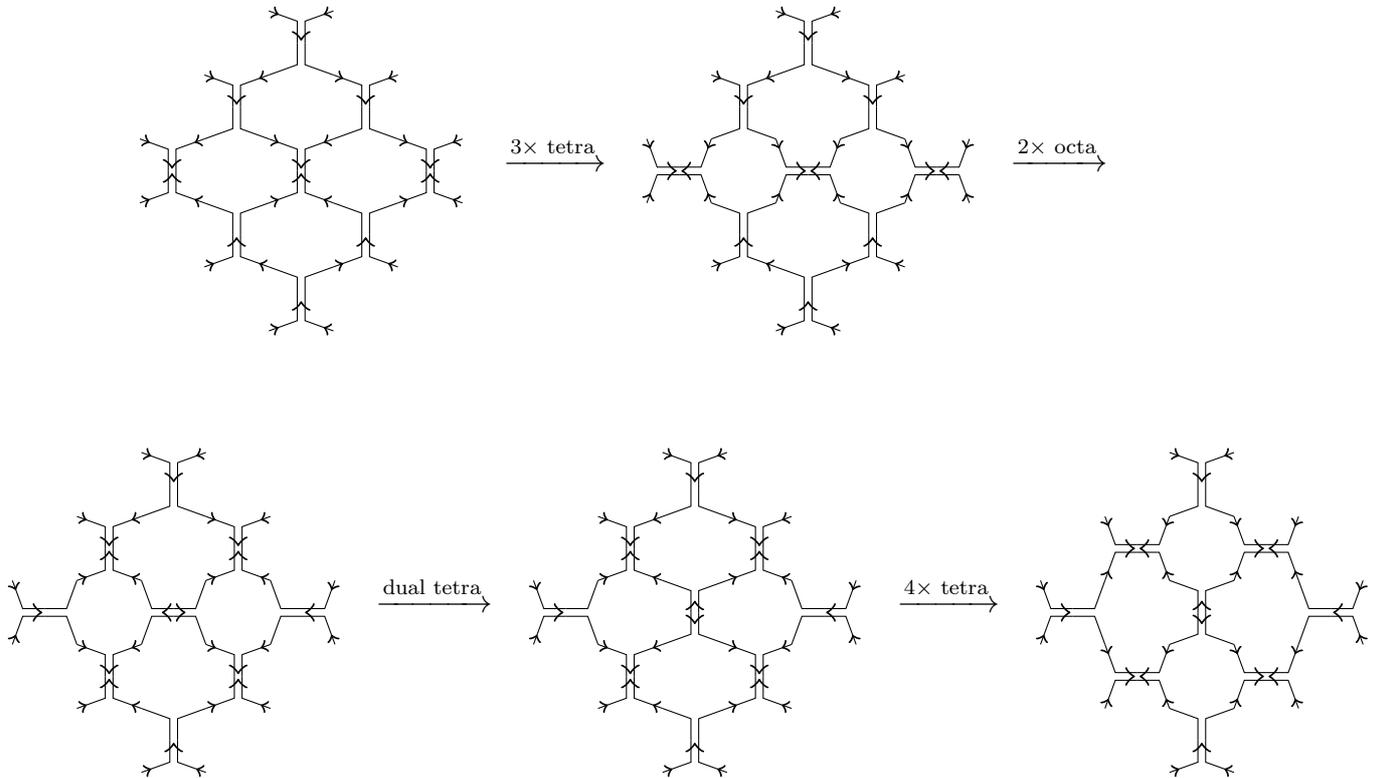
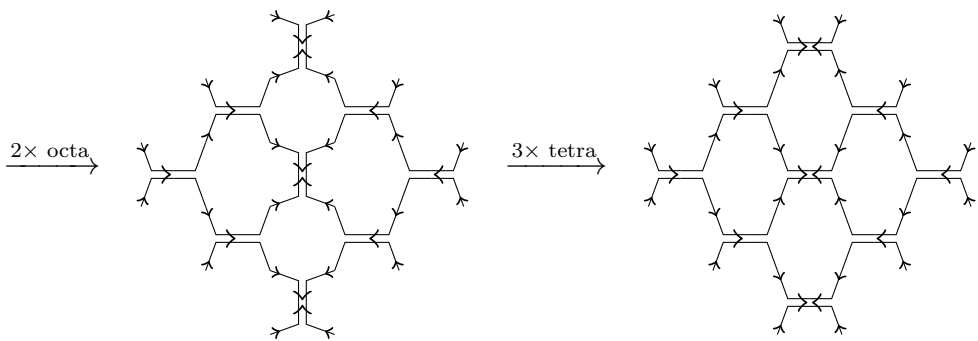
\[
\begin{tikzpicture}[baseline=(current  bounding  box.center),rotate=225,scale=1.2]
\useasboundingbox (1,1) rectangle (4,4); %decoration messed up bounding box
\draw[] (2.,1.5) -- (1.7,1.64) -- (1.53,1.47); \draw[] (1.5,2.) -- (1.64,1.7) -- (1.47,1.53); \draw[] (2.,2.5) -- (1.7,2.64) -- (1.53,2.47); \draw[] (1.5,3.) -- (1.64,2.7) -- (1.47,2.53); \draw[] (2.,3.5) -- (2.3,3.36) -- (2.47,3.53); \draw[] (2.5,3.) -- (2.36,3.3) -- (2.53,3.47); \draw[] (3.,1.5) -- (2.7,1.64) -- (2.53,1.47); \draw[] (2.5,2.) -- (2.64,1.7) -- (2.47,1.53); \draw[] (3.,2.5) -- (3.3,2.36) -- (3.47,2.53); \draw[] (3.5,2.) -- (3.36,2.3) -- (3.53,2.47); \draw[] (3.,3.5) -- (3.3,3.36) -- (3.47,3.53); \draw[] (3.5,3.) -- (3.36,3.3) -- (3.53,3.47); \draw[arrow=0.2] (1.,1.5) -- (1.3,1.36) -- (1.47,1.53); \draw[arrow=0.2] (1.5,1.) -- (1.36,1.3) -- (1.53,1.47); \path decorate[thin,decoration={markings,mark = at position 0.9 with {\arrow[scale=2]{>}}}] { (1.33,1.33) -- (1.5,1.5) }; \draw[arrow=0.2] (1.,2.5) -- (1.3,2.36) -- (1.47,2.53); \draw[arrow=0.2] (1.5,2.) -- (1.36,2.3) -- (1.53,2.47); \path decorate[thin,decoration={markings,mark = at position 0.9 with {\arrow[scale=2]{>}}}] { (1.33,2.33) -- (1.5,2.5) }; \draw[arrow=0.2] (1.,3.5) -- (1.3,3.36) -- (1.47,3.53); \draw[arrow=0.2] (1.5,3.) -- (1.36,3.3) -- (1.53,3.47); \path decorate[thin,decoration={markings,mark = at position 0.9 with {\arrow[scale=2]{>}}}] { (1.33,3.33) -- (1.5,3.5) }; \draw[arrow=0.2] (2.,1.5) -- (2.3,1.36) -- (2.47,1.53); \draw[arrow=0.2] (2.5,1.) -- (2.36,1.3) -- (2.53,1.47); \path decorate[thin,decoration={markings,mark = at position 0.9 with {\arrow[scale=2]{>}}}] { (2.33,1.33) -- (2.5,1.5) }; \draw[arrow=0.2] (2.,2.5) -- (2.3,2.36) -- (2.47,2.53); \draw[arrow=0.2] (2.5,2.) -- (2.36,2.3) -- (2.53,2.47); \path decorate[thin,decoration={markings,mark = at position 0.9 with {\arrow[scale=2]{>}}}] { (2.33,2.33) -- (2.5,2.5) }; \draw[arrow=0.2] (2.,3.5) -- (1.7,3.64) -- (1.53,3.47); \draw[arrow=0.2] (1.5,4.) -- (1.64,3.7) -- (1.47,3.53); \path decorate[thin,decoration={markings,mark = at position 0.9 with {\arrow[scale=2]{>}}}] { (1.67,3.67) -- (1.5,3.5) }; \draw[arrow=0.2] (3.,1.5) -- (3.3,1.36) -- (3.47,1.53); \draw[arrow=0.2] (3.5,1.) -- (3.36,1.3) -- (3.53,1.47); \path decorate[thin,decoration={markings,mark = at position 0.9 with {\arrow[scale=2]{>}}}] { (3.33,1.33) -- (3.5,1.5) }; \draw[arrow=0.2] (3.,2.5) -- (2.7,2.64) -- (2.53,2.47); \draw[arrow=0.2] (2.5,3.) -- (2.64,2.7) -- (2.47,2.53); \path decorate[thin,decoration={markings,mark = at position 0.9 with {\arrow[scale=2]{>}}}] { (2.67,2.67) -- (2.5,2.5) }; \draw[arrow=0.2] (3.,3.5) -- (2.7,3.64) -- (2.53,3.47); \draw[arrow=0.2] (2.5,4.) -- (2.64,3.7) -- (2.47,3.53); \path decorate[thin,decoration={markings,mark = at position 0.9 with {\arrow[scale=2]{>}}}] { (2.67,3.67) -- (2.5,3.5) }; \draw[arrow=0.2] (4.,1.5) -- (3.7,1.64) -- (3.53,1.47); \draw[arrow=0.2] (3.5,2.) -- (3.64,1.7) -- (3.47,1.53); \path decorate[thin,decoration={markings,mark = at position 0.9 with {\arrow[scale=2]{>}}}] { (3.67,1.67) -- (3.5,1.5) }; \draw[arrow=0.2] (4.,2.5) -- (3.7,2.64) -- (3.53,2.47); \draw[arrow=0.2] (3.5,3.) -- (3.64,2.7) -- (3.47,2.53); \path decorate[thin,decoration={markings,mark = at position 0.9 with {\arrow[scale=2]{>}}}] { (3.67,2.67) -- (3.5,2.5) }; \draw[arrow=0.2] (4.,3.5) -- (3.7,3.64) -- (3.53,3.47); \draw[arrow=0.2] (3.5,4.) -- (3.64,3.7) -- (3.47,3.53); \path decorate[thin,decoration={markings,mark = at position 0.9 with {\arrow[scale=2]{>}}}] { (3.67,3.67) -- (3.5,3.5) }; 
\end{tikzpicture}
\xrightarrow{3\times\text{ tetra}}
\begin{tikzpicture}[baseline=(current  bounding  box.center),rotate=225,scale=1.2]
\useasboundingbox (1,1) rectangle (4,4);
\draw[] (2.,1.5) -- (1.7,1.64) -- (1.53,1.47); \draw[] (1.5,2.) -- (1.64,1.7) -- (1.47,1.53); \draw[] (2.,2.5) -- (1.7,2.64) -- (1.53,2.47); \draw[] (1.5,3.) -- (1.64,2.7) -- (1.47,2.53); \draw[] (2.,3.5) -- (2.3,3.36) -- (2.47,3.53); \draw[] (2.5,3.) -- (2.36,3.3) -- (2.53,3.47); \draw[] (3.,1.5) -- (2.7,1.64) -- (2.53,1.47); \draw[] (2.5,2.) -- (2.64,1.7) -- (2.47,1.53); \draw[] (3.,2.5) -- (3.3,2.36) -- (3.47,2.53); \draw[] (3.5,2.) -- (3.36,2.3) -- (3.53,2.47); \draw[] (3.,3.5) -- (3.3,3.36) -- (3.47,3.53); \draw[] (3.5,3.) -- (3.36,3.3) -- (3.53,3.47); \draw[arrow=0.2] (1.,1.5) -- (1.3,1.36) -- (1.47,1.53); \draw[arrow=0.2] (1.5,1.) -- (1.36,1.3) -- (1.53,1.47); \path decorate[thin,decoration={markings,mark = at position 0.9 with {\arrow[scale=2]{>}}}] { (1.33,1.33) -- (1.5,1.5) }; \draw[arrow=0.2] (1.,2.5) -- (1.3,2.36) -- (1.47,2.53); \draw[arrow=0.2] (1.5,2.) -- (1.36,2.3) -- (1.53,2.47); \path decorate[thin,decoration={markings,mark = at position 0.9 with {\arrow[scale=2]{>}}}] { (1.33,2.33) -- (1.5,2.5) }; \draw[arrow=0.2] (1.5,3.) -- (1.64,3.3) -- (1.47,3.47); \draw[arrow=0.2] (2.,3.5) -- (1.7,3.36) -- (1.53,3.53); \path decorate[thin,decoration={markings,mark = at position 0.9 with {\arrow[scale=2]{>}}}] { (1.67,3.33) -- (1.5,3.5) }; \draw[arrow=0.2] (1.5,4.) -- (1.36,3.7) -- (1.53,3.53); \draw[arrow=0.2] (1.,3.5) -- (1.3,3.64) -- (1.47,3.47); \path decorate[thin,decoration={markings,mark = at position 0.9 with {\arrow[scale=2]{>}}}] { (1.33,3.67) -- (1.5,3.5) }; \draw[arrow=0.2] (2.,1.5) -- (2.3,1.36) -- (2.47,1.53); \draw[arrow=0.2] (2.5,1.) -- (2.36,1.3) -- (2.53,1.47); \path decorate[thin,decoration={markings,mark = at position 0.9 with {\arrow[scale=2]{>}}}] { (2.33,1.33) -- (2.5,1.5) }; \draw[arrow=0.2] (2.5,2.) -- (2.64,2.3) -- (2.47,2.47); \draw[arrow=0.2] (3.,2.5) -- (2.7,2.36) -- (2.53,2.53); \path decorate[thin,decoration={markings,mark = at position 0.9 with {\arrow[scale=2]{>}}}] { (2.67,2.33) -- (2.5,2.5) }; \draw[arrow=0.2] (2.5,3.) -- (2.36,2.7) -- (2.53,2.53); \draw[arrow=0.2] (2.,2.5) -- (2.3,2.64) -- (2.47,2.47); \path decorate[thin,decoration={markings,mark = at position 0.9 with {\arrow[scale=2]{>}}}] { (2.33,2.67) -- (2.5,2.5) }; \draw[arrow=0.2] (3.,3.5) -- (2.7,3.64) -- (2.53,3.47); \draw[arrow=0.2] (2.5,4.) -- (2.64,3.7) -- (2.47,3.53); \path decorate[thin,decoration={markings,mark = at position 0.9 with {\arrow[scale=2]{>}}}] { (2.67,3.67) -- (2.5,3.5) }; \draw[arrow=0.2] (3.5,1.) -- (3.64,1.3) -- (3.47,1.47); \draw[arrow=0.2] (4.,1.5) -- (3.7,1.36) -- (3.53,1.53); \path decorate[thin,decoration={markings,mark = at position 0.9 with {\arrow[scale=2]{>}}}] { (3.67,1.33) -- (3.5,1.5) }; \draw[arrow=0.2] (3.5,2.) -- (3.36,1.7) -- (3.53,1.53); \draw[arrow=0.2] (3.,1.5) -- (3.3,1.64) -- (3.47,1.47); \path decorate[thin,decoration={markings,mark = at position 0.9 with {\arrow[scale=2]{>}}}] { (3.33,1.67) -- (3.5,1.5) }; \draw[arrow=0.2] (4.,2.5) -- (3.7,2.64) -- (3.53,2.47); \draw[arrow=0.2] (3.5,3.) -- (3.64,2.7) -- (3.47,2.53); \path decorate[thin,decoration={markings,mark = at position 0.9 with {\arrow[scale=2]{>}}}] { (3.67,2.67) -- (3.5,2.5) }; \draw[arrow=0.2] (4.,3.5) -- (3.7,3.64) -- (3.53,3.47); \draw[arrow=0.2] (3.5,4.) -- (3.64,3.7) -- (3.47,3.53); \path decorate[thin,decoration={markings,mark = at position 0.9 with {\arrow[scale=2]{>}}}] { (3.67,3.67) -- (3.5,3.5) }; 
\end{tikzpicture}
\xrightarrow{2\times\text{ octa}}
\]
\[
\begin{tikzpicture}[baseline=(current  bounding  box.center),rotate=225,scale=1.2]
\useasboundingbox (1,1) rectangle (4,4);
\draw[] (1.5,3.) -- (1.64,3.3) -- (1.47,3.47); \draw[] (2.,3.5) -- (1.7,3.36) -- (1.53,3.53); \draw[] (2.,1.5) -- (1.7,1.64) -- (1.53,1.47); \draw[] (1.5,2.) -- (1.64,1.7) -- (1.47,1.53); \draw[] (2.5,2.) -- (2.64,2.3) -- (2.47,2.47); \draw[] (3.,2.5) -- (2.7,2.36) -- (2.53,2.53); \draw[] (2.5,3.) -- (2.36,2.7) -- (2.53,2.53); \draw[] (2.,2.5) -- (2.3,2.64) -- (2.47,2.47); \draw[] (3.,3.5) -- (3.3,3.36) -- (3.47,3.53); \draw[] (3.5,3.) -- (3.36,3.3) -- (3.53,3.47); \draw[] (3.5,2.) -- (3.36,1.7) -- (3.53,1.53); \draw[] (3.,1.5) -- (3.3,1.64) -- (3.47,1.47); \draw[arrow=0.2] (1.,1.5) -- (1.3,1.36) -- (1.47,1.53); \draw[arrow=0.2] (1.5,1.) -- (1.36,1.3) -- (1.53,1.47); \path decorate[thin,decoration={markings,mark = at position 0.9 with {\arrow[scale=2]{>}}}] { (1.33,1.33) -- (1.5,1.5) }; \draw[arrow=0.2] (1.,2.5) -- (1.3,2.36) -- (1.47,2.53); \draw[arrow=0.2] (1.5,2.) -- (1.36,2.3) -- (1.53,2.47); \path decorate[thin,decoration={markings,mark = at position 0.9 with {\arrow[scale=2]{>}}}] { (1.33,2.33) -- (1.5,2.5) }; \draw[arrow=0.2] (1.5,4.) -- (1.36,3.7) -- (1.53,3.53); \draw[arrow=0.2] (1.,3.5) -- (1.3,3.64) -- (1.47,3.47); \path decorate[thin,decoration={markings,mark = at position 0.9 with {\arrow[scale=2]{>}}}] { (1.33,3.67) -- (1.5,3.5) }; \draw[arrow=0.2] (2.,1.5) -- (2.3,1.36) -- (2.47,1.53); \draw[arrow=0.2] (2.5,1.) -- (2.36,1.3) -- (2.53,1.47); \path decorate[thin,decoration={markings,mark = at position 0.9 with {\arrow[scale=2]{>}}}] { (2.33,1.33) -- (2.5,1.5) }; \draw[arrow=0.2] (2.,2.5) -- (1.7,2.64) -- (1.53,2.47); \draw[arrow=0.2] (1.5,3.) -- (1.64,2.7) -- (1.47,2.53); \path decorate[thin,decoration={markings,mark = at position 0.9 with {\arrow[scale=2]{>}}}] { (1.67,2.67) -- (1.5,2.5) }; \draw[arrow=0.2] (2.,3.5) -- (2.3,3.36) -- (2.47,3.53); \draw[arrow=0.2] (2.5,3.) -- (2.36,3.3) -- (2.53,3.47); \path decorate[thin,decoration={markings,mark = at position 0.9 with {\arrow[scale=2]{>}}}] { (2.33,3.33) -- (2.5,3.5) }; \draw[arrow=0.2] (3.,1.5) -- (2.7,1.64) -- (2.53,1.47); \draw[arrow=0.2] (2.5,2.) -- (2.64,1.7) -- (2.47,1.53); \path decorate[thin,decoration={markings,mark = at position 0.9 with {\arrow[scale=2]{>}}}] { (2.67,1.67) -- (2.5,1.5) }; \draw[arrow=0.2] (3.,2.5) -- (3.3,2.36) -- (3.47,2.53); \draw[arrow=0.2] (3.5,2.) -- (3.36,2.3) -- (3.53,2.47); \path decorate[thin,decoration={markings,mark = at position 0.9 with {\arrow[scale=2]{>}}}] { (3.33,2.33) -- (3.5,2.5) }; \draw[arrow=0.2] (3.,3.5) -- (2.7,3.64) -- (2.53,3.47); \draw[arrow=0.2] (2.5,4.) -- (2.64,3.7) -- (2.47,3.53); \path decorate[thin,decoration={markings,mark = at position 0.9 with {\arrow[scale=2]{>}}}] { (2.67,3.67) -- (2.5,3.5) }; \draw[arrow=0.2] (3.5,1.) -- (3.64,1.3) -- (3.47,1.47); \draw[arrow=0.2] (4.,1.5) -- (3.7,1.36) -- (3.53,1.53); \path decorate[thin,decoration={markings,mark = at position 0.9 with {\arrow[scale=2]{>}}}] { (3.67,1.33) -- (3.5,1.5) }; \draw[arrow=0.2] (4.,2.5) -- (3.7,2.64) -- (3.53,2.47); \draw[arrow=0.2] (3.5,3.) -- (3.64,2.7) -- (3.47,2.53); \path decorate[thin,decoration={markings,mark = at position 0.9 with {\arrow[scale=2]{>}}}] { (3.67,2.67) -- (3.5,2.5) }; \draw[arrow=0.2] (4.,3.5) -- (3.7,3.64) -- (3.53,3.47); \draw[arrow=0.2] (3.5,4.) -- (3.64,3.7) -- (3.47,3.53); \path decorate[thin,decoration={markings,mark = at position 0.9 with {\arrow[scale=2]{>}}}] { (3.67,3.67) -- (3.5,3.5) }; 
\path decorate[thin,decoration={markings,mark = at position 0.9 with {\arrow[scale=2]{<}}}] { (2.67,2.33) -- (2.5,2.5) };
\path decorate[thin,decoration={markings,mark = at position 0.9 with {\arrow[scale=2]{<}}}] { (2.33,2.67) -- (2.5,2.5) }; 
\end{tikzpicture}
\xrightarrow{\text{dual tetra}}
\begin{tikzpicture}[baseline=(current  bounding  box.center),rotate=225,scale=1.2]
\useasboundingbox (1,1) rectangle (4,4);
\draw[] (1.5,3.) -- (1.64,3.3) -- (1.47,3.47); \draw[] (2.,3.5) -- (1.7,3.36) -- (1.53,3.53); \draw[] (2.,1.5) -- (1.7,1.64) -- (1.53,1.47); \draw[] (1.5,2.) -- (1.64,1.7) -- (1.47,1.53); \draw[] (2.,2.5) -- (2.3,2.36) -- (2.47,2.53); \draw[] (2.5,2.) -- (2.36,2.3) -- (2.53,2.47); \draw[] (3.,2.5) -- (2.7,2.64) -- (2.53,2.47); \draw[] (2.5,3.) -- (2.64,2.7) -- (2.47,2.53); \draw[] (3.,3.5) -- (3.3,3.36) -- (3.47,3.53); \draw[] (3.5,3.) -- (3.36,3.3) -- (3.53,3.47); \draw[] (3.5,2.) -- (3.36,1.7) -- (3.53,1.53); \draw[] (3.,1.5) -- (3.3,1.64) -- (3.47,1.47); \draw[arrow=0.2] (1.,1.5) -- (1.3,1.36) -- (1.47,1.53); \draw[arrow=0.2] (1.5,1.) -- (1.36,1.3) -- (1.53,1.47); \path decorate[thin,decoration={markings,mark = at position 0.9 with {\arrow[scale=2]{>}}}] { (1.33,1.33) -- (1.5,1.5) }; \draw[arrow=0.2] (1.,2.5) -- (1.3,2.36) -- (1.47,2.53); \draw[arrow=0.2] (1.5,2.) -- (1.36,2.3) -- (1.53,2.47); \path decorate[thin,decoration={markings,mark = at position 0.9 with {\arrow[scale=2]{>}}}] { (1.33,2.33) -- (1.5,2.5) }; \draw[arrow=0.2] (1.5,4.) -- (1.36,3.7) -- (1.53,3.53); \draw[arrow=0.2] (1.,3.5) -- (1.3,3.64) -- (1.47,3.47); \path decorate[thin,decoration={markings,mark = at position 0.9 with {\arrow[scale=2]{>}}}] { (1.33,3.67) -- (1.5,3.5) }; \draw[arrow=0.2] (2.,1.5) -- (2.3,1.36) -- (2.47,1.53); \draw[arrow=0.2] (2.5,1.) -- (2.36,1.3) -- (2.53,1.47); \path decorate[thin,decoration={markings,mark = at position 0.9 with {\arrow[scale=2]{>}}}] { (2.33,1.33) -- (2.5,1.5) }; \draw[arrow=0.2] (2.,2.5) -- (1.7,2.64) -- (1.53,2.47); \draw[arrow=0.2] (1.5,3.) -- (1.64,2.7) -- (1.47,2.53); \path decorate[thin,decoration={markings,mark = at position 0.9 with {\arrow[scale=2]{>}}}] { (1.67,2.67) -- (1.5,2.5) }; \draw[arrow=0.2] (2.,3.5) -- (2.3,3.36) -- (2.47,3.53); \draw[arrow=0.2] (2.5,3.) -- (2.36,3.3) -- (2.53,3.47); \path decorate[thin,decoration={markings,mark = at position 0.9 with {\arrow[scale=2]{>}}}] { (2.33,3.33) -- (2.5,3.5) }; \draw[arrow=0.2] (3.,1.5) -- (2.7,1.64) -- (2.53,1.47); \draw[arrow=0.2] (2.5,2.) -- (2.64,1.7) -- (2.47,1.53); \path decorate[thin,decoration={markings,mark = at position 0.9 with {\arrow[scale=2]{>}}}] { (2.67,1.67) -- (2.5,1.5) }; \draw[arrow=0.2] (3.,2.5) -- (3.3,2.36) -- (3.47,2.53); \draw[arrow=0.2] (3.5,2.) -- (3.36,2.3) -- (3.53,2.47); \path decorate[thin,decoration={markings,mark = at position 0.9 with {\arrow[scale=2]{>}}}] { (3.33,2.33) -- (3.5,2.5) }; \draw[arrow=0.2] (3.,3.5) -- (2.7,3.64) -- (2.53,3.47); \draw[arrow=0.2] (2.5,4.) -- (2.64,3.7) -- (2.47,3.53); \path decorate[thin,decoration={markings,mark = at position 0.9 with {\arrow[scale=2]{>}}}] { (2.67,3.67) -- (2.5,3.5) }; \draw[arrow=0.2] (3.5,1.) -- (3.64,1.3) -- (3.47,1.47); \draw[arrow=0.2] (4.,1.5) -- (3.7,1.36) -- (3.53,1.53); \path decorate[thin,decoration={markings,mark = at position 0.9 with {\arrow[scale=2]{>}}}] { (3.67,1.33) -- (3.5,1.5) }; \draw[arrow=0.2] (4.,2.5) -- (3.7,2.64) -- (3.53,2.47); \draw[arrow=0.2] (3.5,3.) -- (3.64,2.7) -- (3.47,2.53); \path decorate[thin,decoration={markings,mark = at position 0.9 with {\arrow[scale=2]{>}}}] { (3.67,2.67) -- (3.5,2.5) }; \draw[arrow=0.2] (4.,3.5) -- (3.7,3.64) -- (3.53,3.47); \draw[arrow=0.2] (3.5,4.) -- (3.64,3.7) -- (3.47,3.53); \path decorate[thin,decoration={markings,mark = at position 0.9 with {\arrow[scale=2]{>}}}] { (3.67,3.67) -- (3.5,3.5) }; 
\path decorate[thin,decoration={markings,mark = at position 0.9 with {\arrow[scale=2]{<}}}] { (2.33,2.33) -- (2.5,2.5) };
\path decorate[thin,decoration={markings,mark = at position 0.9 with {\arrow[scale=2]{<}}}] { (2.67,2.67) -- (2.5,2.5) };
\end{tikzpicture}
\xrightarrow{4\times\text{ tetra}}
\begin{tikzpicture}[baseline=(current  bounding  box.center),rotate=225,scale=1.2]
\useasboundingbox (1,1) rectangle (4,4);
\draw[] (1.5,3.) -- (1.64,3.3) -- (1.47,3.47); \draw[] (2.,3.5) -- (1.7,3.36) -- (1.53,3.53); \draw[] (2.,1.5) -- (1.7,1.64) -- (1.53,1.47); \draw[] (1.5,2.) -- (1.64,1.7) -- (1.47,1.53); \draw[] (2.,2.5) -- (2.3,2.36) -- (2.47,2.53); \draw[] (2.5,2.) -- (2.36,2.3) -- (2.53,2.47); \draw[] (3.,2.5) -- (2.7,2.64) -- (2.53,2.47); \draw[] (2.5,3.) -- (2.64,2.7) -- (2.47,2.53); \draw[] (3.,3.5) -- (3.3,3.36) -- (3.47,3.53); \draw[] (3.5,3.) -- (3.36,3.3) -- (3.53,3.47); \draw[] (3.5,2.) -- (3.36,1.7) -- (3.53,1.53); \draw[] (3.,1.5) -- (3.3,1.64) -- (3.47,1.47); \draw[arrow=0.2] (1.,1.5) -- (1.3,1.36) -- (1.47,1.53); \draw[arrow=0.2] (1.5,1.) -- (1.36,1.3) -- (1.53,1.47); \path decorate[thin,decoration={markings,mark = at position 0.9 with {\arrow[scale=2]{>}}}] { (1.33,1.33) -- (1.5,1.5) }; \draw[arrow=0.2] (1.5,2.) -- (1.64,2.3) -- (1.47,2.47); \draw[arrow=0.2] (2.,2.5) -- (1.7,2.36) -- (1.53,2.53); \path decorate[thin,decoration={markings,mark = at position 0.9 with {\arrow[scale=2]{>}}}] { (1.67,2.33) -- (1.5,2.5) }; \draw[arrow=0.2] (1.5,3.) -- (1.36,2.7) -- (1.53,2.53); \draw[arrow=0.2] (1.,2.5) -- (1.3,2.64) -- (1.47,2.47); \path decorate[thin,decoration={markings,mark = at position 0.9 with {\arrow[scale=2]{>}}}] { (1.33,2.67) -- (1.5,2.5) }; \draw[arrow=0.2] (1.5,4.) -- (1.36,3.7) -- (1.53,3.53); \draw[arrow=0.2] (1.,3.5) -- (1.3,3.64) -- (1.47,3.47); \path decorate[thin,decoration={markings,mark = at position 0.9 with {\arrow[scale=2]{>}}}] { (1.33,3.67) -- (1.5,3.5) }; \draw[arrow=0.2] (2.5,1.) -- (2.64,1.3) -- (2.47,1.47); \draw[arrow=0.2] (3.,1.5) -- (2.7,1.36) -- (2.53,1.53); \path decorate[thin,decoration={markings,mark = at position 0.9 with {\arrow[scale=2]{>}}}] { (2.67,1.33) -- (2.5,1.5) }; \draw[arrow=0.2] (2.5,2.) -- (2.36,1.7) -- (2.53,1.53); \draw[arrow=0.2] (2.,1.5) -- (2.3,1.64) -- (2.47,1.47); \path decorate[thin,decoration={markings,mark = at position 0.9 with {\arrow[scale=2]{>}}}] { (2.33,1.67) -- (2.5,1.5) }; \draw[arrow=0.2] (2.5,3.) -- (2.64,3.3) -- (2.47,3.47); \draw[arrow=0.2] (3.,3.5) -- (2.7,3.36) -- (2.53,3.53); \path decorate[thin,decoration={markings,mark = at position 0.9 with {\arrow[scale=2]{>}}}] { (2.67,3.33) -- (2.5,3.5) }; \draw[arrow=0.2] (2.5,4.) -- (2.36,3.7) -- (2.53,3.53); \draw[arrow=0.2] (2.,3.5) -- (2.3,3.64) -- (2.47,3.47); \path decorate[thin,decoration={markings,mark = at position 0.9 with {\arrow[scale=2]{>}}}] { (2.33,3.67) -- (2.5,3.5) }; \draw[arrow=0.2] (3.5,1.) -- (3.64,1.3) -- (3.47,1.47); \draw[arrow=0.2] (4.,1.5) -- (3.7,1.36) -- (3.53,1.53); \path decorate[thin,decoration={markings,mark = at position 0.9 with {\arrow[scale=2]{>}}}] { (3.67,1.33) -- (3.5,1.5) }; \draw[arrow=0.2] (3.5,2.) -- (3.64,2.3) -- (3.47,2.47); \draw[arrow=0.2] (4.,2.5) -- (3.7,2.36) -- (3.53,2.53); \path decorate[thin,decoration={markings,mark = at position 0.9 with {\arrow[scale=2]{>}}}] { (3.67,2.33) -- (3.5,2.5) }; \draw[arrow=0.2] (3.5,3.) -- (3.36,2.7) -- (3.53,2.53); \draw[arrow=0.2] (3.,2.5) -- (3.3,2.64) -- (3.47,2.47); \path decorate[thin,decoration={markings,mark = at position 0.9 with {\arrow[scale=2]{>}}}] { (3.33,2.67) -- (3.5,2.5) }; \draw[arrow=0.2] (4.,3.5) -- (3.7,3.64) -- (3.53,3.47); \draw[arrow=0.2] (3.5,4.) -- (3.64,3.7) -- (3.47,3.53); \path decorate[thin,decoration={markings,mark = at position 0.9 with {\arrow[scale=2]{>}}}] { (3.67,3.67) -- (3.5,3.5) }; 
\path decorate[thin,decoration={markings,mark = at position 0.9 with {\arrow[scale=2]{<}}}] { (2.33,2.33) -- (2.5,2.5) };
\path decorate[thin,decoration={markings,mark = at position 0.9 with {\arrow[scale=2]{<}}}] { (2.67,2.67) -- (2.5,2.5) };
\end{tikzpicture}
\]
\[
\xrightarrow{2\times\text{ octa}}
\begin{tikzpicture}[baseline=(current  bounding  box.center),rotate=225,scale=1.2]
\useasboundingbox (1,1) rectangle (4,4);
\draw[] (1.5,2.) -- (1.64,2.3) -- (1.47,2.47); \draw[] (2.,2.5) -- (1.7,2.36) -- (1.53,2.53); \draw[] (1.5,3.) -- (1.64,3.3) -- (1.47,3.47); \draw[] (2.,3.5) -- (1.7,3.36) -- (1.53,3.53); \draw[] (2.5,2.) -- (2.36,1.7) -- (2.53,1.53); \draw[] (2.,1.5) -- (2.3,1.64) -- (2.47,1.47); \draw[] (2.5,3.) -- (2.64,3.3) -- (2.47,3.47); \draw[] (3.,3.5) -- (2.7,3.36) -- (2.53,3.53); \draw[] (3.5,2.) -- (3.36,1.7) -- (3.53,1.53); \draw[] (3.,1.5) -- (3.3,1.64) -- (3.47,1.47); \draw[] (3.5,3.) -- (3.36,2.7) -- (3.53,2.53); \draw[] (3.,2.5) -- (3.3,2.64) -- (3.47,2.47); \draw[arrow=0.2] (1.,1.5) -- (1.3,1.36) -- (1.47,1.53); \draw[arrow=0.2] (1.5,1.) -- (1.36,1.3) -- (1.53,1.47); \path decorate[thin,decoration={markings,mark = at position 0.9 with {\arrow[scale=2]{>}}}] { (1.33,1.33) -- (1.5,1.5) }; \draw[arrow=0.2] (1.5,3.) -- (1.36,2.7) -- (1.53,2.53); \draw[arrow=0.2] (1.,2.5) -- (1.3,2.64) -- (1.47,2.47); \path decorate[thin,decoration={markings,mark = at position 0.9 with {\arrow[scale=2]{>}}}] { (1.33,2.67) -- (1.5,2.5) }; \draw[arrow=0.2] (1.5,4.) -- (1.36,3.7) -- (1.53,3.53); \draw[arrow=0.2] (1.,3.5) -- (1.3,3.64) -- (1.47,3.47); \path decorate[thin,decoration={markings,mark = at position 0.9 with {\arrow[scale=2]{>}}}] { (1.33,3.67) -- (1.5,3.5) }; \draw[arrow=0.2] (2.,1.5) -- (1.7,1.64) -- (1.53,1.47); \draw[arrow=0.2] (1.5,2.) -- (1.64,1.7) -- (1.47,1.53); \path decorate[thin,decoration={markings,mark = at position 0.9 with {\arrow[scale=2]{>}}}] { (1.67,1.67) -- (1.5,1.5) }; \draw[arrow=0.2] (2.,2.5) -- (2.3,2.36) -- (2.47,2.53); \draw[arrow=0.2] (2.5,2.) -- (2.36,2.3) -- (2.53,2.47); \path decorate[thin,decoration={markings,mark = at position 0.9 with {\arrow[scale=2]{>}}}] { (2.33,2.33) -- (2.5,2.5) }; \draw[arrow=0.2] (2.5,1.) -- (2.64,1.3) -- (2.47,1.47); \draw[arrow=0.2] (3.,1.5) -- (2.7,1.36) -- (2.53,1.53); \path decorate[thin,decoration={markings,mark = at position 0.9 with {\arrow[scale=2]{>}}}] { (2.67,1.33) -- (2.5,1.5) }; \draw[arrow=0.2] (2.5,4.) -- (2.36,3.7) -- (2.53,3.53); \draw[arrow=0.2] (2.,3.5) -- (2.3,3.64) -- (2.47,3.47); \path decorate[thin,decoration={markings,mark = at position 0.9 with {\arrow[scale=2]{>}}}] { (2.33,3.67) -- (2.5,3.5) }; \draw[arrow=0.2] (3.,2.5) -- (2.7,2.64) -- (2.53,2.47); \draw[arrow=0.2] (2.5,3.) -- (2.64,2.7) -- (2.47,2.53); \path decorate[thin,decoration={markings,mark = at position 0.9 with {\arrow[scale=2]{>}}}] { (2.67,2.67) -- (2.5,2.5) }; \draw[arrow=0.2] (3.,3.5) -- (3.3,3.36) -- (3.47,3.53); \draw[arrow=0.2] (3.5,3.) -- (3.36,3.3) -- (3.53,3.47); \path decorate[thin,decoration={markings,mark = at position 0.9 with {\arrow[scale=2]{>}}}] { (3.33,3.33) -- (3.5,3.5) }; \draw[arrow=0.2] (3.5,1.) -- (3.64,1.3) -- (3.47,1.47); \draw[arrow=0.2] (4.,1.5) -- (3.7,1.36) -- (3.53,1.53); \path decorate[thin,decoration={markings,mark = at position 0.9 with {\arrow[scale=2]{>}}}] { (3.67,1.33) -- (3.5,1.5) }; \draw[arrow=0.2] (3.5,2.) -- (3.64,2.3) -- (3.47,2.47); \draw[arrow=0.2] (4.,2.5) -- (3.7,2.36) -- (3.53,2.53); \path decorate[thin,decoration={markings,mark = at position 0.9 with {\arrow[scale=2]{>}}}] { (3.67,2.33) -- (3.5,2.5) }; \draw[arrow=0.2] (4.,3.5) -- (3.7,3.64) -- (3.53,3.47); \draw[arrow=0.2] (3.5,4.) -- (3.64,3.7) -- (3.47,3.53); \path decorate[thin,decoration={markings,mark = at position 0.9 with {\arrow[scale=2]{>}}}] { (3.67,3.67) -- (3.5,3.5) }; 
\end{tikzpicture}
\xrightarrow{3\times\text{ tetra}}
\begin{tikzpicture}[baseline=(current  bounding  box.center),rotate=225,scale=1.2]
\useasboundingbox (1,1) rectangle (4,4);
\draw[] (1.5,2.) -- (1.64,2.3) -- (1.47,2.47); \draw[] (2.,2.5) -- (1.7,2.36) -- (1.53,2.53); \draw[] (1.5,3.) -- (1.64,3.3) -- (1.47,3.47); \draw[] (2.,3.5) -- (1.7,3.36) -- (1.53,3.53); \draw[] (2.5,2.) -- (2.36,1.7) -- (2.53,1.53); \draw[] (2.,1.5) -- (2.3,1.64) -- (2.47,1.47); \draw[] (2.5,3.) -- (2.64,3.3) -- (2.47,3.47); \draw[] (3.,3.5) -- (2.7,3.36) -- (2.53,3.53); \draw[] (3.5,2.) -- (3.36,1.7) -- (3.53,1.53); \draw[] (3.,1.5) -- (3.3,1.64) -- (3.47,1.47); \draw[] (3.5,3.) -- (3.36,2.7) -- (3.53,2.53); \draw[] (3.,2.5) -- (3.3,2.64) -- (3.47,2.47); \draw[arrow=0.2] (1.5,1.) -- (1.64,1.3) -- (1.47,1.47); \draw[arrow=0.2] (2.,1.5) -- (1.7,1.36) -- (1.53,1.53); \path decorate[thin,decoration={markings,mark = at position 0.9 with {\arrow[scale=2]{>}}}] { (1.67,1.33) -- (1.5,1.5) }; \draw[arrow=0.2] (1.5,2.) -- (1.36,1.7) -- (1.53,1.53); \draw[arrow=0.2] (1.,1.5) -- (1.3,1.64) -- (1.47,1.47); \path decorate[thin,decoration={markings,mark = at position 0.9 with {\arrow[scale=2]{>}}}] { (1.33,1.67) -- (1.5,1.5) }; \draw[arrow=0.2] (1.5,3.) -- (1.36,2.7) -- (1.53,2.53); \draw[arrow=0.2] (1.,2.5) -- (1.3,2.64) -- (1.47,2.47); \path decorate[thin,decoration={markings,mark = at position 0.9 with {\arrow[scale=2]{>}}}] { (1.33,2.67) -- (1.5,2.5) }; \draw[arrow=0.2] (1.5,4.) -- (1.36,3.7) -- (1.53,3.53); \draw[arrow=0.2] (1.,3.5) -- (1.3,3.64) -- (1.47,3.47); \path decorate[thin,decoration={markings,mark = at position 0.9 with {\arrow[scale=2]{>}}}] { (1.33,3.67) -- (1.5,3.5) }; \draw[arrow=0.2] (2.5,1.) -- (2.64,1.3) -- (2.47,1.47); \draw[arrow=0.2] (3.,1.5) -- (2.7,1.36) -- (2.53,1.53); \path decorate[thin,decoration={markings,mark = at position 0.9 with {\arrow[scale=2]{>}}}] { (2.67,1.33) -- (2.5,1.5) }; \draw[arrow=0.2] (2.5,2.) -- (2.64,2.3) -- (2.47,2.47); \draw[arrow=0.2] (3.,2.5) -- (2.7,2.36) -- (2.53,2.53); \path decorate[thin,decoration={markings,mark = at position 0.9 with {\arrow[scale=2]{>}}}] { (2.67,2.33) -- (2.5,2.5) }; \draw[arrow=0.2] (2.5,3.) -- (2.36,2.7) -- (2.53,2.53); \draw[arrow=0.2] (2.,2.5) -- (2.3,2.64) -- (2.47,2.47); \path decorate[thin,decoration={markings,mark = at position 0.9 with {\arrow[scale=2]{>}}}] { (2.33,2.67) -- (2.5,2.5) }; \draw[arrow=0.2] (2.5,4.) -- (2.36,3.7) -- (2.53,3.53); \draw[arrow=0.2] (2.,3.5) -- (2.3,3.64) -- (2.47,3.47); \path decorate[thin,decoration={markings,mark = at position 0.9 with {\arrow[scale=2]{>}}}] { (2.33,3.67) -- (2.5,3.5) }; \draw[arrow=0.2] (3.5,1.) -- (3.64,1.3) -- (3.47,1.47); \draw[arrow=0.2] (4.,1.5) -- (3.7,1.36) -- (3.53,1.53); \path decorate[thin,decoration={markings,mark = at position 0.9 with {\arrow[scale=2]{>}}}] { (3.67,1.33) -- (3.5,1.5) }; \draw[arrow=0.2] (3.5,2.) -- (3.64,2.3) -- (3.47,2.47); \draw[arrow=0.2] (4.,2.5) -- (3.7,2.36) -- (3.53,2.53); \path decorate[thin,decoration={markings,mark = at position 0.9 with {\arrow[scale=2]{>}}}] { (3.67,2.33) -- (3.5,2.5) }; \draw[arrow=0.2] (3.5,3.) -- (3.64,3.3) -- (3.47,3.47); \draw[arrow=0.2] (4.,3.5) -- (3.7,3.36) -- (3.53,3.53); \path decorate[thin,decoration={markings,mark = at position 0.9 with {\arrow[scale=2]{>}}}] { (3.67,3.33) -- (3.5,3.5) }; \draw[arrow=0.2] (3.5,4.) -- (3.36,3.7) -- (3.53,3.53); \draw[arrow=0.2] (3.,3.5) -- (3.3,3.64) -- (3.47,3.47); \path decorate[thin,decoration={markings,mark = at position 0.9 with {\arrow[scale=2]{>}}}] { (3.33,3.67) -- (3.5,3.5) }; 
\end{tikzpicture}
\]

Now let us consider an example. Set $k=3$, $\lambda=\mu=(1)$,
$\nu=(1,1)$, $\rho=(2,1,1)$. The coefficient of $P^\rho$ in the triple
product $P^\lambda P^\mu P^\nu$ is $2+2t+t^2$; however it decomposes in two
different ways as $(1)+(1+t)^2=(1+t)+(1+t+t^2)$:
\begin{alignat*}{2}
P^{(1)} P^{(1)}&=P^{(2)}+(1+t)P^{(1,1)},&\quad &\left\{
\begin{aligned}
P^{(2)}P^{(1,1)}&=P^{(2,1,1)}+\cdots,
\\
P^{(1,1)}P^{(1,1)}&=(1+t) P^{(2,1,1)}+\cdots
\end{aligned}\right.
\\
%\intertext{vs}
P^{(1)}P^{(1,1)}&=P^{(2,1)}+(1+t+t^2)P^{(1,1,1)},&\quad &\left\{
\begin{aligned}
P^{(1)}P^{(2,1)}&=(1+t)P^{(2,1,1)}+\cdots,
\\
P^{(1)}P^{(1,1,1)}&=P^{(2,1,1)}+\cdots
\end{aligned}\right.
\end{alignat*}
The number of honeycombs in both cases is $2$ (the value at $t=0$); however, because of the different
decomposition, one cannot have a fugacity-preserving bijection between these two pairs of honeycombs.

We now draw the sequence of double puzzles establishing the equality of the two decompositions, following
Section~\ref{sec:assoc}. We shall switch to the dual graphical notation, and rather than writing out
the value of every label explicitly, which would be very hard to read, we shall use a ``two-lane road'' notation:
to each index $\{0,1,2,3\}$ is associated a color $\{\text{red},\text{green},\text{blue},\text{yellow}\}$
and to each group of labels $a_{\alpha,\beta}$ is associated an (unordered) set of $a_{\alpha,\beta}$ roads whose left 
(resp.\ right) incoming lane is colored $\alpha$ (resp.\ $\beta$). 
The rule is that traffic should be able to go through, i.e.,
a left incoming lane should also be a left outgoing lane (with the strongest constraint at $D$ vertices
that pairs of lanes stay together). We then obtain the following pictures, with the same sequence of moves:

\tikzset{vert/.style={circle,fill=#1,inner sep=0.3mm,draw,transform shape}}
%\tikzset{myname/.style={name=v#1,label={[node font=\tiny]above:#1}}}
\tikzset{myname/.style={name=v#1}}
\tikzset{edge/.style={draw=#1,very thick}}
\begin{center}

\end{center}
\vfill\eject

% to fix MR issues
\gdef\MRshorten#1 #2MRend{#1}%
\gdef\MRfirsttwo#1#2{\if#1M%
MR\else MR#1#2\fi}
\def\MRfix#1{\MRshorten\MRfirsttwo#1 MRend}
\renewcommand\MR[1]{\relax\ifhmode\unskip\spacefactor3000 \space\fi
\MRhref{\MRfix{#1}}{{\scriptsize \MRfix{#1}}}}
\renewcommand{\MRhref}[2]{%
\href{http://www.ams.org/mathscinet-getitem?mr=#1}{#2}}
\bibliographystyle{amsalphahyper}
\bibliography{biblio}

\end{document}